\documentclass[11pt]{article}
\usepackage[margin=32mm]{geometry}
\usepackage[numbers,sort]{natbib}
\usepackage{iwona,microtype,amsmath,amsthm,amsfonts,amssymb,color,thmtools,tocloft,tikz,lscape,calc,graphicx,float,enumitem}
\usepackage[
pdftitle={Defective and Clustered Graph Colouring},
pdfauthor={David R. Wood},
colorlinks =true,
linkcolor={black},
urlcolor={blue!80!black},
filecolor={blue!80!black},
citecolor={black},
pagebackref,
anchorcolor=blue,
filecolor=blue, 
menucolor=blue]{hyperref} 

\makeatletter
\def\NAT@spacechar{~}
\makeatother

\renewcommand*{\backref}[1]{}
\renewcommand*{\backrefalt}[4]{\footnotesize\hspace*{0pt}\hfill \ifcase #1 \mbox{[not cited]} \or  \mbox{[p.\,#2]}  \else \mbox{[pp.\,#2]} \fi}

\usepackage[noabbrev,capitalise]{cleveref}
\crefname{lem}{Lemma}{Lemmas}
\crefname{thm}{Theorem}{Theorems}
\crefname{prop}{Proposition}{Propositions}
\crefname{conj}{Conjecture}{Conjectures}
\crefname{openproblem}{Open Problem}{Open Problems}

\crefformat{equation}{(#2#1#3)}
\Crefformat{equation}{Equation #2(#1)#3}

\theoremstyle{plain}
\newtheorem{thm}{Theorem}
\newtheorem{lem}[thm]{Lemma}
\newtheorem{cor}[thm]{Corollary}
\newtheorem{prop}[thm]{Proposition}
\newtheorem{conj}[thm]{Conjecture}
\newtheorem{openproblem}[thm]{Open Problem}

\renewcommand{\thefootnote}{\fnsymbol{footnote}}	
\allowdisplaybreaks
\newcommand{\VdHW}{Van den Heuvel and Wood~\cite{vdHW}}

\newcommand{\arXiv}[1]{arXiv:\,\href{http://arxiv.org/abs/#1}{#1}}
\newcommand{\msn}[1]{MR:\,\href{http://www.ams.org/mathscinet-getitem?mr=MR#1}{#1}}
\newcommand{\fulldoi}[2]{doi:\,\href{http://dx.doi.org/#1}{#2}}

\renewcommand{\thefootnote}{\fnsymbol{footnote}}

\newcommand{\bceil}[1]{\bigl\lceil{#1}\bigr\rceil}
\newcommand{\ceil}[1]{\lceil{#1}\rceil}
\newcommand{\FLOOR}[1]{\ensuremath{\protect\left\lfloor#1\right\rfloor}}
\newcommand{\CEIL}[1]{\left\lceil{#1}\right\rceil}
\newcommand{\floor}[1]{\lfloor{#1}\rfloor}
\newcommand{\half}{\ensuremath{\protect\tfrac{1}{2}}}
\renewcommand{\geq}{\geqslant}
\renewcommand{\leq}{\leqslant}

\newcommand{\bigchi}{\raisebox{1.55pt}{\scalebox{1.2}{\ensuremath\chi}}}
\newcommand{\cchi}{\bigchi_{\star}\hspace*{-0.2ex}}
\newcommand{\lchi}{\bigchi^{\ell}}
\newcommand{\lcchi}{\bigchi^{\ell}_{\star}\hspace*{-0.2ex}}
\newcommand{\dchi}{\bigchi\hspace*{-0.1ex}_{\Delta}\hspace*{-0.2ex}}
\newcommand{\ldchi}{\bigchi\hspace*{-0.1ex}_{\Delta}^{\ell}\hspace*{-0.2ex}}
\newcommand{\chigen}[1]{\bigchi\hspace*{-0.01ex}_{#1}\hspace*{-0.15ex}}

\newcommand{\GG}{\mathcal{G}}
\newcommand{\DD}{\mathcal{D}}
\newcommand{\AAA}{\mathcal{A}}
\newcommand{\OO}{\mathcal{O}}
\newcommand{\XX}{\mathcal{X}}
\newcommand{\EE}{\mathcal{E}}
\newcommand{\MM}{\mathcal{M}}
\newcommand{\LL}{\mathcal{L}}
\newcommand{\KK}{\mathcal{K}}
\newcommand{\QQ}{\mathcal{Q}}
\newcommand{\PP}{\mathcal{P}}
\newcommand{\SSS}{\mathcal{S}}
\newcommand{\CC}{\mathcal{C}}
\newcommand{\II}{\mathcal{I}}
\newcommand{\TT}{\mathcal{T}}
\newcommand{\HH}{\mathcal{H}}
\newcommand{\VV}{\mathcal{V}}

\DeclareMathOperator{\mad}{mad}
\DeclareMathOperator{\td}{td}
\DeclareMathOperator{\tw}{tw}
\DeclareMathOperator{\ctd}{\overline{td}}
\DeclareMathOperator{\dist}{dist}

\renewcommand{\baselinestretch}{1.1}
\usepackage[hang]{footmisc} 

\renewcommand{\thefootnote}{\fnsymbol{footnote}}	
\begin{document}

\vspace*{2ex}
\textbf{\Large Defective and Clustered Graph Colouring}\footnote{ \copyright\, David R. Wood. Released under the CC BY license (International 4.0). \\
This is a preliminary version of a dynamic survey to be published in the \emph{Electronic Journal of Combinatorics}, \#DS23, \url{http://www.combinatorics.org/DS23}.}

\medskip
\bigskip
{\large David~R.~Wood}\,\footnote{School of Mathematical Sciences, Monash   University, Melbourne, Australia  (\texttt{david.wood@monash.edu}). \\
Research supported by the Australian Research Council.}

\bigskip
\bigskip
\emph{Abstract.} Consider the following two ways to colour the vertices of a graph where the requirement that adjacent vertices get distinct colours is relaxed. A colouring has \emph{defect} $d$ if each monochromatic component has maximum degree at most $d$. A colouring has \emph{clustering} $c$ if each monochromatic component has at most $c$ vertices. This paper surveys research on these types of colourings, where the first priority is to minimise the number of colours, with small defect or small clustering as a secondary goal. List colouring variants are also considered. The following graph classes are studied: outerplanar graphs, planar graphs, graphs embeddable in surfaces, graphs with given maximum degree, graphs with given maximum average degree, graphs excluding a given subgraph, graphs with linear crossing number, linklessly or knotlessly embeddable graphs,  graphs with given Colin de Verdi{\`e}re parameter, graphs with given circumference, graphs excluding a fixed graph as an immersion, graphs with given thickness, graphs with given stack- or queue-number, graphs excluding $K_t$ as a minor, graphs excluding $K_{s,t}$ as a minor, and graphs excluding an arbitrary graph $H$ as a minor. Several open problems are discussed.

\author{David~R.~Wood\footnote{School of Mathematical Sciences, Monash University, Melbourne, Australia (\texttt{david.wood@monash.edu}). 
Research supported by the Australian Research Council. }}

\renewcommand{\thefootnote}{\arabic{footnote}}

\newpage
\renewcommand{\baselinestretch}{0.93}\normalsize
\tableofcontents
\renewcommand{\baselinestretch}{1.0}\normalsize
\newpage

\section{Introduction}
\label{Introduction}

Consider a graph where each vertex is assigned a colour. A \emph{monochromatic component} is a connected component of the subgraph induced by all the vertices assigned a single colour. A graph $G$ is $k$-colourable with \emph{clustering} $c$ if each vertex can be assigned one of $k$ colours such that each monochromatic component has at most $c$ vertices. A graph $G$ is $k$-colourable with \emph{defect} $d$ if each vertex of $G$ can be assigned one of $k$ colours such that each vertex is adjacent to at most $d$ neighbours of the same colour; that is, each monochromatic component  has maximum degree at most $d$. 

This paper surveys results and open problems regarding clustered and defective graph colouring, where the first priority is to minimise the number of colours, with small defect or small clustering as a secondary goal. We include various proofs that highlight the main methods employed. The emphasis is on general results for broadly defined classes of graphs, rather than more precise results for more specific classes. With this viewpoint the following definitions naturally arise. 

The \emph{clustered chromatic number} of a graph class $\GG$, denoted by $\cchi(\GG)$, is the minimum integer $k$ for which there exists an integer $c$ such that every graph in $\GG$ has a $k$-colouring with clustering $c$. If there is no such integer $k$, then  $\GG$ has \emph{unbounded} clustered chromatic number.  A graph class $\GG$ is \emph{defectively} $k$-colourable if there exists an integer $d$ such that every graph in $\GG$ is $k$-colourable with defect $d$. The \emph{defective chromatic number} of $\GG$, denoted by $\dchi(\GG)$,  is the minimum integer $k$ such that $\GG$ is defectively $k$-colourable. If there is no such integer $k$, then  $\GG$ has \emph{unbounded} defective chromatic number. Every colouring of a graph with clustering $c$ has defect $c-1$. Thus $\dchi(\GG)\leq \cchi(\GG) \leq\bigchi(\GG)$ for every class $\GG$. \cref{DefectiveSummary,ClusteredSummary} summarise the results presented in this survey; see \cref{Definitions,Choosability} for the relevant definitions. 

\subsection{History and Terminology}

There is no single origin for the notions of defective and clustered graph colouring, and the terminology used in the literature is inconsistent. 

Early papers on defective colouring include~\citep{AJ85,AJ87,Tuza89,HS86,Maddox88,Jones74}, although these did not use the term `defect'. The definition of ``$k$-colourable with defect $d$'', often written \emph{$(k,d)$-colourable}, was introduced by \citet{CCW86}. This terminology is fairly standard, although \emph{$d$-relaxed} or \emph{$d$-improper} is sometimes used. The minimum number of colours in a colouring of a graph $G$ with defect $d$ has been called the \emph{$d$-improper chromatic number}~\citep{KMS08} or \emph{$d$-chromatic number}~\citep{AJ85} of $G$. \citet{CGJ97} introduced the \emph{defective chromatic number} of a graph class (as defined above). 

Also for clustered colouring, the literature is inconsistent. One of the early papers is by \citet{KMRV97}, who defined a colouring to be \emph{$(k,c)$-fragmented} if, in our language, it is a $k$-colouring with clustering $c$. I prefer ``clustering'' since as $c$ increases, intuitively the ``fragmentation'' of the monochromatic components decreases.  \citet{EF13,EF05} called a monochromatic component a \emph{chromon} and called the clustered chromatic number of a class the \emph{metachromatic number}.

\setlength{\tabcolsep}{0.55ex}

\begin{table}[H]
\caption{Summary of Results for Defective Colouring \label{DefectiveSummary}}
\resizebox{\textwidth}{!}{%
\begin{tabular}{lr|cc|c}
\hline
\hline
graph class & $\GG$ & $\dchi(\GG)$ & $\ldchi(\GG)$ & Sect.\\
\hline
outerplanar & $\OO$ & 2 & 2& \labelcref{OuterplanarGraphs}\\
planar   & $\PP$ & 3 & 3 & \labelcref{DefectivePlanar} \\
Euler genus $\leq g$  & $\EE_g$ & 3 & 3  & \labelcref{SurfacesDefective}\\
max average degree $\leq m$
& $\AAA_m$ & $\floor{\frac{m}{2}}+1$ & $\floor{\frac{m}{2}}+1$ & \labelcref{MaximumAverageDegree}\\
linklessly embeddable &$\LL$ & $4$ & $4$ & \labelcref{Linkless} \\
knotlessly embeddable &$\KK$ & $5$ & $5$ &  \labelcref{Knotless} \\
Colin de Verdi\`ere $\leq k$ & $\VV_k$ & $k$ & $k$ &  \labelcref{ColindeVerdiere} \\
$k$-stack graphs & $\SSS_k$ & $k+1$ & $k+1$ &   \labelcref{StackQueueLayouts} \\
$k$-queue graphs & $\QQ_k$ & $k+1,\ldots, 2k+1$ & $2k+1$ & \labelcref{StackQueueLayouts} \\
no $K_t$ immersion & $\II_t$ & $2$ & $t-1$ & \labelcref{Immersions} \\
treewidth $\leq k$ & & $k+1$ & $k+1$ &   \labelcref{MinorDegree}\\
\multicolumn{2}{l|}{treewidth $\leq k$, max degree $\leq \Delta$ \hspace*{2mm}} & $2$ & $2$ & \labelcref{MinorDegree}\\
no $K_t$-minor & $\MM_{K_t}$ & $t-1$ & $t-1$  & \labelcref{KtMinorFree} \\
no $H$-minor & $\MM_H$ & $\ctd(H)-1,\ldots, 2^{\ctd(H)+1}-4$ & $\min\{s:\exists t\, H \preceq K_{s,t}\}$ & \labelcref{HMinorFree}\\
no $K_{s,t}$-minor $(s\leq t$) & $\MM_{K_{s,t}}$ & $s$ & $s$  & \labelcref{ExcludeCompleteBipartiteMinor} \\
circumference $\leq k$ & $\CC_k$ &  $\floor{\log_2 k}+1,\ldots,\floor{3\log_2 k}$ & $\ceil{\frac{k+1}{2}}$ & \labelcref{Circumference} \\
no $(k+1)$-path & $\HH_k$ & $\ceil{ \log _{2}(k+2)}-1,\ldots, \floor{3 \log_2 k}$ & $\floor{\frac{k+1}{2}}$ & \labelcref{Circumference} \\
$g$-thickness $\leq k$ & $\TT^g_k$ & $2k+1$ & $2k+1$ &  \labelcref{Thickness}\\
\hline
\hline
\end{tabular}}
\end{table}

\begin{table}[H]
\caption{Summary of Results for Clustered Colouring\label{ClusteredSummary}}
\resizebox{\textwidth}{!}{%
\begin{tabular}{lr|cc|c}
\hline
\hline
graph class & $\GG$ & $\cchi(\GG)$ & $\lcchi(\GG)$ & Sect.\\
\hline
outerplanar & $\OO$  & 3 & 3 & \labelcref{OuterplanarGraphs}\\
planar   & $\PP$  & 4 & 4 & \labelcref{ClusteredPlanar} \\
Euler genus $\leq g$  & $\EE_g$  & 4 & 4 & \labelcref{SurfacesClustered}\\
\multicolumn{2}{l|}{Euler genus $\leq g$, max degree $\leq \Delta$ \hspace*{2mm}} & $3$ & $3,4$ & \labelcref{MinorDegree}\\
max degree $\leq \Delta$  & $\DD_\Delta$ & 
           $ \FLOOR{\frac{\Delta+6}{4}} ,\ldots, \CEIL{\frac{\Delta+1}{3}}$ &  
           $ \FLOOR{\frac{\Delta+6}{4}} ,\ldots, \Delta+1$ & \labelcref{MaximumDegree}\\
max average degree $\leq m$
& $\AAA_m$ & $\floor{\frac{m}{2}}+1,\ldots,\floor{m}+1$ & $\floor{\frac{m}{2}}+1,\ldots,\floor{m}+1$ & \labelcref{MaximumAverageDegree}\\
linklessly embeddable &$\LL$ &$5$ & $5$ & \labelcref{Linkless} \\
knotlessly embeddable &$\KK$ & $6$ & $6$ & \labelcref{Knotless} \\
Colin de Verdi\`ere $\leq k$ & $\VV_k$ & open & open & \labelcref{ColindeVerdiere} \\
$k$-stack graphs & $\SSS_k$ & $k+2,\ldots, 2k+2$ & $k+2\ldots 2k+2$ & \labelcref{StackQueueLayouts} \\
$k$-queue graphs & $\QQ_k$ & $k+1,\ldots, 4k$ & $k+1,\ldots, 4k$ & \labelcref{StackQueueLayouts} \\
no $K_t$ immersion & $\II_t$ & $\geq \floor{\frac{t+4}{4}}$ & $\geq t-1$ & \labelcref{Immersions} \\
treewidth $\leq k$ & & $k+1$ & $k+1$ &  \labelcref{MinorDegree}\\
\multicolumn{2}{l|}{treewidth $\leq k$, max degree $\leq \Delta$ \hspace*{2mm}} & $2$ & $2$ & \labelcref{MinorDegree}\\
no $K_t$-minor & $\MM_{K_t}$ & $t-1,\ldots, 2t-2$ & $t-1,\ldots, \ceil{\frac{31}{2}t}$ & \labelcref{KtMinorFree} \\
\multicolumn{2}{l|}{no $K_t$-minor, max  degree $\leq \Delta$ \hspace*{2mm}} & $3$ & $\geq 3$ & \labelcref{KtMinorFree} \\
no $H$-minor & $\MM_H$ & $\ctd(H)-1,\ldots, 2^{\ctd(H)+1}-4$ & open & \labelcref{HMinorFree}\\
no $K_{s,t}$-minor $(s\leq t$) & $\MM_{K_{s,t}}$ & $s+1,\ldots, 2s+2$ & open  & \labelcref{ExcludeCompleteBipartiteMinor} \\
circumference $\leq k$ & $\CC_k$ & $\floor{\log_2 k}+1,\ldots,\floor{3\log_2 k}$ & open & \labelcref{Circumference} \\
no $(k+1)$-path & $\HH_k$ & $\ceil{ \log _{2}(k+2)}-1,\ldots, \floor{3 \log_2 k}$ & open & \labelcref{Circumference} \\
$g$-thickness $\leq k$ & $\TT^g_k$ & $2k+2,\ldots, 6k+1$ & $2k+2,\ldots,6k+1$ & \labelcref{Thickness}\\
\hline
\hline
\end{tabular}}
\end{table}

\subsection{Definitions}
\label{Definitions}

This section briefly states standard graph theoretic definitions, familiar to most readers. 

A \emph{clique} in a graph is a set of pairwise adjacent vertices.

Let $G$ be a graph. A \emph{$k$-colouring} of $G$ is a function that assigns one of $k$ colours to each vertex of $G$. 
An edge $vw$ of $G$ is \emph{bichromatic} if $v$ and $w$ are assigned distinct colours. 
A vertex $v$ of $G$ is \emph{properly} coloured if $v$ is assigned a colour distinct from every neighbour of $v$. 
A colouring of $G$ is \emph{proper} if every vertex is properly coloured. The \emph{chromatic number} of $G$, denoted $\bigchi(G)$, is the minimum integer $k$ such that there is a proper $k$-colouring of $G$. 

A \emph{graph parameter} is a real-valued function $f$ on the class of graphs such that $f(G_1)=f(G_2)$ whenever graphs $G_1$ and $G_2$ are isomorphic. Say $f$ is \emph{bounded} on a graph class $\GG$ if there exists $c$ such that $f(G)\leq c$ for every $G\in\GG$, otherwise $f$ is \emph{unbounded} on $\GG$. If $f$ is bounded on $\GG$, then let $f(\GG):=\sup\{f(G):G\in\GG\}$. Most graph parameters considered in this survey are integer-valued, in which case, if $f$ is bounded on $\GG$, then $f(\GG)=\max\{f(G):G\in\GG\}$. 

For a graph $G$, let $\mad(G)$ be the maximum average degree of a subgraph of $G$. 

A graph is \emph{$k$-degenerate} if every non-empty subgraph has a vertex of degree at most $k$. A greedy colouring algorithm shows that every $k$-degenerate graph is $(k+1)$-colourable. 

A graph $H$ is a \emph{minor} of a graph $G$ if a graph isomorphic to $H$ can be obtained from a subgraph of $G$ by contracting edges. A class of graphs $\GG$ is \emph{minor-closed} if for every graph $G\in\GG$ every minor of $G$ is in $\GG$, and some graph is not in $\GG$. A graph $G$ is \emph{$H$-minor-free} if $H$ is not a minor of $G$. Let $\MM_H$ be the class of $H$-minor-free graphs. 

To \emph{subdivide} an edge $vw$ in a graph $G$ means to delete $vw$, add a new vertex $x$, and add new edges $vx$ and $xw$. A \emph{subdivision} of $G$ is any graph obtained from $G$ by repeatedly subdividing edges. The \emph{1-subdivision} of $G$ is the graph obtained from $G$ by subdividing each edge of $G$ exactly once. A graph $H$ is a \emph{topological minor} of a graph $G$ if a graph isomorphic to a subdivision of $H$ is a subgraph of $G$. 


The \emph{Euler genus} of the orientable surface with $h$ handles is $2h$. The \emph{Euler genus} of  the non-orientable surface with $c$ cross-caps is $c$. The \emph{Euler genus} of a graph $G$ is the minimum Euler genus of a surface in which $G$ embeds (with no crossings). See~\citep{MoharThom} for background on embeddings of graphs on surfaces. 

A \emph{tree decomposition} of  a graph $G$ is given by a tree $T$ whose nodes index a collection $(T_x\subseteq V(G):x\in V(T))$ of sets of vertices in $G$ called  \emph{bags}, such that (1) for every edge $vw$ of $G$, some bag $T_x$ contains both $v$ and $w$, and (2) for every vertex $v$ of $G$, the set $\{x\in V(T):v\in T_x\}$ induces a non-empty (connected) subtree of $T$. 
The \emph{width} of a tree decomposition $T$ is $\max\{|T_x|-1:x\in V(T)\}$. 
The \emph{treewidth}  of a graph $G$, denoted by $\tw(G)$, is the minimum width of the tree decompositions of $G$. 
See~\citep{Reed97,HW17,Bodlaender-AC93,Bodlaender-TCS98,Reed03} for surveys on treewidth.

A \emph{layering} of a graph $G$ is a partition $(V_0,V_1,\dots,V_\ell)$ of $V(G)$ such that for every edge $vw\in E(G)$, if $v\in V_i$ and $w\in V_j$, then $|i-j|\leq1$. Each set $V_i$ is called a \emph{layer}. If $r$ is a vertex in a connected graph $G$ and $V_i:=\{v\in V(G):\dist_G(v,r)=i\}$ for $i\geq 0$, then $V_0,V_1,\dots$ is a layering called the \emph{BFS layering} of $G$ starting at $r$. 

The \emph{layered treewidth} of a graph $G$ is the minimum integer $k$ such that there is a tree decomposition $(T_x\subseteq V(G):x\in V(T))$ of  $G$ and a layering $(V_0,V_1,\dots,V_\ell)$ of $G$, such that $|V_i\cap T_x|\leq k$ for every $i\in[0,\ell]$ and every $x\in V(T)$. Layered treewidth was introduced by \citet{DMW17}. 


A pair $(G_1,G_2)$ is a \emph{separation} of a graph $G$ if $G_1$ and $G_2$ are induced subgraphs of $G$ such that $G=G_1\cup G_2$, and  $V(G_1)\setminus V(G_2) \neq \emptyset$ and $V(G_2)\setminus V(G_1) \neq \emptyset$. If, in addition, $|V(G_1\cap G_2)| \leq k$, then  $(G_1,G_2)$ is a \emph{$k$-separation}. A separation $(G_1,G_2)$ of $G$ is \emph{minimal} if every vertex in $V(G_1)\cap V(G_2)$ has a neighbour in both $V(G_1)\setminus V(G_2)$ and $V(G_2)\setminus V(G_1)$. 


A \emph{balanced separator} in a graph $G$ is a set $S\subseteq V(G)$ such that every component of $G-S$ has at most $\frac12|V(G)|$ vertices. 

The \emph{radius} of a connected graph $G$ is the minimum integer $r$ such that for some vertex $v$ of $G$, every vertex of $G$ is at distance at most $r$ from $v$. 

\subsection{Choosability}
\label{Choosability}

Many defective and clustered colouring results hold in the setting of list colouring. 
\citet{EH99} first introduced defective list colouring. 

A \emph{list assignment} for a graph $G$ is a function $L$ that assigns a set $L(v)$ of colours to each vertex $v\in V(G)$. Define a graph $G$ to be \emph{$L$-colourable} if there is a proper colouring of $G$ such that  each vertex $v\in V(G)$ is assigned a colour in $L(v)$.  A list assignment $L$ is a \emph{$k$-list assignment} if $|L(v)|\geq k$ for each vertex $v\in V(G)$. The \emph{choice number} of a graph $G$ is the minimum integer $k$ such that $G$ is $L$-colourable for every $k$-list-assignment $L$ of $G$.

For a list-assignment $L$ of a graph $G$ and integer $d\geq 0$, define $G$ to be \emph{$L$-colourable with defect $d$} if there is a  colouring of $G$ with defect $d$ such that  each vertex $v\in V(G)$ is assigned a colour in $L(v)$.  Define $G$ to be \emph{$k$-choosable with defect $d$} if $G$ is $L$-colourable with defect $d$ for every $k$-list assignment $L$ of $G$. Similarly, for an integer $c\geq 1$, $G$ is \emph{$L$-colourable with clustering $c$} if there is a  colouring of $G$ with clustering $c$ such that each vertex $v\in V(G)$ is assigned a colour in $L(v)$.  Define $G$ to be \emph{$k$-choosable with clustering $c$} if $G$ is $L$-colourable with clustering $c$ for every $k$-list assignment $L$ of $G$. 

The \emph{defective choice number} of a graph class $\GG$, denoted by $\ldchi(\GG)$, is the minimum integer $k$ for which there exists an integer $d\geq 0$, such that every graph $G\in \GG$ is $k$-choosable with defect $d$. The \emph{clustered choice number} of a graph class $\GG$, denoted by $\lcchi(\GG)$, is the minimum integer $k$ for which there exists an integer $c\geq 1$, such that every graph $G\in \GG$ is $k$-choosable with clustering $c$.

\subsection{Standard Examples}

The following construction, or variants of it, have been used by several authors~\citep{HS06,EKKOS15,OOW16,NSSW,vdHW} to provide lower bounds on the defective chromatic number. 
As illustrated in \cref{StandardExample}, let $S(h,d)$ be defined recursively as follows. 
Let $S(0,d)$ be the graph with one vertex and no edges. 
For $h\geq 1$, let $S(h,d)$ be the graph obtained from $d+1$ disjoint copies of $S(h-1,d)$ by adding one dominant vertex. Note that $S(1,d) =  K_{1,d+1}$, the star with $d+1$ leaves. 


\begin{figure}[h]
\centering
\includegraphics{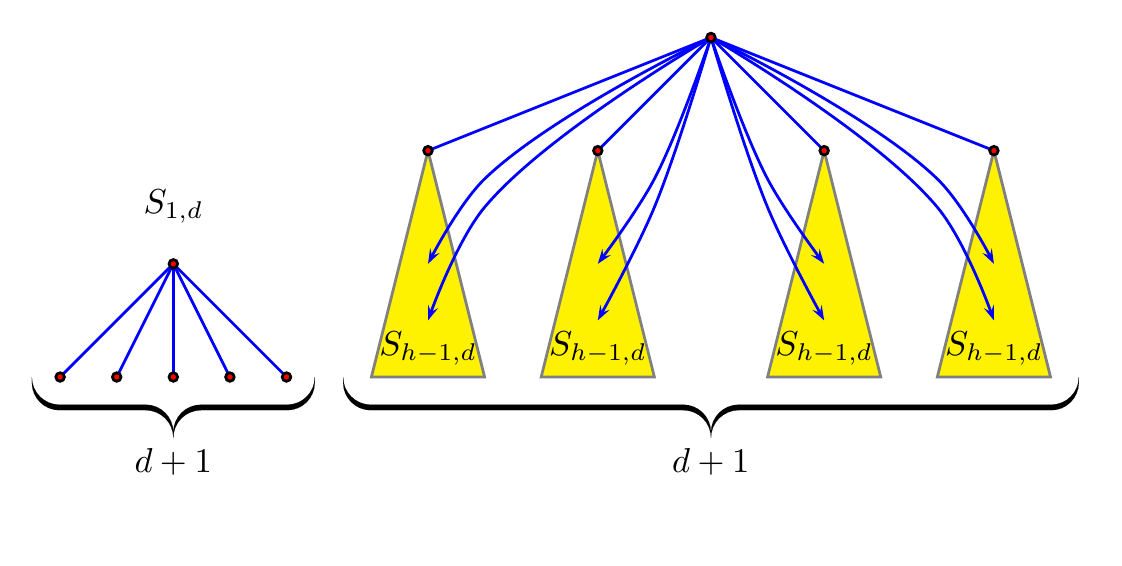}

\vspace*{-10mm}
\caption{The standard example $S(h,d)$.
\label{StandardExample}}
\end{figure}

\begin{lem}
\label{StandardDefect}
For integers $h\geq 1 $ and $d\geq 0$, the graph $S(h,d)$ has no $h$-colouring with defect $d$.
\end{lem}

\begin{proof}
We proceed by induction on $h$. In the base case, $S(1,d)=K_{1,d+1}$, which obviously has no 1-colouring with defect $d$. Now assume that $h\geq 2$ and the claim holds for $h-1$. Suppose on the contrary that $S(h,d)$ has an $h$-colouring with defect $d$. Let $v$ be the dominant vertex in $S(h,d)$, and say $v$ is blue. 
Then $S(h,d)-v$ has $d+1$ components $C_1,\dots,C_{d+1}$. Since $v$ is dominant and has monochromatic degree at most $d$, at most $d$ of $C_1,\dots,C_{d+1}$ contain a blue vertex. Thus some $C_i$ is $(h-1)$-coloured with defect $d$. This is a contradiction since $C_i$ is isomorphic to $S(h-1,d)$. Thus $S(h,d)$ has no $h$-colouring with defect $d$.
\end{proof}


Similarly, as illustrated in \cref{ClusteredStandardExample}, let $\overline{S}(h,c)$ be defined recursively as follows. Let $\overline{S}(1,c)$ be the path on $c+1$ vertices. For $h\geq 2$, let $\overline{S}(h,c)$ be the graph obtained from $c$ disjoint copies of $\overline{S}(h-1,c)$ by adding one dominant vertex.

\begin{figure}[h]
\centering

\includegraphics{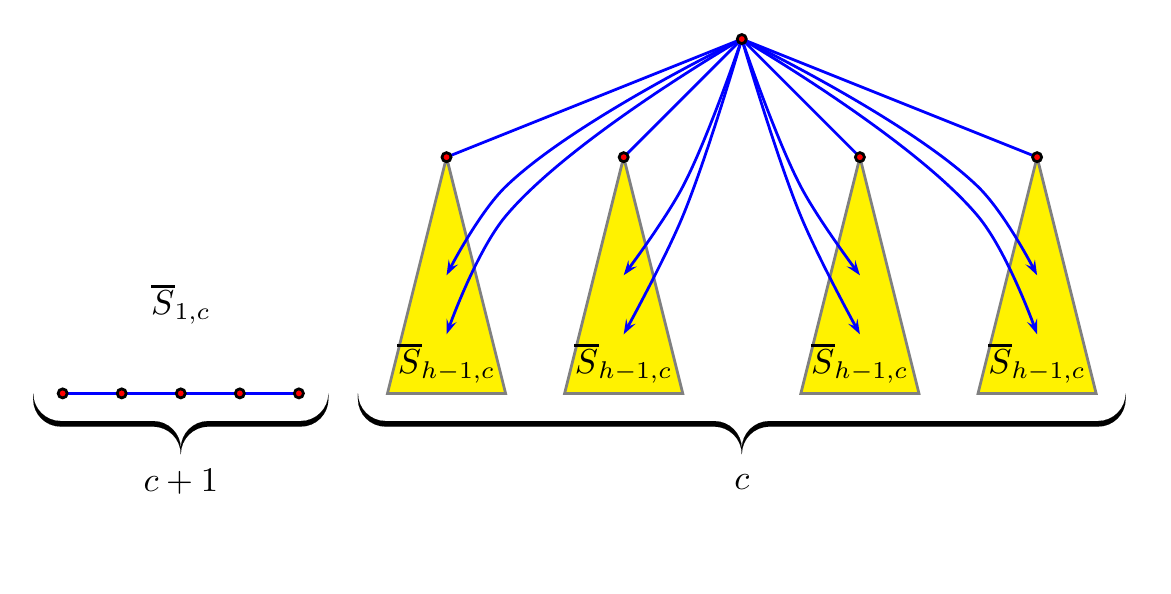}

\vspace*{-10mm}
\caption{The standard example $\overline{S}(h,c)$.
\label{ClusteredStandardExample}}
\end{figure}

\begin{lem}
\label{StandardClustering}
For integers $h,d\geq 1$, the graph $\overline{S}(h,c)$ has no $h$-colouring with clustering $c$.
\end{lem}

\begin{proof}
We proceed by induction on $h$. In the base case, $\overline{S}(1,c)$ is the path on $c+1$ vertices, which obviously has no 1-colouring with clustering $c$. Now assume that $h\geq 2$ and the claim holds for $h-1$. Suppose on the contrary that $\overline{S}(h,c)$ has an $h$-colouring with defect $c$. Let $v$ be the dominant vertex in $\overline{S}(h,c)$, and say $v$ is blue. Then $\overline{S}(h,c)-v$ has $c$ components $C_1,\dots,C_{c}$. Since $v$ is dominant and the monochromatic component containing $v$ has at most $c$ vertices, at most $c-1$ of $C_1,\dots,C_{c}$ contain a blue vertex. Thus some $C_i$ is $(h-1)$-coloured with clustering $c$. This is a contradiction since $C_i$ is isomorphic to $\overline{S}(h-1,c)$. Thus $\overline{S}(h,d)$ has no $h$-colouring with clustering $c$.
\end{proof}

As illustrated in \cref{PlanarStandardExamples}, $S(2,d)$ is planar, $\overline{S}(2,c)$ is outerplanar, and $\overline{S}(3,c)$ is planar. 

\begin{figure}[h]
\centering

\includegraphics{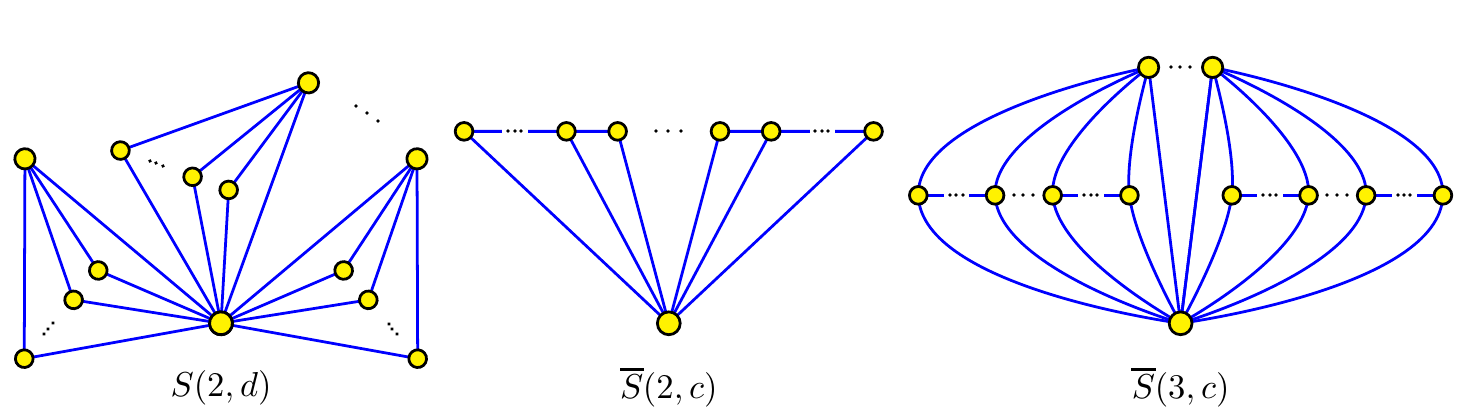}

\caption{Planar standard examples.
\label{PlanarStandardExamples}}
\end{figure}

\subsection{Two Fundamental Observations}

The following elementary, but fundamental, result characterises those graph classes with bounded  defective or clustered chromatic number. 

\begin{prop}
\label{Bounded}
The following are equivalent for a graph class $\GG$:
\begin{enumerate}[label=(\arabic*),itemsep=0ex,topsep=0ex]
\item $\GG$  has bounded defective chromatic number,
\item $\GG$  has bounded clustered chromatic number,
\item $\GG$  has bounded chromatic number.
\end{enumerate}
Moreover, $\dchi(\GG)\leq\cchi(\GG)\leq \bigchi(\GG)$. 
\end{prop}

\begin{proof}
We first show that (3) implies (1) and (2). 
Suppose that $\GG$  has bounded chromatic number.
That is, for some integer $k$, every graph $G$ in $\GG$ is properly $k$-colourable.
Thus $G$ is $k$-colourable with defect 0 and with clustering 1. 
Hence $\dchi(\GG)\leq \cchi(\GG) \leq k$. 

We now show that (1) implies (3). Suppose that $\GG$  has bounded defective chromatic number. That is, for some integers $k$ and $d$, every graph $G$ in $\GG$ is $k$-colourable such that each monochromatic subgraph has maximum degree $d$. Every graph with maximum degree $d$ is properly $(d+1)$-colourable by a greedy algorithm. Apply this result to each monochromatic subgraph of $G$. Hence $G$ is properly $k(d+1)$-colourable, and $\GG$  has bounded chromatic number.

We finally show that (2) implies (1). Suppose that $\GG$  has bounded clustered chromatic number. That is, for some integers $k$ and $c$, every graph $G$ in $\GG$ is $k$-colourable such that each monochromatic subgraph has at most  $c$ vertices, and therefore has maximum degree at most $c-1$. Hence $\dchi(\GG)\leq k$, and $\GG$  has bounded defective chromatic number.
\end{proof}

We have the following analogous result for defective and clustered choosability. 

\begin{prop}
\label{BoundedChoice}
The following are equivalent for a graph class $\GG$:
\begin{enumerate}[label=(\arabic*),itemsep=0ex,topsep=0ex]
\item $\GG$  has bounded defective choice number,
\item $\GG$  has bounded clustered choice number,
\item $\GG$  has bounded choice number.
\item $\GG$  has bounded maximum average degree. 
\end{enumerate}
\end{prop}

\begin{proof}
A greedy algorithm shows that (4) implies (3). 
\citet{Alon00} proved that (3) implies (4). 
More generally, \citet{Kang13} proved that (1) implies (4). 
It is immediate that (3) implies (1) and (2). 
As in the proof of \cref{Bounded}, if a graph $G$ is $k$-choosable with clustering $c$, then $G$ is $k$-choosable with defect $c-1$. Thus (2) implies (1). 
\end{proof}

\subsection{Related Topics}

We briefly mention here some related topics not covered in this survey:
\begin{itemize}[itemsep=0ex,topsep=0ex]
\item defective colourings of random graphs~\citep{KM10,KM15,KMS08},
\item defective versions of Ohba's list colouring conjecture~\citep{YWX15,WQY17},
\item defective Nordhaus-Gaddum type results~\citep{AAS11a,AAS96}, 
\item acyclic defective colourings~\citep{BJS11,FS16,AEKMP10,BSV99,FS14},
\item uniquely defectively colourable graphs~\citep{FH94}
\item weighted defective colouring~\citep{ABGHMM12,HMG16,ABHCG15}.
\item defective colourings of triangle-free graphs~\citep{AAS11,SAA97,SAA97b},
\item defective colouring of directed graphs~\citep{MRS14a}, 
\item defective edge colouring~\citep{HSS98,Hilton96,Woodall90}, 
\item defective circular and fractional colouring~\citep{MOS11,Klostermeyer02,FS15,GX16}, 
\item equitable defective colouring~\citep{WVY12},
\item defective co-colouring~\citep{AE15},
\end{itemize}

Also note that defective colouring is computationally hard. In particular, it is NP-complete to decide if a given graph $G$ is 3-colourable with defect 1, even when $G$ has maximum degree 6 or is planar with maximum degree 7~\citep{ABLDKKMRRS17}. See~\citep{GHKMN18,KMZ11,BLM17,CHS09,BS09,MO15,CG17,BFKW06} for more computational aspects of defective and clustered clustered colouring. 

\section{Greedy Approaches}

\subsection{Light Edges}

Light edges provide a generic method for proving results about defective colourings. An edge $e$ in a graph is \emph{$\ell$-light} if both endpoints of $e$ have degree at most $\ell$.  There is a large literature on light edges in graphs; see~\citep{BS-MN94,JV-DM01,JV02,JenMad96,BSW-DM04,Borodin-JRAM89} for example. Several results on defective colouring use, sometimes implicitly, the following lemma~\citep{LSWZ01,OOW16,EKKOS15,Skrekovski00,HS06}.

\begin{lem}
\label{light}
For integers $\ell\geq k\geq 1$, if every subgraph $H$ of a graph $G$ has a vertex of degree at most $k$ or an $\ell$-light edge, then $G$ is $(k+1)$-choosable with defect $\ell-k$. 
\end{lem}

\begin{proof}
Let $L$ be a $(k+1)$-list assignment for $G$. We prove by induction on $|V(H)|+|E(H)|$ that every subgraph $H$ of $G$ is $L$-colourable with defect $\ell-k$. The base case with $|V(H)|+|E(H)|=0$ is trivial. Consider a subgraph $H$ of $G$. If $H$ has a vertex $v$ of degree at most $k$, then by induction $H-v$ is $L$-colourable with defect $\ell-k$, and there is a colour in $L(v)$ used by no neighbour of $v$ which can be assigned to $v$. Now assume that $H$ has minimum degree at least $k+1$. By assumption, $H$ contains an $\ell$-light edge $xy$. By induction, $H-xy$ has an $L$-colouring $c$ with defect $\ell-k$.  If $c(x)\neq c(y)$, then $c$ is also an $L$-colouring of $H$ with defect $\ell-k$. Now assume that $c(x)=c(y)$. We may further assume that $c$ is not an $L$-colouring of $H$ with defect $\ell-k$. Without loss of generality, $x$ has exactly $\ell-k+1$ neighbours (including $y$) coloured by $c(x)$. Since $\deg_H(x)\leq \ell$, there are at most $k-1$ neighbours not coloured by $c(x)$. Since $L(v)$ contains $k$ colours different from $c(x)$, there is a colour used by no neighbour of $x$ which can be assigned to $x$.
\end{proof}

\subsection{Islands}

\citet{EO16} introduced the following definition and lemma. A \emph{$k$-island} in a graph $G$ is a non-empty set $S\subseteq V(G)$ such that every vertex in $S$ has at most $k$ neighbours in $V(G)\setminus S$. 

\begin{lem}[\citep{EO16}] 
\label{IslandColouring}
If every non-empty subgraph of a graph $G$ has a $k$-island of size at most $c$, then $G$ is $(k+1)$-choosable with clustering $c$. 
\end{lem}

\begin{proof}
We proceed by induction on $|V(G)|$. The base case is trivial. Let $L$ be a $(k+1)$-list assignment for $G$. By assumption, $G$ has a $k$-island $S$. By induction, $G-S$ is $L$-colourable with clustering $c$. Assign each vertex $v\in S$ a colour in $L(v)$ not assigned to any neighbour of $v$ in $V(G)\setminus S$. Each  monochromatic component is contained in $S$ or is a monochromatic component of $G-S$. Since $|S|\leq c$, $G$ is $L$-coloured with clustering $c$. Thus $G$ is $k$-choosable with clustering $c$. 
\end{proof}

Note that islands generalise the notion of degeneracy, since a graph is $k$-degenerate if and only if every non-empty subgraph has a $k$-island of size 1. Thus \cref{IslandColouring} with $c=1$ is equivalent to the well-known statement that every $k$-degenerate graph is properly $(k+1)$-choosable. 

\section{Graphs on Surfaces}

\subsection{Outerplanar Graphs}
\label{OuterplanarGraphs}

Let $\OO$ be the class of outerplanar graphs. 
Every outerplanar graph is 2-degenerate, and thus is properly 3-colourable and 3-choosable. 
Since $\overline{S}(2,c)$ is outerplanar, by \cref{StandardClustering}, 
$$\cchi(\OO)=\lcchi(\OO)=3.$$
\citet{CCW86} proved the following result for defective colourings of outerplanar graphs. 

\begin{thm}[\citep{CCW86}] 
\label{Outerplanar}
Every outerplanar graph $G$  is $2$-colourable such that each monochromatic component is a path (and thus with defect 2).
\end{thm}

\begin{proof}
We may assume $G$ is connected. Let $V_0,V_1,\dots$ be the BFS layering starting from some vertex $r$. Thus $V_0=\{r\}$. If $G[V_i]$ has maximum degree at least 3, for some $i\geq 1$, then contracting $V_0\cup\dots\cup V_{i-1}$ into a single vertex gives a $K_{2,3}$-minor, which is not outerplanar. Thus $G[V_i]$ has maximum degree at most 2. If $G[V_i]$ contains a cycle, then contracting $V_0\cup\dots\cup V_{i-1}$ into a single vertex gives a $K_4$-minor, which is not outerplanar. Thus each component of $G[V_i]$ is a path. Colour each vertex in $V_i$ by $i\bmod{2}$. 
Then $G$ is 2-coloured  such that each monochromatic component is a path.
\end{proof}

\citet{EH99} generalised \cref{Outerplanar} by showing that every outerplanar graph  is $2$-choosable with defect 2. By \cref{StandardDefect},  the outerplanar graph $K_{1,d+1}$ is not 1-colourable with defect $d$. Thus
$$\dchi(\OO)=\ldchi(\OO)=2.$$
Moreover, the defect bound in the above results is best possible since the graph shown in \cref{OuterplanarExample} is not 2-colourable with defect 1. 

\begin{figure}[h]
\centering
\includegraphics{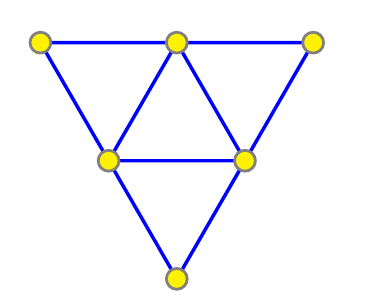}
\caption{An outerplanar graph that is not 2-colourable with defect 1. 
\label{OuterplanarExample}}
\end{figure}

See~\citep{Woodall02,Woodall90} for more results on defective colouring of outerplanar graphs. 

\subsection{Planar Graphs}
\label{DefectivePlanar}

Let $\PP$ be the class of planar graphs. We now discuss defective colourings of $\PP$. First note that many results about light edges in a planar graphs are known~\citep{JV13,JV02}. \citet{Borodin-JRAM89} proved that every planar graph with minimum degree at least 3 contains an edge $vw$ with $\deg(v)+\deg(w)\leq 13$ (which is best possible for the graph obtained from the icosahedron by stellating each face). This edge is 10-light. By \cref{light}, every planar graph is  $3$-choosable with defect $8$. \citet{CCW86} improved the defect bound here to 2, and \citet{Poh90} proved an analogous result in which each monochromatic component is a path (see~\citep{vdHW} for an alternative proof). 

\begin{thm}[\citep{Poh90}]
\label{Planar32}
Every planar graph is $3$-colourable such that each monochromatic component is a path (and thus with defect 2).
\end{thm}

\begin{proof}
We proceed by induction on $|V(G)|$ with the hypothesis that every planar graph $G$ is 3-colourable such that for each edge $v_1v_2$ of $G$, there is such a 3-colouring of $G$ such that each monochromatic component is a path, and $v_1$ and $v_2$ are properly coloured. (Recall that this means that every neighbour of $v_1$ is assigned a distinct colour from $v_1$, and every neighbour of $v_2$ is assigned a distinct colour from $v_2$.) In the base case, if $|V(G)|\leq 4$, then assign $v_1$ and $v_2$ distinct colours, and assign the (at most two) other vertices a third colour. Each monochromatic component is a path. Now assume that $|V(G)|\geq 5$. By adding edges, we may assume that $G$ is a planar triangulation. Say the faces containing $v_1v_2$ are $v_1av_2$ and $v_1bv_2$. Let $G'$ be obtained from $G$ by deleting the edge $v_1v_2$, and introducing a new vertex $x$ adjacent to $v_1,v_2,a,b$. Then $G'$ is a planar triangulation. Let $C$ be a shortest cycle in $(G'-v_1)-v_2$ such that $axb$ is a subpath of $C$. Since $G'$ is 3-connected and $|V(G)|\geq 5$, such a cycle exists. Since $C$ is shortest, $C$ is an induced cycle. Let $G_1'$ and $G_2'$ be the subgraphs of $G'$ `inside' and `outside' of $C$ including $C$. That is, $G'=G_1'\cup G_2'$ and $V(G_1')\cap V(G_2')=V(C)$ and $E(G_1')\cap E(G_2')=E(C)$. Without loss of generality, $v_i\in V(G_i')$ for $i\in\{1,2\}$. 
Note that $$|V(G)|=|V(G')|-1=|V(G'_1)|+|V(G'_2)|-|V(C)|-1. $$
Let $G_i''$ be obtained from $G_i'$ by contracting $C$ into vertex $x_i$.  Then $v_ix_i\in E(G_i'')$ and 
$|V(G''_i)| = |V(G'_i)|-|V(C)|+1$. Thus 
\begin{align*}
|V(G)| & = ( |V(G''_1)|+|V(C)|-1) + ( |V(G''_2)|+|V(C)|-1) -|V(C)|-1\\
& =  |V(G''_1)|+  |V(G''_2)|+|V(C)|-3.
\end{align*}
Since $|V(G''_2)|\geq 2$ and $|V(C)|\geq 3$, we have
$|V(G)| \geq |V(G''_1)|+  2$, implying $|V(G_1'')|<|V(G)|$.
Similarly, $|V(G_2'')|<|V(G)|$. By induction, each $G_i''$ is 3-colourable such that each monochromatic component is a path, and $v_i$ and $x_i$ are properly coloured. Permute the colours so that $x_1$ and $x_2$ get the same colour, and $v_1$ and $v_2$ get distinct colours. Colour each vertex in $V(C)\setminus\{x\}$ by the colour assigned to $x_1$ and $x_2$. Note that $V(C)\setminus\{x\}$ induces a path in $G$ (since $x$ is not a  vertex of $G$). Moreover, $V(C)\setminus\{x\}$ is a monochromatic component, since each $x_i$ is properly coloured in $G_i''$. Every other monochromatic component of $G$ is a monochromatic component of $G_1''$ or $G_2''$, and is therefore a path in $G$. 
\end{proof}

\citet{EH99} strengthened \cref{Planar32} by showing that every planar graph is $3$-choosable with defect 2. (See \cref{OOW} or an alternative proof of a more general result with a weaker defect bound.) Since $S(2,d)$ is planar, by \cref{StandardDefect}, 
$$\dchi(\PP)=\ldchi(\PP)=3.$$

See~\citep{CR15,LSWZ01,Skrekovski00,Skrekovski99,Xu08,Bor13,BI09,GZ07,Woodall90} for more on defective colourings of planar graphs. See~\citep{Zhang04,WX13,Woodall01} for more on defective choosability of planar graphs. 

\label{ClusteredPlanar}

Now consider clustered colourings of planar graphs. The 4-colour theorem~\citep{AH89,RSST97} says that every planar graph is properly $4$-colourable.  \citet{CCW86} proved the weaker result that every planar graph is $4$-colourable with defect 1 and thus with clustering 2 (with a computer-free elementary proof). 
%
%
%
\citet{CK10} strengthened this result by proving that every planar graph is $4$-choosable with defect 1 and thus with clustering 2. Since $\overline{S}(3,c)$ is planar, by \cref{StandardClustering}, 
$$\cchi(\PP)=\lcchi(\PP)=4.$$
\citet{Thomassen-JCTB94} proved that every planar graph is properly $5$-choosable. \citet{Voigt-DM93} and later \citet{Mirz96} constructed planar graphs that are not 4-choosable. Thus $\lchi(\PP)=5$.

\subsection{Hex Lemma}

Here we consider colourings of planar graphs with bounded degree. The following result is a dual version of the Hex Lemma, which says that the game of Hex cannot end in a draw. The  proof is based on the proof of the Hex Lemma by \citet{Gale79}. See~\citep[Section~6.1]{MatNes98} for another proof. 

\begin{thm}
\label{HexTheorem}
For every integer $k\geq 2$ there is planar graph $G$ with maximum degree 6 such that every 2-colouring of $G$ has a monochromatic path of length $k$. 
\end{thm}

\begin{proof}
A suitable subgraph of the triangular grid forms an embedded plane graph with maximum degree 6, such that every internal face is a triangle, the outerface is a cycle $(a,\dots,b,\dots,c,\dots,d,\dots)$, the distance between $\{a,\dots,b\}$ and $\{c,\dots,d\}$ is at least $k$, and the distance between $\{b,\dots,c\}$ and $\{d,\dots,a\}$ is at least $k$. By \cref{Gale} below, every 2-colouring of $G$ contains a monochromatic path between $\{a,\dots,b\}$ and $\{c,\dots,d\}$ or between $\{b,\dots,c\}$ and $\{d,\dots,a\}$, which has length at least $k$. 
\end{proof}

\begin{figure}[b!]
\centering 
\includegraphics[width=110mm]{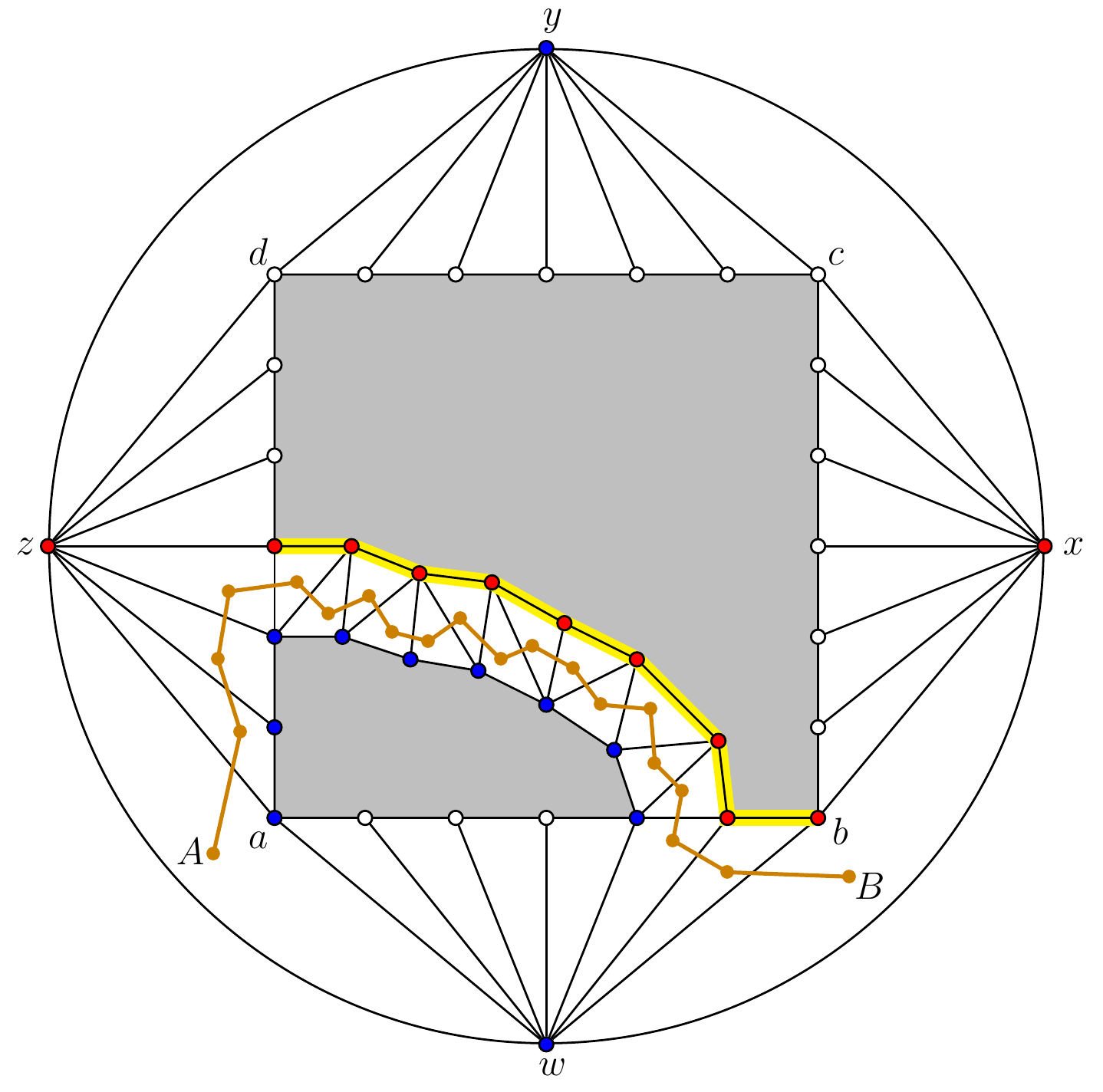}
\caption{\label{HexLemmaFigure}Proof of the Hex Lemma}
\end{figure}

\begin{lem}
\label{Gale}
Let $G$ be an embedded plane graph with outerface $(a,\dots,b,\dots,c,\dots,d,\dots)$, such that every internal face is a triangle, and $a,b,c,d$ are distinct. Then for every 2-colouring of $G$ there is a monochromatic path between $\{a,\dots,b\}$ and $\{c,\dots,d\}$  or between $\{b,\dots,c\}$ and $\{d,\dots,a\}$. 
\end{lem}

\begin{proof}
Say the colours are blue and red. 
As shown in \cref{HexLemmaFigure}, let $G'$ be obtained from $G$ by adding four new vertices $w,x,y,z$, where 
$N(w)=\{a,\dots,b\}\cup\{z,x\}$ and 
$N(x)=\{b,\dots,c\}\cup\{w,y\}$ and 
$N(y)=\{c,\dots,d\}\cup\{x,z\}$ and 
$N(z)=\{d,\dots,a\}\cup\{y,w\}$. 
Colour $w$ and $y$ blue. Colour $x$ and $z$ red. 
Note that $G'$ embeds in the plane, such that every internal face of $G'$ is a triangle, and the outerface of $G'$ is the 4-cycle $(w,x,y,z)$. The four internal faces of $G'$ that share an edge with the outerface are $(a,w,z)$, $(b,x,w)$, $(c,x,y)$ and $(d,y,z)$. Call these faces \emph{special}. Let $H$ be the graph with one vertex for each internal face of $G'$, where two vertices of $H$ are adjacent if the corresponding faces of $H$ share an edge whose endpoints are coloured differently.  $H$ is a subgraph of the dual of $G'$ and is therefore planar. Let $A,B,C,D$ be the vertices of $H$ respectively corresponding to the special faces $(a,w,z)$, $(b,x,w)$, $(c,x,y)$, $(d,y,z)$. Since $w$ and $z$ are coloured differently, $a$ has the same colour as exactly one of $w$ and $z$, implying $A$ has degree 1 in $H$. Similarly, $B$, $C$ and $D$ each have degree 1 in $H$. If some face of $G'$ is monochromatic, then the corresponding vertex of $H$ has degree 0. Every non-monochromatic non-special face $F$ has two vertices of one colour and one vertex of the other colour, and $F$ does not share an edge with the outerface. Thus the vertex of $H$ corresponding to $F$ has degree 2 in $H$. In summary, every vertex of $H$ has degree 0 or 2, except for $A,B,C,D$, which have degree 1. Thus each component of $H$ is either an isolated vertex, a cycle, or a path joining two of $A,B,C,D$. The two paths joining $A,B,C,D$ are disjoint and do not cross. Thus, without loss of generality, $H$ contains a path $P$ with endpoints $A$ and $B$. The red vertices on the faces corresponding to vertices in $P$ form a walk from $\{d,\dots,a\}$ to $\{b,\dots,c\}$, which contains the desired red path.
\end{proof}

See~\citep{BDN17,MP08,Matdinov13,Karasev13} for multi-dimensional generalisations of the Hex Lemma.

\subsection{Defective Colouring of Graphs on Surfaces}
\label{SurfacesDefective}

For every integer $g\geq 0$, let $\EE_g$ be the class of graphs with Euler genus at most $g$. This section considers defective colourings of graphs in $\EE_g$. \citet{CCW86} proved that every graph in $\EE_g$ is defectively 4-colourable. \citet{Archdeacon87} proved the conjecture of \citet{CCW86} that every graph in $\EE_g$ is defectively 3-colourable. Since $S(2,d)$ is planar, $\dchi(\EE_g)=3$ by \cref{StandardDefect}. The following proof of \citet{CGJ97}\footnote{\citet{CGJ97} actually claim an upper bound on the defect of $\max\{12, \sqrt{6g} + 6 \}$. We could not replicate this calculation.} provides a defect bound that is within a constant factor of optimal (since $K_n$ has Euler genus $\Theta(n^2)$).

\begin{thm}[\citep{CGJ97}]
\label{GenusDefective3Colouring}
Every graph $G$ with Euler genus $g$ is 3-colourable with defect $d:=\max\{12,\ceil{ \sqrt{6g}} + 7 \}$. 
\end{thm}

\begin{proof}
We proceed by induction on $|V(G)|+|E(G)|$. 
If some vertex $v$ has degree at most 2, then by induction, $G-v$ is 3-colourable with defect $d$. 
Assign $v$ a colour different from the colours assigned to the neighbours of $v$. 
Then $G$ is 3-coloured with defect $d$. 
Now assume that $G$ has minimum degree at least 3. 
Let $A$ be the set of vertices with degree at most $d$. 
If $v,w\in A$ for some edge $vw$, then by induction, $G-vw$ is 3-colourable with defect $d$, which is a 3-colouring of $G$ with defect $d$. 
Now assume that $A$ is a stable set. 
Let $B$ be the set of vertices with degree at least $d+1$. 
Since $d\geq 12$ and by \cref{GenusDefective3ColouringLemma} below,  
$|B| \leq \frac{\max\{12(g-2),0\}}{d-11}$.
Colour the vertices in $A$ blue. 
Colour $\ceil{\frac{|B|}{2}}$ of the vertices in $B$ red.
Colour the other vertices in $B$ green.

If $g\leq 2$ then $B=\emptyset$ and the defect is $0 \leq d$, as desired. 
Otherwise $|B| \leq  \frac{12(g-2)}{d-11} $, and the defect is at most  $\ceil{\frac{|B|}{2}} - 1$. 
If $g\in\{3,4\}$ then $d=12$ and $|B|\leq 24$ and the defect is at most $11$, as desired. 
Otherwise, $g\geq 5$, and the defect is at most 
$\frac{|B|-1}{2} \leq \frac{6(g-2)}{d-11} - \frac{1}{2}$. 
We now show this bound is at most $d$. 
Since $g\geq 5$, we have 
$12g + 36 \leq 12 g + 7 \sqrt{6g} $, implying
$$12(g-2) \leq 12g + (15 -8) \sqrt{6g} - 60 = (2 \sqrt{6g}+ 15)(\sqrt{6g} - 4).$$
Since $d\geq \sqrt{6g}+7$, we have 
$12(g-2) \leq (2d+1)(d-11)$. That is,  
$ 6(g-2) \leq (d+\frac12)(d-11)$. 
Since $d\geq 12$, we have 
$\frac{6(g-2)}{d-11} - \frac{1}{2} \leq d$, as claimed. 
\end{proof}

\begin{lem}[\citep{CGJ97}]
\label{GenusDefective3ColouringLemma}
Let $G$ be a graph with Euler genus $g$ and minimum degree at least 3.
Fix an integer  $d$. Let $A$ be the set of vertices with degree at most $d$.
Assume that $A$ is a stable set. Let $B$ be the set of vertices with degree at least $d+1$.
Then $(d-11) |B| \leq \max\{12(g-2),0\}$.
\end{lem}

\begin{proof}
If $B=\emptyset$ then the result is vacuous. Now assume that $B\neq\emptyset$. 
For each non-triangular face $f$, add an edge between two non-consecutive vertices in $B$ (which must exist since $A$ is a stable set). 
We obtain a multigraph triangulation $G'$, in which $A$ is a stable set.
Let $n_i$ be the number of vertices with degree $i$ in $G'$.
Let $\alpha$ be the number of edges in $G'$ incident with vertices in $A$.
Let $\beta$ be the number of edges in $G'$ with both endpoints in $B$. 
By Euler's formula, 
\begin{align*}
 6(g-2) 
  = \sum_{i\geq 3} (i-6)n_i 
 & = \sum_{3\leq i\leq d} (i-6)n_i + \sum_{i\geq d+1} (i-6)n_i \\
 & = \sum_{3\leq i\leq d} (i-6)n_i + \sum_{i\geq d+1} (\tfrac{i}{2}-6)n_i + \frac{1}{2} \sum_{i\geq d+1} i n_i \\
 & = \sum_{3\leq i\leq d} (i-6)n_i + \sum_{i\geq d+1} (\tfrac{i}{2}-6)n_i + \frac{1}{2} ( \alpha + 2 \beta ) .
 \end{align*}
 Since $G'$ is a triangulation, each face of $G'$ has at most two edges incident with $A$, and at least one edge with endpoints in $B$. 
 It follows that $\alpha\leq 2\beta$ and 
 \begin{align*}
 6(g-2) 
 & \geq \sum_{3\leq i\leq d} (i-6)n_i + \sum_{i\geq d+1} (\tfrac{i}{2}-6)n_i + \alpha\\
 & = \sum_{3\leq i\leq d} (2i-6)n_i + \sum_{i\geq d+1} (\tfrac{i}{2}-6)n_i \\
 & \geq 0 + \left( \tfrac{d+1}{2}-6 \right) |B|.
 \end{align*}
The result follows. 
\end{proof}


\citet{Woodall11} improved \cref{GenusDefective3Colouring} to show that every graph with Euler genus $g$ is 3-choosable with defect  $\max\{9, 2+ \sqrt{4g + 6}\}$. Thus $$\dchi(\EE_g)=\ldchi(\EE_g)=3.$$ 
See \cref{typegk,ColourGenusThickness} for generalisations of this result, and see~\citep{Rackham11,Woodall90,HZ16,CE16,CCJS17,Zhang16} for further results on defective colourings of graphs embeddable on surfaces. One direction of interest is the following  definition, which allows for results that bridge the gap between proper and defective colourings~\citep{DWX17,CWLX16,MO15,CE16,BIMOR10,BI11,CCJS17}:  a graph $G$ is \emph{$(d_1,\dots,d_k)$-colourable} if there is a partition $V_1,\dots,V_k$ of $V(G)$ such that each induced subgraph $G[V_i]$ has maximum degree at most $d_i$. For example, \citet{CE16} proved the following analogue of the 4-colour theorem: every graph with Euler genus $g>0$ is $(0, 0, 0, 9g - 4)$-colourable.


Note that the light edge approach also proves 3-choosability, but with a weaker defect bound. In particular, results of \citet{Ivanco92} and \citet{JT06} together imply that every graph in $\EE_g$ with minimum degree at least 3 has an edge $vw$ with $\deg(v)+\deg(w) \leq  \max\{2g+7,19\}$. (Better results are known for specific surfaces with $g\leq 5$, and all the bounds are tight.)\ 
Thus every graph in $\EE_g$ with minimum degree at least 3 has a $\max\{2g+4,16\}$-light edge. (See \cref{LightEdgeGen} for a more general result with a slightly weaker bound.)\ \cref{light} then implies that every graph with Euler genus $g$ is $3$-choosable with defect $\max\{2g+2,14\}$. See~\citep{CZW08,Miao03,Zhang16,Zhang15,Zhang13,Zhang12} for more on defective choosability of graphs embedded on surfaces.
	
\subsection{Clustered Colouring of Graphs on Surfaces}
\label{SurfacesClustered}

This section considers clustered colouring of graphs embeddable on surfaces. \citet{EO16} proved that every graph of bounded Euler genus has a 4-island of bounded size, and is thus 5-colourable with bounded clustering by \cref{IslandColouring}. \citet{KT12} also proved that every graph of bounded Euler genus is 5-colourable with bounded clustering. \citet{DN17} improved 5 to 4 via the following remarkably simple argument.

\begin{lem}[\citep{DN17}]
\label{SeparatorIsland}
Let $G$ be a graph, such that for some constants $\alpha,c>0$ and $\beta\in(0,1)$, 
$$|E(G)| < (k+1-\alpha) |V(G)|,$$
and every subgraph of $G$ with $n$ vertices has a  balanced separator of size at most $cn^{1-\beta}$. 
Then $G$ has a $k$-island of size at most 
$$\CEIL{2 \left( \frac{ c (k+1) }{ \alpha(2^\beta-1)}\right)^{1/\beta} }.$$
\end{lem}

\begin{proof}
Let $\epsilon := \frac{\alpha}{k+1}$. By \cref{LT2} below, there exists $X\subseteq V(G)$ of size at most $\epsilon |V(G)|$ such that if $K_1,\dots, K_p$ are the components of $G-X$, then each $K_i$ has at most 
$\ceil{2 (  \frac{c}{ \epsilon (2^\beta-1) } )^{1/\beta} } $ vertices. 
Let $e(K_i)$ be the number of edges of $G$ with at least one endpoint in $K_i$.
Then 
\begin{align*}
 \sum_i e(K_i) 
\leq  |E(G)| 
< (k+1-\alpha)|V(G)| 
= (1-\epsilon)(k+1)\, |V(G)|  
& \leq (k+1) \, |V(G) \setminus X |\\
& = (k+1) \sum_i |V(K_i)| .
\end{align*}
Hence  $e(K_i) < (k+1)\,|V(K_i)|$ for some $i$.
Repeatedly remove vertices from $K_i$ with at least $k+1$ neighbours outside of $K_i$. 
Doing so maintains the property that $e(K_i) < (k+1)\,|V(K_i)|$. 
Thus the final set is non-empty. 
We obtain a $k$-island of size at most $|V(K_i)| \leq \ceil{2 (  \frac{c(k+1)}{ \alpha (2^\beta-1) } )^{1/\beta} } $.
\end{proof}

The above proof depends on the following result by \citet{EM94}. \citet{LT79,LT80} implicitly proved an analogous result for planar graphs.

\begin{lem}[\citep{EM94}] 
\label{LT} 
Fix $c>0$ and $\beta\in(0,1)$. Let $G$ be a graph with $n$ vertices such that every subgraph $G'$ of $G$ has a balanced separator of size at most $c|V(G')|^{1-\beta}$. Then for all $p\geq 1$ there exists $S\subseteq V(G)$ of size at most $\frac{c 2^\beta n}{(2^\beta-1) p^\beta}$ such that each component of $G-S$ has at most $p$ vertices.
\end{lem}

\begin{proof}
Run the following algorithm. Initialise $S:=\emptyset$. While $G-S$ has a component $X$ with more than $p$ vertices, let $S_X$ be a balanced separator of $X$ with size at most $c|V(X)|^{1-\beta}$, and add $S_X$ to $S$. 

Say a component of $G-S$ at the end of the algorithm has \emph{level} 0. Say $X$ is a component of $G-S$ at some stage of the algorithm, but $X$ is not a component of $G-S$ at the end of the algorithm. Then $X$ is separated by some set $S_X$, which is then added to $S$. Define the \emph{level} of $X$ to 1 plus the maximum level of a component of $X-S_X$. 

By assumption, level 0 components have at most $p$ vertices. 
Each level 1 component has more than $p$ vertices. 
By induction on $i$, each level $i\geq 1$ component has more than $2^{i-1} p$ vertices. 
Let $t_i$ be the number of components at level $i\geq 1$. 
Say $X_1,\dots,X_{t_i}$ are the components at level $i$. 
Since level $i$ components are pairwise disjoint, 
$$t_i 2^{i-1} p < \sum_{j=1}^{t_i}|V(X_j)| \leq n, $$ 
implying $t_i<  \frac{n}{ 2^{i-1} p}$. 
The number of vertices added to $S$ by separating 
$X_1,\dots,X_{t_i}$ is at most 
$$c\sum_{j=1}^{t_i} |V(X_j)|^{1-\beta},$$
which is maximised, subject to $\sum_j|V(X_j)| \leq n$, when 
$|V(X_j)| = \frac{n}{t_i}$. Thus
$$
c\sum_{j=1}^{t_i} |V(X_j)|^{1-\beta}
\leq 
ct_i \left(\frac{n}{t_i} \right)^{1-\beta}
= ct_i^\beta n^{1-\beta}
< c\left(\frac{n}{ 2^{i-1} p} \right)^\beta n^{1-\beta}
= cn  \left(\frac{2^{1-i}}{ p} \right)^\beta 
.$$
Hence 
\begin{equation*}
|S| 
\leq c \sum_{i\geq 1} n  \left(\frac{2^{1-i}}{  p} \right)^\beta 
= \frac{cn}{p^\beta} \sum_{i\geq 1}  (2^{1-i})^\beta 
= \frac{c\,2^\beta n}{(2^\beta-1) p^\beta}.\qedhere
\end{equation*}
\end{proof}

\cref{LT} implies the following result. In the language of \citet{EM94}, this lemma provides  a sufficient condition for a graph to be `fragmentable';  this idea is extended by \citet{EF01,EF13}. 

\begin{lem}[\citep{EM94}] 
\label{LT2} 
Fix $c>0$ and $\beta\in(0,1)$. Let $G$ be a graph with $n$ vertices such that every subgraph $G'$ of $G$ has a balanced separator of size at most $c|V(G')|^{1-\beta}$. Then for $\epsilon>0$ there exists $S\subseteq V(G)$ of size at most $\epsilon |V(G)|$ such that each component of $G-S$ has at most  $\ceil{2 (  \frac{c}{ \epsilon (2^\beta-1) } )^{1/\beta} } $ vertices.
\end{lem}


We now reach the main result of this section. 

\begin{thm}[\citep{DN17}]
\label{Surface4Colour}
Every graph $G$ with Euler genus $g$ is 4-choosable with clustering $1500(g + 2)$. 
\end{thm}

\begin{proof}
We proceed by induction on $|V(G)|$. Let $L$ be a $4$-list assignment for $G$. The claim is trivial if $|V(G)|=0$. Now assume that $|V(G)|\geq 1$. First suppose that $|V(G)|\leq 6000g$. Let $v$ be any vertex of $G$. By induction, $G-v$ is $L$-colourable with clustering $1500(g+2)$. Since $|L(v)|=4$ and $|V(G-v)|< 6000g$, some colour $c\in L(v)$ is assigned to at most $1500g$ vertices in $G-v$. Colour $v$ by $c$. Thus $G$ is $L$-coloured with clustering $1500(g+2)$.  Now assume that  $|V(G)| > 6000 g$. Define $\alpha := \frac{1999}{2000}$ and $k:=3$. It follows from Euler's formula that $|E(G)| < 3(|V(G)| +g) \leq  (k+1-\alpha) |V(G)|$. Various authors~\citep{Eppstein-SODA03,GHT-JAlg84,AS-SJDM96,Djidjev87}  proved that every $n$-vertex graph with Euler genus at most $g$ has a balanced separator of size $O(\sqrt{gn})$. \citet{DMW17} proved a concrete upper bound of $2\sqrt{(2g+3)n}$. Thus \cref{SeparatorIsland} with $\beta=\frac{1}{2}$ implies that $G$ has a $3$-island of size at most 
$$
\CEIL{2 \left( \frac{ 4 \cdot 2\sqrt{2g+3} }{ \frac{1999}{2000}(\sqrt{2}-1)}\right)^{2} } 
\leq 
\CEIL{ 747 (2g+3)  } 
<  747 (2g+3)  + 1 < 
 1500 (g + 2).$$
By induction, $G-S$ is $L$-colourable with clustering $1500(g + 2)$. 
By the argument in \cref{IslandColouring}, $G$ is $L$-colourable with clustering $1500(g+2)$. 
\end{proof}

Since $\overline{S}(3,d)$ is planar, \cref{StandardClustering,Surface4Colour} imply 
$$\cchi(\EE_g)=\lcchi(\EE_g)=4.$$
It is still open to determine the best possible clustering function. 

\begin{openproblem}
Does every graph in $\EE_g$ have a 4-colouring with clustering $O(\sqrt{g})$?
\end{openproblem}

The following question also remains open; see \cref{MinorDegree} for relevant material. 

\begin{openproblem}
\label{Genus3Choosability}
Are graphs with bounded Euler genus and bounded maximum degree 3-choosable with bounded clustering?
\end{openproblem}

The above method extends for embedded graphs with large girth (since $|E(G)| < 2( |V(G)| + g)$ if the girth is at least 4, and $|E(G)| < \frac{5}{3}( |V(G)| + g)$ if the girth is at least 5). 

\begin{thm}[\citep{DN17}]
\label{Surface3Colour}
Let $G$ be a graph with Euler genus $g$ and girth $k$. 
If $k\geq 4$, then $G$ is 3-choosable with clustering $O(g)$. 
If $k\geq 5$, then $G$ is 2-choosable with clustering $O(g)$. 
\end{thm}

See \cref{KtMinorFree} for more applications of the island method. Also note that \citet{LMST08} use sublinear separators in a slightly different way (compared with \cref{SeparatorIsland}) to obtain bounds on the size of monochromatic components in 2-colourings of graphs. 

\section{Maximum Degree}
\label{MaximumDegree}

\label{DefectiveMaximumDegree}

The defective chromatic number of any graph class with bounded maximum degree equals 1. Thus defective colourings in the setting of bounded degree graphs are only interesting if one also considers the bound on the defect. 
\citet{Lovasz66} proved the following result for defective colourings of bounded degree graphs; see~\citep{Bernardi87,BK77,Lawrence78a,Gerencser65} for related results and extensions. 

\begin{thm}[\citep{Lovasz66}] 
\label{Lovasz}
For $d\geq 0$, every graph with maximum degree $\Delta$ is $k$-colourable with defect $d$, where $k:=\floor{\frac{\Delta}{d+1}}+1$. 
\end{thm}

\begin{proof}
Consider a $k$-colouring of $G$ that maximises the number of bichromatic edges.
Suppose that some vertex $v$ is adjacent to at least $d+1$ vertices of the same colour. 
Some other colour is assigned to at most $\floor{(\deg(v)-d-1)/(k-1)}\leq d$ neighbours of $v$. 
Recolour $v$ this colour. The number of bichromatic edges increases by at least 1. 
This contradiction shows that every vertex $v$ is adjacent to at most $d$ vertices of the same colour.
\end{proof}


We now show that \cref{Lovasz} is best possible. Say $G=K_n$ is $k$-colourable with defect $d$. Some monochromatic subgraph has at least $\ceil{\frac{n}{k}}$ vertices. Thus $d \geq  \ceil{\frac{n}{k}}-1$ and $k \geq  \frac{n}{d+1}=\frac{\Delta(G)+1}{d+1}$. Moreover, if $k$ does not divide $n$, then $k > \frac{\Delta(G)+1}{d+1}$, implying $k\geq \floor{\frac{\Delta(G)+1}{d+1}}+1$, which exactly matches the bound in \cref{Lovasz}. 

\label{ClusteredMaximumDegree}

\bigskip Clustered colourings of bounded degree graphs are more challenging than their defective cousins. Let $\DD_\Delta$ be the class of graphs with maximum degree $\Delta$. First note the following straightforward lemma. 

\begin{lem}
\label{MaxDegreeBlowUp}
For $\Delta > d\geq 1$, 
$$\cchi(\DD_\Delta) \leq \left(\FLOOR{\frac{\Delta}{d+1}}+1\right) \cchi(\DD_d) .$$
\end{lem}

\begin{proof}
Let $k_1:=\floor{\frac{\Delta}{d+1}}+1$ and $k_2:=\cchi(\DD_d)$. 
Let $G$ be a graph with maximum degree $\Delta$. 
By \cref{Lovasz}, $G$ is $k_1$-colourable with defect $d$. 
Each monochromatic subgraph, which has maximum degree $d$, is $k_2$-colourable with clustering $c$ (depending only on $d$).
The product gives a $k_1k_2$-colouring of $G$ with clustering $c$. 
Thus $\cchi(\DD_\Delta) \leq k_1k_2$.
\end{proof}


We now show a series of improving upper bounds on  $\cchi(\DD_\Delta)$. 
\cref{Lovasz} with $d=1$ implies  every graph with maximum degree $\Delta$ is $(\floor{\Delta/2}+1)$-colourable with defect 1, and thus with clustering 2. Hence $$\cchi(\DD_\Delta) \leq \FLOOR{\frac{\Delta}{2}}+1.$$
In particular, this shows that every graph with maximum degree $3$ is $2$-colourable with clustering 2. 
\citet{ADOV03} proved that every graph with maximum degree 4 is 2-colourable with clustering 57.
\citet{HST03} improved this bound on the cluster size from 57 to 6. 
\cref{MaxDegreeBlowUp} with $d=4$ then implies that 
$$\cchi(\DD_\Delta) \leq 2\left(\FLOOR{\frac{\Delta}{5}}+1\right).$$ \citet{ADOV03} pushed their method further to prove that 
$$\cchi(\DD_\Delta) \leq \CEIL{ \frac{\Delta+2}{3}}.$$ 
In fact, \citet{ADOV03} showed that for $\epsilon\in(0,3)$ every graph of maximum degree $\Delta$ is $\ceil{(\Delta+2)/(3-\epsilon)}$-colourable with clustering $c(\epsilon)$ (independent of $\Delta$). 

For the sake of brevity, we present slightly weaker results with simpler proofs. The following result was implicitly proved by \citet{ADOV03}.

\begin{thm}
\label{Defect2Clustering}
Let $G$ be a graph with maximum degree $\Delta$. If $G$ has a $k$-colouring with defect 2, 
then $G$ has a $(k+1)$-colouring with clustering $24\Delta$. 
\end{thm}

\begin{proof}
Say an induced cycle or path in $G$ is \emph{short} if it has at most $8\Delta$ vertices, otherwise it is \emph{long}. 
Let $X_1,\dots,X_k$ be a partition of $V(G)$ corresponding to the given $k$-colouring with defect $2$.  
Thus each $G[X_i]$ is a collection of pairwise disjoint induced cycles and paths. 
Consider such a cycle or path $Y$ that is long. Then $|Y| = a(8\Delta) + b$ for some $a\geq 1$ and $b\in[1,8\Delta]$. 
Partition $Y$ into a set of paths $Y_1,\dots,Y_a,Y_{a+1}$, where $|V(Y_j)|=8\Delta$ for $j\in[1,a]$, and $|V(Y_{a+1})|=b$. 
Here the last vertex in $Y_j$ is adjacent to the first vertex in $Y_{j+1}$ for $j\in[1,a]$, and if $Y$ is a cycle then 
the last vertex in $Y_{a+1}$ is adjacent to the first vertex in $Y_1$. 
Let $Z_1,\dots,Z_n$ be the collection of all these induced paths with exactly $8\Delta$ vertices (which might include some $Y_{a+1}$). 

Let $G'$ be the subgraph of $G$ induced by $V(Z_1\cup\dots\cup Z_n)$. 
Thus $G'$ has maximum degree at most $\Delta$. 
By \cref{Haxell} below, $G$ has a stable set $S=\{v_1,\dots,v_n\}$ with $v_i \in Z_i$ for each $i\in[n]$. 
We claim that $\{S,X_1\setminus S,\dots,X_k\setminus S\}$ defines a $(k+1)$-colouring with clustering $24\Delta$.
By construction, $S$ is a stable set in $G$. 
Say $Q$ is a component of $G[X_i\setminus S]$. 
Then $Q$ is contained in some component $Y$ of $G[X_i$]. 
If $Y$ is short, then $|Q| \leq 8\Delta$ as desired. 
Now assume that $Y$ is long. 
Let $(Y_1,\dots,Y_a,Y_{a+1})$ be the above partition of $Y$. 
Since each of $Y_1,\dots,Y_a$ has a vertex in $S$, 
$Q$ is contained in $Y_j\cup Y_{j+1} \cup Y_{j+2}$ for some $j\in[1,a+1]$, where $Y_{a+2}$ means $Y_1$ and $Y_{a+3}$ means $Y_2$. 
Thus $|Q| \leq 24\Delta$. 
\end{proof}

The proof of \cref{Defect2Clustering} used the following well-known lemma about `independent transversals'.

\begin{lem}
\label{Haxell}
Let $G$ be a graph with maximum degree at most $\Delta$. 
Let $V_1,\dots, V_n$ be a partition of $V(G)$, with $|V_i|\geq 8\Delta$ for each $i\in[n]$. 
Then $G$ has a stable set $\{v_1,\dots,v_n\}$ with $v_i \in V_i$ for each $i\in[n]$.
\end{lem}

\begin{proof}
The proof uses the Lov\'{a}sz Local Lemma~\citep{EL75}, which says that if $\XX$ is a set of events in a probability space, such that each event in $\XX$ has probability at most $p$ and is mutually independent of all  but $D$ other events in $\XX$, and $4pD\leq 1$, then with positive  probability no event in $\XX$ occurs.

We may assume that $|V_i|=8\Delta$ for $i\in[n]$. For each $i\in[n]$, independently and randomly choose one vertex $v_i\in V_i$.  Each vertex in $V_i$ is chosen with probability at most $\frac{1}{8\Delta}$.  Consider an edge $vw$, where $v\in V_i$ and $w\in V_j$.  Let $X_{vw}$ be the event that both $v$ and $w$ are chosen. Thus $X_{vw}$ has probability at most $p:=\frac{1}{64\Delta^2}$.  Observe that $X_{vw}$ is mutually independent of every event $X_{xy}$ where $x\not\in V_i\cup V_j$ and $y\not\in V_i\cup V_j$.  Thus $X_{vw}$ is  mutually independent of all but at most $D:=\Delta( |V_i| + |V_j| )=16\Delta^2$ other events. Thus $4pD=4(\frac{1}{64\Delta^2})(16\Delta^2) \leq 1$. By the Lov\'{a}sz Local Lemma, with positive probability, no event $X_{vw}$ occurs.  Hence there exist $v_1,\dots,v_n$ such that no event $X_{vw}$ occurs.  That is, $\{v_1,\dots,v_n\}$ is the desired stable set.
\end{proof}

The $8\Delta$ term in \cref{Haxell} was improved to $2\Delta$ by \citet{Haxell01}, which means the $24\Delta$ term in \cref{Defect2Clustering} can be improved to $6\Delta$. 


\cref{Lovasz} with $d=2$ and \cref{Defect2Clustering} imply
\begin{equation}
\label{BasicMaxDegreeClustering}
\cchi(\DD_\Delta) \leq \FLOOR{\frac{\Delta}{3}}+2.
\end{equation}
Answering a question of \citet{ADOV03}, 
\citet{HST03} proved that every graph with maximum degree 5 is 2-colourable with clustering less than 20000. 
For two colours, maximum degree 5 is best possible, by the Hex Lemma (\cref{HexTheorem}). Thus 
$\cchi(\DD_\Delta)=2$ if and only if $\Delta\in\{2,\dots,5\}$. 
\cref{MaxDegreeBlowUp} with $d=5$ then implies
$$\cchi(\DD_\Delta) \leq 2 \left(\FLOOR{\frac{\Delta}{6}}+1\right).$$
\citet{HST03} proved that every graph with maximum degree 8 is 3-colourable with bounded clustering. 
Using their result for the $\Delta=5$ case and the $\Delta=8$ case,  \citet{HST03} proved that 
$$\cchi(\DD_\Delta) \leq  \CEIL{\frac{\Delta+1}{3}}.$$
Moreover, \citet{HST03} proved for large $\Delta$ one can do slightly better: for some constants $\epsilon>0$ and for all $\Delta\geq\Delta_0$, 
$$\cchi(\DD_\Delta) \leq \left(\frac13-\epsilon\right)\Delta.$$
Note that for both these results by \citet{HST03} the clustering bound is independent of $\Delta$. 

It is open whether every  graph with maximum degree 9 is 3-colourable with bounded clustering~\citep{HST03}. 
If this is true, then \cref{MaxDegreeBlowUp} with $d=9$ would imply
$$\cchi(\DD_\Delta) \leq 3 \left(\FLOOR{\frac{\Delta}{10}}+1\right),$$
which would be the best known upper bound on $\cchi(\DD_\Delta)$. 
Graphs with maximum degree 10 are not 3-colourable with bounded clustering~\citep{ADOV03}, as shown by the following general lower bound, which also implies a 3-colour lower bound for graphs of maximum degree 6. 

\begin{thm}[\citep{ADOV03,HST03}] 
\label{MaxDegreeLowerBound}
For every integer $\Delta\geq 2$, 
$$\cchi(\DD_\Delta) \geq \FLOOR{\frac{\Delta+6}{4}}.$$
\end{thm}

\begin{proof}
If $\Delta$ is odd, then $\floor{\frac{\Delta+6}{4}}=\floor{\frac{(\Delta-1)+6}{4}}$, and the result for $\Delta-1$ implies the result for $\Delta$. Thus we may assume that $\Delta$ is even. Let $k:=\floor{\frac{\Delta+6}{4}}$.  
Our goal is to show that for every integer $c\geq 3$ there is a graph $G$ that has no $(k-1)$-colouring with clustering $c$. 
\citet{ES63} proved that there is a $(\frac{\Delta}{2}+1)$-regular graph $G_0$ with girth greater than $c$. Say $|V(G_0)|=n$. 
Let $G$ be the line graph of $G_0$. Then $G$ is $\Delta$-regular with $(\frac{\Delta+2}{4})n$ vertices. 
Suppose on the contrary that $G$ is $(k-1)$-colourable with clustering $c$.
For some colour class $X$ of $G$, 
$$|X| \geq \frac{|V(G)|}{k-1} = \frac{(\Delta+2)n}{4k-4} \geq n.$$
Thus $X$ corresponds to a set $X'$ of at least $n$ edges in $G_0$, which therefore contains a cycle of size greater than $c$. Thus, $X$ contains a monochromatic component of size greater than $c$, which is a contradiction. 
\end{proof}

The following problem remains open for $\Delta\geq 9$. 

\begin{openproblem}
What is $\cchi(\DD_\Delta)$? The best known bounds are 
$$ \FLOOR{\frac{\Delta+6}{4}} \leq \cchi(\DD_\Delta) \leq \CEIL{\frac{\Delta+1}{3}},$$
and  for some $\epsilon>0$ and all $\Delta$ at least some constant  $\Delta_0$.
$$  \cchi(\DD_\Delta) \leq \left(\frac13-\epsilon\right)\Delta.$$
\end{openproblem}

\begin{openproblem}
What is $\lcchi(\DD_\Delta)$? 
The best known bounds are 
$$ \FLOOR{\frac{\Delta+6}{4}} \leq \lcchi(\DD_\Delta) \leq \Delta,$$
where the upper bound follows from known Brooks-type bounds on the choice number~\citep{CR-JGT15}. 
\end{openproblem}

It is interesting that for many of the above results the bound on the clustering is independent of $\Delta$. We now show that from any upper bound on $\cchi(\DD_\Delta)$ with clustering linear in $\Delta$, one can obtain a slightly larger upper bound that is independent of $\Delta$. The idea of the proof is by \citet{ADOV03}. 

%
%

\begin{thm}
\label{GeneralMaxDegreeEpsilonClustering}
Suppose that every graph with maximum degree $\Delta$ is $(\frac{\Delta}{x}+y)$-colourable with clustering $\alpha\Delta$, for some constants $\alpha,x,y>0$. Then for every $\epsilon>0$ there is a number $c=c(\epsilon,\alpha,x,y)$ such that every  graph with maximum degree $\Delta$ is $((1+\epsilon)\frac{\Delta}{x}+y)$-colourable with clustering $c$. 
\end{thm}

\begin{proof}
Let $d:=\ceil{\frac{2}{\epsilon}(xy-1)}-1$ and $c:=\max\{\alpha d,\frac{2\alpha d}{\epsilon}\}$. 
If $\alpha\Delta \leq c$ then the result holds by assumption. 
Now assume that $\alpha\Delta > c \geq \frac{2\alpha d}{\epsilon} $ implying 
$d \leq \frac{ \epsilon\Delta}{2}$. 
By \cref{Lovasz}, $G$ is $k$-colourable with defect $d$, where $k:=\floor{\frac{\Delta}{d+1}}+1$. 
By assumption, each of these $k$ monochromatic subgraphs is $(\frac{d}{x}+y)$-colourable with clustering $\alpha d$. Thus $G$ is $k'$-colourable with clustering $\alpha d\leq c$, where
\begin{align*}
k' :=  \left(\frac{\Delta}{d+1}+1\right) \left(\frac{d}{x}+y \right)
= \frac{\Delta}{x} + \frac{\Delta( xy-1)}{x(d+1)} + \frac{d}{x} + y.
\end{align*}
Since  $xy-1 \leq \frac{\epsilon}{2}(d+1 )$ and $d \leq \frac{ \epsilon\Delta}{2}$, 
\begin{equation*}
k' \leq  \frac{\Delta}{x} + \frac{ \epsilon\Delta}{2x}   + \frac{ \epsilon\Delta}{2x} +y =  \frac{(1+\epsilon)\Delta}{x} + y. \qedhere
\end{equation*}
\end{proof}

See~\citep{BS07,LMST08,EF05,Berge08} for more results about clustered colourings of graphs with given maximum degree.



\section{Maximum Average Degree}
\label{MaximumAverageDegree}

Recall that $\mad(G)$ is the maximum average degree of a subgraph of $G$. For $m\in \mathbb{R}^+$, let $\AAA_m$ be the class of graphs $G$ with $\mad(G)\leq m$. A greedy algorithm shows that $\bigchi(\AAA_m)\leq\floor{m}+1$. \citet{HS06} determined $\dchi(\AAA_m)$ as follows.

\begin{thm}[\citep{HS06}]
\label{DefectiveMaximumAverageDegree}
For $m\in \mathbb{R}^+$, 
$$\ldchi(\AAA_m)=\dchi(\AAA_m)=\FLOOR{\frac{m}{2}}+1.$$
\end{thm}

We prove \cref{DefectiveMaximumAverageDegree} below. The key is the following lemma, which is a slightly weaker version of a result by \citet{HS06}, who proved that every graph $G$ with $\mad(G)<k+\frac{kd}{k+d}$ is $k$-choosable with defect $d$. 

\begin{lem}[Fr{\'e}d{\'e}ric Havet]
\label{WeakDefectiveMaximumAverageDegree}
For all $m\in\mathbb{R}^+$ and $k,d\in\mathbb{Z}^+$ such that $\frac{1}{k}+\frac{1}{d}\leq \frac{2}{m}$, every graph $G$ with $\mad(G)\leq m$ is $k$-choosable with defect $d$. 
\end{lem}

\begin{proof}
Let $G$ be a counterexample with $|V(G)|+|E(G)|$ minimum (with $r,k,d$ fixed). By \cref{light}, $G$ has minimum degree $k$ and has  no $d$-light edge. Let 
\begin{align*}
A:=& \{v\in V(G):\deg(v)\leq m\},\\
B:=&\{v\in V(G):m < \deg(v)\leq d\}\text{ and} \\
C:=&\{v\in V(G): d < \deg(v)\}.
\end{align*}
Since $\frac{1}{k}+\frac{1}{d}\leq \frac{2}{m}$ we have $\frac{m}{k} - 1 \leq 1 - \frac{m}{d}$. 
Let $\gamma$ be a real number with $\frac{m}{k} - 1 \leq \gamma \leq 1 - \frac{m}{d}$. 
Associate with each vertex $v$ an initial charge of $\deg(v)$. 
The total charge is $\sum_v\deg(v)=2|E(G)|$. 
Redistribute the charge as follows: for each edge $vw$ with $v\in A$, let $w$ send $\gamma$ charge to $v$. 
Note that $w\in C$ since $G$ has no $d$-light edge. 
Now, each vertex $v\in A$ has charge $(1+\gamma)\deg(v)\geq (1+\gamma)k \geq m $. 
The charge for each vertex $v\in B$ is unchanged. 
Each vertex $v\in C$ has charge at least $(1-\gamma)\deg(v)> (1-\gamma)d\geq m$. 
We may assume that $C\neq\emptyset$, as otherwise $G$ is 1-colourable with defect $d$. 
Thus the total charge is greater than $m|V(G)|$, implying the average degree is greater than $m$, which is a contradiction.
\end{proof}

\begin{proof}[Proof of \cref{DefectiveMaximumAverageDegree}] 
Let $k:= \floor{\frac{m}{2}}+1$. Thus $k>\frac{m}{2}$. Let $d:=\ceil{\frac{m^2}{4k-2m}+\frac{m}{2}}$, which is well-defined since $4k>2m$. Then
$$4d-2m \geq \frac{m^2}{k-\frac{m}{2}},$$
implying 
$$k  \geq \frac{m}{2} + \frac{m^2}{4d-2m} 
= \frac{m(2d-m) + m^2}{4d-2m} 
= \frac{md}{2d-m} = \frac{1}{\frac{2}{m}-\frac{1}{d}}.$$
Since $2d>m$, 
$$\frac{1}{k}  \leq \frac{2}{m}-\frac{1}{d}$$
and
$\frac{1}{k} +\frac{1}{d} \leq \frac{2}{m}$. 
By \cref{WeakDefectiveMaximumAverageDegree}, every graph $G$ with $\mad(G)\leq m$ is $(\floor{\frac{m}{2}}+1)$-choosable with defect $d$. Thus $\dchi(\AAA_m)\leq \ldchi(\AAA_m)\leq \FLOOR{\frac{m}{2}}+1$. 

For the lower bound, let $h:=\floor{\frac{m}{2}}$. Then the standard example $S(h,d)$ has maximum average degree less than $2h \leq m$ and is not $h$-colourable with defect $d$ by \cref{StandardDefect}. Thus $\dchi(\AAA_m)\geq h+1$, as required. 
\end{proof}

%


See~\citep{KKZ16,BK13,DKMR14,BKY13,KKZ14,BIMR12,BIMR11,BIMOR10,BI11,BoKo11,BI09a,BorKos11} for more results about defective colourings of graphs with given maximum average degree. 

Little is known about clustered colourings of graphs with given maximum average degree. 

\begin{openproblem} 
What is $\cchi(\AAA_m)$? The best known bounds are $ \floor{\frac{m}{2}}+1 \leq \cchi(\AAA_m) \leq \floor{m} +1 $. 
\end{openproblem}

Maximum average degree is closely related to degeneracy. Recall that a graph $G$ is $k$-degenerate if every subgraph of $G$ has minimum degree at most $k$. A greedy algorithm shows that every $k$-degenerate graph is properly $(k+1)$-colourable. Since the standard example $S(k,d)$ is $k$-degenerate, this bound cannot be improved even for defective colourings. Thus for the class of $k$-degenerate graphs, the defective chromatic number, defective choice number, clustered chromatic number, clustered choice number, and (proper) chromatic number all equal $k+1$. 

\section{Excluding a Subgraph}

For every graph $H$, the class of graphs with no $H$ subgraph has bounded chromatic number if and only if $H$ is a forest. The same result holds for defective chromatic number and clustered chromatic number. To see this, observe that if $H$ contains a cycle, then graphs with girth greater than $|V(H)|$ contain no $H$ subgraph, and by the classical result of \citet{Erdos59} there are graphs with arbitrarily large girth and chromatic number. 
By \cref{Bounded}, the defective and clustered chromatic numbers are also arbitrarily large. Conversely, say $F$ is a forest with $n$ vertices. A well known greedy embedding procedure shows that every graph with minimum degree at least $n-1$ contains $F$ as a subgraph. That is, every graph containing no $F$ subgraph is $(n-2)$-degenerate, and is thus $(n-1)$-colourable. This bound is tight since $K_{n-1}$ contains no $F$ subgraph and is $(n-1)$-chromatic. In short, for the class of graphs containing no $F$ subgraph, the chromatic number equals $n-1$. The following result by \citet{OOW16} shows that defective colourings exhibit qualitatively different behaviour. 

\begin{thm}[\citep{OOW16}]
\label{ExcludeTreeSubgraph}
Let $T$ be a tree with $n\geq 2$ vertices and radius $r\geq 1$. 
Then every graph containing no $T$ subgraph is $r$-colourable with defect $n-2$. 
\end{thm}

\begin{proof}
For $i=1,2,\dots,r-1$, let $V_i$ be the set of vertices $v\in V(G)\setminus(V_1\cup\dots\cup V_{i-1})$ that have at most $n-2$ neighbours in $V(G)\setminus(V_1\cup\dots\cup V_{i-1})$. Let $V_r:=V(G)\setminus(V_1\cup\dots\cup V_{r-1})$. Then $V_1\cup\dots\cup V_r$ is a partition of $V(G)$. For $i\in[1,r-1]$, by construction, $G[V_i]$ has maximum degree at most $n-2$, as desired. Suppose that $G[V_r]$ has maximum degree at least $n-1$. 
We now show that $T$ is a subgraph of $G$, where each vertex $v$ of $T$ is mapped to a vertex $v'$ of $G$. Let $x$ be the centre of $T$. 
Map the vertices of $T$ to vertices in $G$ in order of their distance from $x$ in $T$, where  $x$ is mapped to a vertex $x'$ with degree at least $n-1$ in $G[V_r]$.  The key invariant is that each vertex $v$ at distance $j\in[1,r]$ from $x$ in $T$ is mapped to a vertex $v'$ in $V_{r-j+1}\cup\dots\cup V_r$. If $j=0$ then $v=x$ and by assumption, $v'=x'$ has at least $n-1$ neighbours in $V_r$. If $j\in[1,r-1]$ then by construction, $v'$ has at least $n-1$ neighbours in $V_{r-j}\cup\dots\cup V_r$ (otherwise $v'$ would be in $V_{r-j}$). Thus there are always unmapped vertices in $V_{r-j}\cup\dots\cup V_r$ to choose as the children of $v$. Hence $T$ is a subgraph of $G$. This contradiction shows that  $G[V_r]$ has maximum degree at most $n-2$, and $G$ is $r$-colourable with defect $n-2$. 
\end{proof}

The number of colours in \cref{ExcludeTreeSubgraph} is best possible for the complete binary tree $T$ of radius $r$. 
Since $S(r-1,d)$ contains no $T$ subgraph, \cref{StandardDefect,ExcludeTreeSubgraph} imply that the defective chromatic number of the class of graphs containing no $T$ subgraph equals $r$.

\begin{openproblem}
For a tree $T$, what is the clustered chromatic number of the class of graphs with no $T$ subgraph? 
\end{openproblem}

The results in \cref{ClusteredMaximumDegree} on  $\cchi(\DD_\Delta)$ are relevant to this question since a graph has maximum degree at most $\Delta$ if and only if it excludes $K_{1,\Delta+1}$ as a subgraph. \cref{Circumference} studies colourings of graphs that exclude a given path subgraph.


\section{Excluding a Shallow Minor}

\subsection{Excluding $K_{s,t}^*$}

As illustrated in \cref{Kstar}, for integers $s,t\geq 1$, let $K_{s,t}^*$ be the bipartite graph obtained from $K_{s,t}$ by adding $\binom{s}{2}$ new vertices, each adjacent to a distinct pair of vertices in the colour class of $s$ vertices in $K_{s,t}$. \citet{OOW16} studied defective colourings for graphs excluding $K^*_{s,t}$ as a subgraph, where the defect bound depends on the density of shallow topological minors. Let $\nabla(G)$ be the maximum average degree of a graph $H$ such that the 1-subdivision of $H$ is a subgraph of $G$. 

\begin{figure}
        \begin{tikzpicture}[scale=1.3]
          \tikzstyle{every node}=[circle,draw,fill=black!50,inner sep=0pt,minimum width=4pt]
          \foreach [evaluate=\yy using int(\x+1)]\x in {1,2,3,4,5,6} {
            \foreach \y in {\yy,...,7} {
              \draw (2*\x,0) [bend right] to node [midway] {} (2*\y,0) ;
            }
          }
          \foreach \x in {1,...,7} {
            \draw (2*\x,0) node[fill=red] (a\x){};
          }
          \foreach \y in {1,...,13}  {
             \draw (\y+1,3) node[fill=blue] (b\y){};
             \foreach \x in {1,...,7}{
               \draw (a\x) to (b\y);
             }
           }
        \end{tikzpicture}
\caption{The graph $K^*_{7,13}$. \label{Kstar}}
\end{figure}
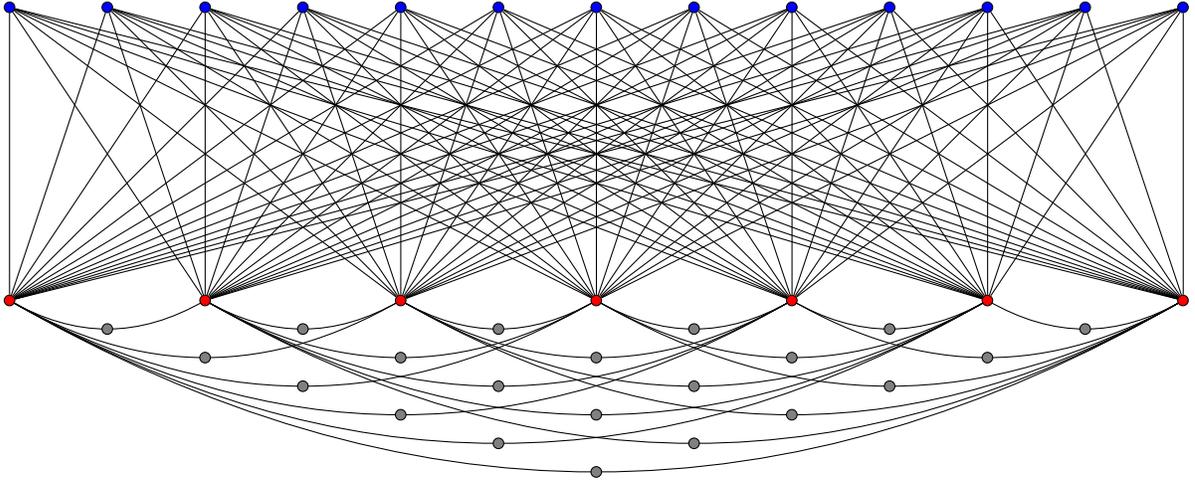

\begin{thm}[\citep{OOW16}]
\label{OOW}
Every graph $G$ with no $K_{s,t}^*$ subgraph is $s$-choosable with defect $\ell-s+1$, where  
$\delta=\mad(G)$ and $\nabla=\nabla(G)$ and
\begin{equation*}
\ell := \ell(s,t,\delta,\nabla) :=
\begin{cases}
\floor{(\delta-s)\left(\binom{\lfloor \nabla\rfloor}{s-1}(t-1)+\tfrac12 \nabla\right)+\delta} &\text{if          }s>2,\\
\floor{\tfrac12 (\delta -2)\nabla t+\delta} &\text{if          }s=2,\\
t-1&\text{if }s=1.
\end{cases}
\end{equation*}
\end{thm}

\begin{proof}
Assume for contradiction that $G$ has minimum degree at least $s$ (thus $s\leq \delta$) and that $G$ contains no 
$\ell$-light edge. The case $s=1$ is simple: Since $G$ has minimum degree at least 1, $G$ has at least one edge, which is $\ell$-light since $\Delta(G)\leq t-1$ and $\ell=t-1$. Now assume that $s\geq 2$.

Let $A$ be the set of vertices in $G$ of degree at most $\ell$. Let $B:=V(G)\setminus A$.
Let $a:=|A|$ and $b:=|B|$. Since $G$ has a vertex of degree at most $\delta$
and $\delta\leq \ell$, we deduce that $a>0$. Note that no two vertices in $A$ are adjacent.

  Since the average degree of $G$ is at most $\delta$, 
  $$(\ell+1)b+sa\leq 2|E(G)| \leq  \delta(a+b).$$
  That is, 
  \begin{equation}\label{eq:0}
  (\ell+1-\delta)b\leq (\delta-s) a.
  \end{equation}

Let $G'$ be the graph obtained from $G-E(G[B])$ by greedily finding a vertex $w\in A$ having a pair of non-adjacent neighbours $x$, $y$ in $B$ and replacing $w$ by an edge joining $x$ and $y$ (by deleting all edges incident with $w$ except $xw$, $yw$ and contracting $xw$), until no such vertex $w$ exists.

  Let $A':=V(G')\setminus B$ and $a':=|A'|$.
  Clearly the $1$-subdivision of $G'[B]$ is  a subgraph of $G$. 
  So every subgraph of $G'[B]$ has average degree at most $\nabla$. 
  Since $G'[B]$ contains at least $a-a'$ edges, 
  \begin{equation}\label{eq:1}
  a-a'\leq \tfrac12 \nabla b.
  \end{equation}

  Let $M$ be the number of cliques of size $s$ in $G'[B]$.
  Since $G'[B]$ is $\lfloor \nabla\rfloor$-degenerate, 
  $$M\leq \binom{\lfloor \nabla\rfloor}{s-1}  b$$ 
  (See~\cite[p. 25]{Sparsity} or~\citep{Wood16}).
  If $s=2$, then the following better inequality holds: 
  $$M\leq \tfrac12 \nabla b.$$

  For each vertex $v\in A'$, since $v$ was not contracted in the creation of $G'$, the set of neighbours of $v$ in $B$ is a clique of size at least $s$. 
  Thus if $a'>  M(t-1)$, 
  then there are at least $t$ vertices in $A'$ sharing at least 
  $s$ common neighbours in $B$.
  These $t$ vertices and their $s$ common neighbours in $B$ 
  with the vertices in $A-A'$ 
  form a $K_{s,t}^*$ subgraph of $G$, contradicting our assumption. 
  Thus, \begin{equation}\label{eq:2}
  a'\leq M(t-1).
  \end{equation}
  By \eqref{eq:0}, \eqref{eq:1} and \eqref{eq:2}, 
$$  \ell+1\leq (\delta-s)\left(\frac{M}{b}(t-1)+\tfrac12 \nabla\right)+\delta,$$
  contradicting the definition of $\ell$. Thus $G$ has minimum degree at most $s-1$ or $G$ contains an $\ell$-light edge. The theorem now follows from \cref{light}.
\end{proof}

\cref{OOW}, in conjunction with the standard example, determines the defective chromatic number for several graph classes of interest; see \cref{Crossings,Linkless,Knotless,ColindeVerdiere,StackQueueLayouts}. Moreover, 
\cref{OOW} determines the defective choice number for a very broad class of graphs---complete bipartite subgraphs are the key.  

\begin{thm}[\citep{OOW16,DN17}]
\label{DefectiveChoosability}
Let $\GG$ be a subgraph-closed class of graphs with $\mad(\GG)$ and $\nabla(\GG)$ bounded (which holds if $\GG$ is minor-closed). 
Then $\ldchi(\GG)$ equals the minimum integer $s$ such that $K_{s,t}\not\in \GG$ for some integer $t$. 
\end{thm}

\begin{proof}
If $K_{s,t}\not\in \GG$ for some $s,t\geq 1$, then by \cref{OOW}, every graph in $\GG$ is $s$-choosable with defect bounded by a function of $s$, $t$, $\mad(\GG)$ and $\nabla(\GG)$. This proves the claimed upper bound. Conversely, let $s:=\ldchi(\GG)$. Then for some $d$, every graph in $\GG$ is $s$-choosable with defect $d$. By \cref{Kkn} below, if $t=(ds+1)s^s$ then $K_{s,t}$ is not in $\GG$.
\end{proof}


\begin{lem}
\label{Kkn}
For $s\geq 1$ and $d\geq 0$, if $t=(ds+1)s^s$, then the complete bipartite graph $K_{s,t}$ is not $s$-choosable with defect $d$.
\end{lem}

\begin{proof}
Let $A$ and $B$ be the colour classes of $K_{s,t}$ with $|A|=s$ and $|B|=t$. 
Say $A=\{v_1,\dots,v_s\}$. 
Let $X_1,\dots,X_s$ be pairwise disjoint sets of colours, each of size $s$. 
Let $L$ be the following $s$-list assignment for $K_{s,t}$. 
Let $L(v_i):=X_i$ for each vertex $v_i\in A$.
For each vector $(c_1,\dots,c_s)$ with $c_i\in X_i$ for each $i\in[s]$, 
let $L(x):=\{c_1,\dots,c_s\}$ for $ds+1$ vertices $x$ in $B$. 
This is possible since $|B|=(ds+1)s^s$. 
Consider an $L$-colouring of $K_{s,t}$. Say each vertex $v_i$ is coloured $c_i\in L(v_i)$. 
Since $X_i\cap X_j=\emptyset$, we have $c_i\neq c_j$ for distinct $i,j\in[s]$. By construction, there are $ds+1$ vertices $x\in B$ with $L(x)=\{c_1,\dots,c_s\}$. At least $d+1$ of these vertices are assigned the same colour, say $c_i$. Thus $v_i$ has monochromatic degree at least $d+1$. 
Hence $K_{s,t}$ is not $L$-colourable with defect $d$.
Therefore $K_{s,t}$ is not $s$-choosable with defect $d$. 
\end{proof}

Note that \cref{DefectiveChoosability} generalises several previous results. For example,  \cref{DefectiveChoosability} says that $\ldchi(\OO)=2$ since $K_{2,3}$ is not outerplanar, but $K_{1,n}$ is outerplanar for all $n$. Similarly, $\ldchi(\PP)=3$ since $K_{3,3}$ is not planar, but $K_{2,n}$ is planar for all $n$. More generally, \cref{DefectiveChoosability} immediately implies:

\begin{cor}
For every graph $H$, $\ldchi(\MM_H)$ equals the minimum integer $s$ such that $H$ is a minor of $K_{s,t}$ for some integer $t$. 
\end{cor}

\cref{DefectiveChoosability} also determines the defective choice number for graphs excluding a fixed immersion (see \cref{DefectiveChoosabilityImmersion}). 

\subsection{Linklessly Embeddable Graphs}
\label{Linkless}

A graph is \emph{linklessly embeddable} if it has an embedding in $\mathbb{R}^3$ with no two linked cycles~\citep{Sachs83,RST93a}. Let $\LL$ be the class of linklessly embeddable graphs. Then $\LL$ is a minor-closed class whose minimal excluded minors are the so-called Petersen family~\citep{RST95}, which includes $K_6$, $K_{4,4}$ minus an edge, and the Petersen graph. Since linklessly embeddable graphs exclude $K_6$ minors, they are $5$-colourable~\citep{RST-Comb93} and $8$-choosable~\citep{BJW11}. It is open whether $K_6$-minor-free graphs or linklessly embeddable graphs are $6$-choosable~\citep{BJW11}. 

\citet{OOW16} determined $\dchi(\LL)$ as follows. A graph is \emph{apex} if deleting at most one vertex makes it planar. Every apex graph is linklessly embeddable~\citep{RST93a}. Since $S(2,d)$ is planar, $S(3,d)$ is apex, and thus linklessly embeddable. By \cref{StandardDefect}, $\dchi(\LL)\geq 4$. Note that the weaker lower bound, $\ldchi(\LL)\geq 4$, follows from \cref{DefectiveChoosability} since $K_{4,4}\not\in\LL$. Mader's theorem \cite{Mader68} for $K_6$-minor-free graphs implies that linklessly embeddable graphs have average degree less than 8 and minimum degree at most 7. Since linklessly embeddable graphs exclude $K_{4,4}$ minors, \cref{OOW} implies the upper bound in the following theorem.

\begin{thm}[\citep{OOW16}]
\label{LinklessDefect}
Every linklessly embeddable graph is $4$-choosable with defect $440$, and 
$$\dchi(\LL)=\ldchi(\LL)=4.$$
\end{thm}

Note that \cref{DefectiveChoosability} also implies $\ldchi(\LL)=4$ since $K_{4,4}$ is not linkless, but $K_{3,n}$ is linkless for all $n$.



We have the following result for clustered colourings of linklessly embeddable graphs.

\begin{thm}
\label{LinklessClustered}
Every linklessly embeddable graph is $5$-choosable with clustering $62948$, and 
$$\cchi(\LL)=\lcchi(\LL)=5.$$
\end{thm}

\begin{proof}
The upper bound follows from \cref{IslandsMinors} since every linkless graph contains no $K_6$-minor.
Since $\overline{S}(3,c)$ is planar, $\overline{S}(4,c)$ is apex, and thus linklessly embeddable. 
The lower bound then follows from \cref{StandardClustering}.
\end{proof}

\subsection{Knotlessly Embeddable Graphs}
\label{Knotless}

A graph is \emph{knotlessly embeddable} if it has an embedding in $\mathbb{R}^3$ in which every cycle forms a trivial knot; see~\citep{Alfonsin05} for a survey. Let $\KK$ be the class of knotlessly embeddable graphs. Then $\KK$ is a minor-closed class whose minimal excluded minors include $K_7$ and $K_{3,3,1,1}$~\citep{CG83,Foisy02}. 
More than 260 minimal excluded minors are known~\cite{GMN14}, but the full list of minimal excluded minors is unknown. Since knotlessly embeddable graphs exclude $K_7$ minors, they are $8$-colourable~\citep{AG13,Jakobsen71}. \citet{Mader68} proved that $K_7$-minor-free graphs have average degree less than 10, which implies they are $9$-degenerate and thus $10$-choosable. It is open whether $K_7$-minor-free graphs or knotlessly embeddable graphs are $6$-colourable or $7$-choosable~\citep{BJW11}. 

\citet{OOW16} determined the defective chromatic number of knotlessly embeddable graphs as follows. A graph is \emph{$2$-apex} if deleting at most two vertices makes it planar. \citet{BBFFHL07} and \citet{OT07} proved that every $2$-apex graph is knotlessly embeddable. Since every block of $S(4,d)$ is $2$-apex, $S(4,d)$ is knotlessly embeddable, as illustrated in \cref{StandardExampleKnotless}. By \cref{StandardDefect}, $\dchi(\KK)\geq 5 $. Since $K_{3,3,1,1}$ is a minor of $K_{5,3}^*$, knotlessly embeddable graphs do not contain a $K_{5,3}^*$ subgraph. Since $\mad(\KK)<10$, \cref{OOW} implies the following result.

\begin{thm}[\citep{OOW16}]
\label{KnotlessDefect}
Every knotlessly embeddable graph is $5$-choosable with defect $660$, and 
$$\dchi(\KK)=\ldchi(\KK)=5.$$
\end{thm}

\begin{figure}
\centering
\includegraphics{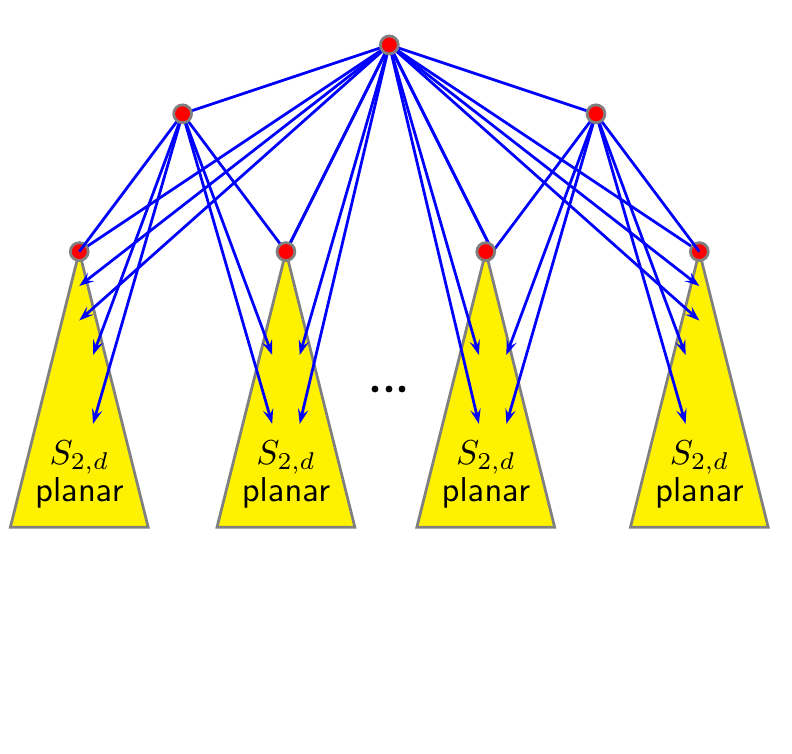}

\vspace*{-12ex}
\caption{$S(4,d)$ is knotlessly embeddable.\label{StandardExampleKnotless}}
\end{figure}

We have the following result for clustered colourings of knotlessly embeddable graphs.

\begin{thm}
\label{KnotlessClustered}
Every knotlessly embeddable graph is $6$-choosable with clustering $99958$, and 
$$\cchi(\KK)=\lcchi(\KK)=6.$$
\end{thm}

\begin{proof}
The upper bound follows from \cref{IslandsMinors} since every knotless graph contains no $K_7$-minor.
Since every block of $\overline{S}(5,d)$ is $2$-apex, $\overline{S}(5,d)$ is knotlessly embeddable. 
The lower bound then follows from \cref{StandardClustering}.
\end{proof}

\subsection{Colin de Verdi\`ere Parameter}
\label{ColindeVerdiere}

The  Colin de Verdi\`ere parameter $\mu(G)$ is an important graph invariant introduced by \citet{CdV90,CdV93}; see~\citep{HLS,Schrijver97} for surveys. It is known that $\mu(G)\leq 1$ if and only if $G$ is a disjoint union of paths, $\mu(G)\leq 2$ if and only if $G$ is outerplanar, $\mu(G)\leq 3$ if and only if $G$ is planar, and $\mu(G)\leq 4$ if and only if $G$ is linklessly embeddable. A famous conjecture of \citet{CdV90} states that $\bigchi(G)\leq \mu(G)+1$ (which implies the 4-colour theorem, and is implied by Hadwiger's Conjecture). 
\citet{OOW16} showed that for defective colourings one fewer colour suffices. Let $\VV_k:=\{G:\mu(G)\leq k\}$. 

\begin{thm}[\citep{OOW16}]
\label{Colin}
For $k\geq 1$, $$\dchi(\VV_k)=\ldchi(\VV_k)=k.$$ 
\end{thm}

\begin{proof}
$\VV_k$ is a minor-closed class~\citep{CdV90,CdV93}. \citet{HLS} proved that $\mu(K_{s,t}) = s+1$ for $t\geq\max\{s,3\}$. Thus, if $\mu(G)\leq k$ then $G$ contains no $K_{k,\max(k,3)}$ minor, and $\mad(G)\leq 2\nabla(G)\leq  O(k\sqrt{\log k})$.  \cref{OOW} with $s=k$ and $t=\max\{k,3\}$ implies that $G$ is $k$-choosable with defect $2^{O(k\log\log k)}$. Thus $\dchi(\VV_k)\leq \ldchi(\VV_k)\leq k$. For the lower bound, \citet{HLS} proved that $\mu(G)$ equals the maximum of $\mu(G')$, taken over the components $G'$ of $G$, and if $G$ has a dominant vertex $v$, then $\mu(G)=\mu(G-v)+1$. It follows that the standard example $S(k-1,d)$ is in $\VV_k$ for $d\geq 2$.  \cref{StandardDefect} then implies that $\dchi(\VV_k)\geq k$. Note that the weaker lower bound, $\ldchi(\VV_k)\geq k$, follows from \cref{DefectiveChoosability} since $K_{k,\max\{k,3\}}\not\in\VV_k$. 
\end{proof}

\cref{Colin} generalises \cref{LinklessDefect} which corresponds to the case $k=4$.

Clustered colourings provide a natural approach to the conjecture of \citet{CdV90} mentioned above. 

\begin{conj} 
\label{ClusteredColin}
$\cchi(\VV_k)=k+1$. 
\end{conj}

Note that \cref{ClusteredColin} with $k\leq 7$ is implied by \cref{IslandsMinors} below since graphs in $\VV_k$ contain no $K_{k+2}$ minor. 

\subsection{Crossings}
\label{Crossings}

This section considers defective colourings of graphs with linear crossing number. For an integer $g\geq0$ and real number $k\geq 0$, let $\EE_g^k$ be the class of graphs $G$ such that every subgraph $H$ of $G$ has a drawing on a surface of Euler genus $g$ with at most $k\,|E(H)|$ crossings. (In a drawing, we assume that no three edges cross at a  common point.)\ This says that the average number of crossings per edge is at most $2k$ (for every subgraph). Of course, a graph is planar if and only if it is in $\EE_0$, and a graph has Euler genus at most $g$ if and only if it is in $\EE_g^0$.

Graphs that can be drawn in the plane with at most $k$ crossings per edge, so called \emph{$k$-planar graphs}, are examples of graphs in $\EE_0^{(k/2)}$. \citet{PachToth97} proved that $k$-planar graphs have average degree $O(\sqrt{k})$. It follows that $k$-planar graphs are $O(\sqrt{k})$-colourable, which is best possible since $K_n$ is $O(n^2)$-planar. Say a graph is \emph{$(g,k)$-planar} if it can be drawn on a surface with Euler genus $g$ with at most $k$ crossings per edge. Such graphs are in $\EE_g^{(k/2)}$. Also note that even 1-planar graphs do not form a minor-closed class. For example, the $n\times n\times 2$ grid graph is 1-planar, as illustrated in \cref{DoubleGrid}, but contracting the $i$-th row in the front grid with the $i$-column in the back grid (plus the edge joining them) creates $K_n$ as a minor. 

\begin{figure}[h]
\centering
\includegraphics{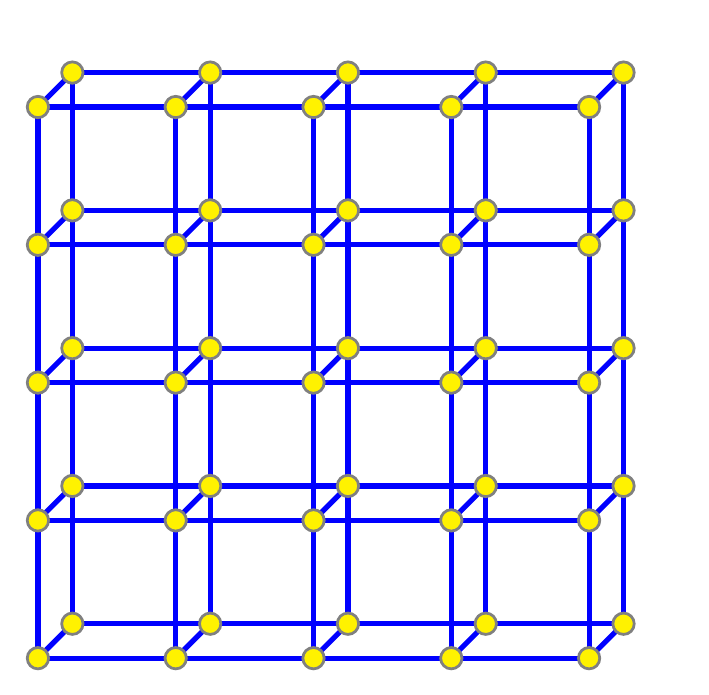}
\caption{The $n\times n\times 2$ grid graph is 1-planar.
\label{DoubleGrid}}
\end{figure}

\citet{OOW16} showed that \cref{OOW} is applicable for graphs in $\EE_g^k$. In particular, such graphs contain no $K_{3,3k(2g+3)(2g+2)+2}$ subgraph and have $\mad\leq O(\sqrt{(k+1)(g+1)})$ and $\nabla\leq O(\sqrt{(k+1)(g+1)})$. The first claim here is proved using a standard technique of counting copies of $K_{3,3}$. The second claim is proved using the crossing lemma. The next theorem follows. It is a substantial generalisation of \cref{Planar32} (the $g=k=0$ case) and \cref{GenusDefective3Colouring} (the $k=0$ case), with a worse defect bound. 

\begin{thm}[\citep{OOW16}]
\label{typegk}
For every integer $g\geq0$ and real number $k\geq 0$,
$$\dchi(\EE_g^k) = \ldchi(\EE_g^k) = 3.$$
In particular, every graph in $\EE_g^k$ is $3$-choosable with defect $O((k+1)^{5/2}(g+1)^{7/2})$. 
\end{thm}

\begin{openproblem}
What is the clustered chromatic number of $k$-planar graphs?
What is the clustered chromatic number of $(g,k)$-planar graphs?
What is the clustered chromatic number of $\EE_g^k$?
\end{openproblem}

It may be that the answer to all these questions is 4. Here we prove the answer is at most 12 for the first two questions. 

\begin{prop}
Every $(g,k)$-planar graph $G$ is 12-colourable with clustering $O((k+1)^{7/2}(g+1)^{9/2})$. 
\end{prop}

\begin{proof}
By \cref{typegk}, $G$ is 3-colourable with defect $O((k+1)^{5/2}(g+1)^{7/2})$. That is, $G$ contains three induced subgraphs $G_1$, $G_2$ and $G_3$ of $G$ each with maximum degree at most $O((k+1)^{5/2}(g+1)^{7/2})$, where $V(G)=\bigcup _i V(G_i)$. \citet{DEW17} proved that every $(g,k)$-planar graph has layered treewidth $O((g+1)(k+1))$. Apply this result to each $G_i$. Thus, for some layering $V_1,\dots,V_n$ of $G_i$, each layer $G_i[V_j]$ has treewidth $O((g+1)(k+1))$. By \cref{ClusteringDegreeTreewidth}, $G_i[V_j]$ is 2-colourable with clustering $O((k+1)^{7/2}(g+1)^{9/2})$. Within $G_i$, use two colours for odd $j$ and two distinct colours for even $j$. Each monochromatic component is contained in some $V_i$. The total number of colours is  $3 \times 2 \times 2 = 12$.
\end{proof}

%
%

\subsection{Stack and Queue Layouts}
\label{StackQueueLayouts}

A \emph{$k$-stack layout} of a graph $G$ consists of a linear ordering $v_1,\dots,v_n$ of $V(G)$ and a partition $E_1,\dots,E_k$ of $E(G)$ such that no two edges in $E_i$ cross with respect to $v_1,\dots,v_n$ for each $i\in[1,k]$. Here edges $v_av_b$ and $v_cv_d$  \emph{cross} if $a<c<b<d$. A graph is a \emph{$k$-stack graph} if it has a $k$-stack layout. The \emph{stack-number} of a graph $G$ is the minimum integer $k$ for which $G$ is a $k$-stack graph. Stack layouts are also called \emph{book embeddings}, and stack-number is also called \emph{book-thickness}, \emph{fixed outer-thickness} and \emph{page-number}. Let $\SSS_k$ be the class of $k$-stack graphs. \citet{DujWoo04} showed that $\bigchi(\SSS_k)\in\{2k,2k+1,2k+2\}$. For defective colourings, \citet{OOW16} showed that $k+1$ colours suffice.  

\begin{thm}[\citep{OOW16}]
\label{stack}
The class of $k$-stack graphs has defective chromatic number and defective choice number equal to $k+1$. In particular, every $k$-stack graph is $(k+1)$-choosable with defect $2^{O(k\log k)}$. 
\end{thm}

\begin{proof}
The lower bound follows from \cref{StandardDefect} since an easy inductive argument shows that $S(k,d)$ is a $k$-stack graph for all $d$. 
For the upper bound, $K_{k+1,k(k+1)+1}$ is not a $k$-stack graph~\citep{BK79}; see also~\citep{dKPS14}. Every $k$-stack graph $G$ has average degree less than $2k+2$ (see~\citep{BK79,DujWoo04}) and  $\nabla(G)\leq 20k^2$ (see~\citep{EMO99,NOW11}). The result follows from  \cref{OOW} with $s=k+1$ and $t=k(k+1)+1$, since $\ell(k+1,k(k+1)+1,2k+2,40k^2)\leq 2^{O(k\log k)}$.  
\end{proof}



A \emph{$k$-queue layout} of a graph $G$ consists of a linear ordering $v_1,\dots,v_n$ of $V(G)$ and a partition $E_1,\dots,E_k$ of $E(G)$ such that no two edges in $E_i$ are nested with respect to $v_1,\dots,v_n$ for each $i\in[1,k]$. Here edges $v_av_b$ and $v_cv_d$ are \emph{nested} if $a<c<d<b$. The \emph{queue-number} of a graph $G$ is the minimum integer $k$ for which $G$ has a $k$-queue layout. A graph is a \emph{$k$-queue graph} if it has a $k$-queue layout. Let $\QQ_k$ be the class of $k$-queue graphs. \citet{DujWoo04} state that determining $\bigchi(\QQ_k)$ is an open problem, and showed lower and upper bounds of  $2k+1$ and $4k$. \citet{OOW16} proved the following partial answer to this question. 

\begin{thm}[\citep{OOW16}]
\label{queue}
Every $k$-queue graph is $(2k+1)$-choosable with defect $2^{O(k\log k)}$.
\end{thm}

\begin{proof}
\citet{HR92} proved that $K_{2k+1,2k+1}$ is not a $k$-queue graph. Every $k$-queue graph $G$ has $\mad(G)<4k$ (see~\citep{HR92,Pemmaraju-PhD,DujWoo04}) and $\nabla(G)<(2k+2)^2$ (see~\citep{NOW11}). The result then follows from \cref{OOW} with $s=2k+1$ and $t=2k+1$, since 
$\ell(2k+1,2k+1,4k,2(2k+2)^2) \leq 2^{O(k\log k)}$.
\end{proof}

An easy inductive construction shows that $S(k,n)$ has a $k$-queue layout. Thus $k+1\leq \dchi(\QQ_k)\leq 2k+1$ by \cref{StandardDefect} and \cref{queue}. It remains an open problem to determine $\dchi(\QQ_k)$. 

Now consider clustered colourings of $k$-stack and $k$-queue graphs. 
The standard example $\overline{S}(2,c)$ is outerplanar, and thus has a 1-stack layout. 
An easy inductive construction then shows that $\overline{S}(k,c)$ has a $(k-1)$-stack layout (for $k\geq 2$) 
and a $k$-queue layout. The best known upper bounds come from degeneracy: 
every $k$-stack graph is $(2k+1)$-degenerate and every $k$-queue graph is $(4k-1)$-degenerate (see~\citep{DujWoo04}). 
Thus
\begin{align*}
k+2 & \leq \cchi(\SSS_k) \leq \lcchi(\SSS_k) \leq \lchi(\SSS_k) \leq 2k+2\text{ and }\\
k+1 & \leq \cchi(\QQ_k) \leq \lcchi(\QQ_k) \leq \lchi(\QQ_k) \leq 4k.
\end{align*}
Closing the gap in these bounds is interesting because the existing methods say nothing about clustered colourings of $k$-stack or $k$-queue graphs. For example, \cref{SeparatorIsland} is not applicable since 3-stack and 2-queue graphs do not have sublinear balanced separators. Indeed, \citet{DSW16} constructed (cubic bipartite) 3-stack expander graphs and (cubic bipartite) 2-queue expander graphs.


%
%

\subsection{Excluded Immersions}
\label{Immersions}

This section considers colourings of graphs excluding a fixed immersion. A graph $G$ contains a graph $H$ as an \emph{immersion} if the vertices of $H$ can be mapped to distinct vertices of $G$, and the edges of $H$ can be mapped to pairwise edge-disjoint paths in $G$, such that each edge $vw$ of $H$ is mapped to a path in $G$ whose endpoints are the images of $v$ and~$w$. The image in $G$ of each vertex in $H$ is called a \emph{branch vertex}. A graph $G$ contains a graph $H$ as a \emph{strong immersion} if $G$ contains $H$ as an immersion such that for each edge $vw$ of $H$, no internal vertex of the path in $G$ corresponding to $vw$ is a branch vertex.  Let $\II_t$ be the class of graphs not containing $K_t$ as an immersion.  Let $\II'_t$ be the class of graphs not containing $K_t$ as a strong immersion. 

\citet{LM89} and \citet{AL03} independently conjectured that every
$K_t$-immersion-free graph is properly $(t-1)$-colourable. Often motivated
by this question, structural and colouring properties of graphs excluding a
fixed immersion have recently been widely studied. 
The best upper bound, due to \citet{GLW17}, says that every $K_t$-immersion-free
graph is properly $(3.54t+3)$-colourable

\VdHW\ proved that the defective chromatic number of $K_t$-immersion-free graphs equals~2. The proof, presented below, is based on the following structure theorem of \citet{DMMS13}.  Almost the same result can be concluded from a structure theorem by \citet{Wollan15}.  Since \cref{WeakImmStructure} is not proved explicitly in~\citep{DMMS13} we include the full proof, which relies  on the following definitions. 
For each edge $xy$ of a tree $T$, let $T(xy)$ and $T(yx)$ be the components of $T-xy$, where $x$ is in~$T(xy)$ and $y$ is in~$T(yx)$. For a tree $T$ and graph $G$, a \emph{$T$-partition} of $G$ is a partition $(T_x\subseteq V(G):x\in V(T))$ of~$V(G)$ indexed by the nodes of $T$. Each set $T_x$ is called a \emph{bag}. Note that a bag may be empty. For each edge $e=xy\in E(T)$, let $G(T,xy)=\bigcup_{z\in V(T(xy))}T_z$ and $G(T,yx)=\bigcup_{z\in V(T(yx))}T_z$, and let $E(G,T,e)$ be the set of edges in $G$ between $G(T,xy)$ and $G(T,yx)$. The \emph{adhesion} of a $T$-partition is the maximum, taken over all edges~$e$ of $T$, of $|E(G,T,e)|$. For each node $x$ of $T$, the \emph{torso of $x$} (with respect to a $T$-partition) is the graph obtained from $G$ by identifying $G(T,yx)$ into a single vertex for each edge $xy$ incident to $x$, deleting resulting parallel edges and loops. Note that a tree $T$ for which $V(G)\subseteq V(T)$ implicitly defines a $T$-partition of $G$ with $T_x=\{x\}\cap V(G)$ for each node $x\in V(T)$. 

\begin{lem}[\citep{DMMS13}]
\label{WeakImmStructure}
  For every graph $G$ that does not
  contain $K_t$ as an immersion, there is a tree $T$ and a $T$-partition of
  $G$ with adhesion less than $(t-1)^2$, such that each bag has at most
  $t-1$ vertices.
\end{lem}

\begin{proof}
\citet{GH61} proved that for every graph $G$ there is a tree $F$ with vertex set $V(G)$ such that for all distinct vertices $v,w\in V(G)$, the size of the smallest edge-cut in $G$ separating $v$ and $w$ equals $$\zeta(e):= \min_e |E(G,F,e)|,$$ 
where the minimum is taken over all edges $e$ on the $vw$-path in $F$. Let $S$ be the set of edges $e\in E(F)$ with $\zeta(e)<(t-1)^2$.

Suppose that some component $X$ of $F-S$ has at least $t$ vertices. Let
   $x,v_2,v_3,\dots,v_t$ be distinct vertices in $X$. Let $G'$ be the
   multigraph obtained from $G$ by adding a new vertex $w$ and adding
   $t-1$ parallel edges between $w$ and $v_i$ for each $i\in[2,t]$. We
   claim that $G'$ contains $(t-1)^2$ edge-disjoint $xw$-paths. Consider a
   set of vertices $R\subseteq V(G)$ with $x\in R$, such that there are
   less than $(t-1)^2$ edges between $R$ and $V(G')-R$ in $G'$. Since
   $w\in V(G')-R$ has degree $(t-1)^2$, some neighbour $v_i$ of $w$ is
   also in $V(G')-R$ (where $i\in[2,t]$). Thus in $G$, there is an
   edge-cut with less than $(t-1)^2$ edges separating $x$ and $v_i$,
   meaning that $\zeta(e)<(t-1)^2$ for some edge $e$ on the $xv_i$-path in
   $F$. But every such edge $e$ is in $X$, which implies
   $\zeta(e)\geq(t-1)^2$. This contradiction shows (by Menger's Theorem)
   that $G'$ contains $(t-1)^2$ edge-disjoint $xw$-paths. Hence $G$
   contains $(t-1)^2$ edge-disjoint paths starting at $x$, exactly $t-1$
   of which end at $v_i$ for each $i\in[2,t]$. Label these $xv_i$-paths as
   $P_{i,j}$, for $j\in[2,t]$. For $j\ne i$, the combined path
   $P_{i,j}P_{j,i}$ is a $v_iv_j$-path, while each $P_{i,i}$ is a $v_ix$
   path. Since all these paths are edge-disjoint, $G$ contains a
   $K_t$-immersion. This contradiction shows that every component of $F-S$
   has at most $t-1$ vertices.

   Let $T$ be obtained from $F$ by contracting each connected component of
   $F-S$ into a single vertex. For each node $x$ of $T$, let $T_x$ be the
   set of vertices contracted into $x$. Then $\bigl(T_x:x\in V(T)\bigr)$
   is the desired partition.
 \end{proof}


\begin{lem}[\citep{vdHW}]
\label{Bijection}
If a graph $G$ has a $T$-partition with $V(G)\subseteq V(T)$ and adhesion at most $k$, then $G$ is  $2$-colourable with defect $k$.
\end{lem}



\begin{proof}
We proceed by induction  on $|V(G)|+|E(G)|$. The result is trivial if
  $|V(G)|\leq2$. Now assume $|V(G)|\geq3$. Call a vertex~$v$ of~$G$
  \emph{large} if $\deg_G(v)\geq k+1$; otherwise~$v$ is \emph{small}.
If $\deg_G(v)\leq 1$ for some vertex $v$, then by induction, $G-v$ is 2-colourable with defect $k$; assign $v$ a colour distinct from its neighbour. Now $G$ is 2-coloured with defect $k$. Now assume that $G$ has minimum degree at least $2$. If two small vertices $v$ and $w$ are adjacent, then by induction, $G-vw$ is 2-colourable with defect $k$, which is also a 2-colouring of $G$ with defect $k$. Now assume that the small vertices are a stable set. 
  If $G$ has no large vertices, then every 2-colouring of $G$ has defect $k$. Now assume that $G$ has some large vertex. 
  Let $X$ be the union of all paths in $T$ whose endpoints are large. Let $u$ be a leaf in $X$. Thus $u$ is large. Let $v$ be the neighbour of $u$ in $X$, or any neighbour of $u$ if $|V(X)|=1$. 
 Let $Y:= V(T(uv))\setminus\{u\}$. Every vertex in $Y$ is small.  Since no two small vertices are adjacent and $G$ has minimum degree at least 2, every vertex  in $Y$ has at least one neighbour in $T(vu)$. Also, $u$ has at least $\deg_G(u)-|Y|$ neighbours in $T(vu)$. Thus $|E(G,T,uv)| \geq |Y| +  \deg_G(u) - |Y| \geq k+1$, which is a contradiction.
\end{proof}

Note that \cref{Bijection} with a $O(k^3)$ defect bound can be concluded from 
\cref{OOW} with $s=2$.


The following result by \VdHW\ is the first main contribution of this section. 

\begin{thm}[\citep{vdHW}]
\label{2DefectiveImmersion}
Every graph $G\in\II_t$ is $2$-colourable with defect $(t-1)^3$,  
and $$\dchi(\II_t)=2.$$
\end{thm}

\begin{proof}
  By \cref{WeakImmStructure}, there is a tree $T$ and a $T$-partition of
  $G$ with adhesion at most $(t-1)^2-1$, such that each bag has at most
  $t-1$ vertices. Let $Q$ be the graph with vertex set $V(T)$, where $xy\in
  E(Q)$ whenever there is an edge of $G$ between $T_x$ and $T_y$. Any one
  edge of~$Q$ corresponds to at most $t-1$ edges in $G$. By
  \cref{Bijection}, the graph $Q$ is $2$-colourable with defect
  $(t-1)^2-1$. Assign to each vertex $v$ in $G$ the colour assigned to the
  vertex $x$ in~$Q$ with $v\in T_x$. Since at most $t-1$ vertices of $G$
  are in each bag, $G$ is $2$-coloured with defect $(t-1)\cdot\bigl((t-1)^2-1\bigr)+(t-2)<(t-1)^3$.
\end{proof}

\VdHW\  proved the following result for excluded strong immersions. The omitted proof is based on a more involved structure theorem of \citet{DW16}. 

\begin{thm}[\citep{vdHW}]
\label{2DefectiveStrongImmersion}
Every graph $G\in \II'_t$ is $2$-colourable with defect at most some function $d(t)$, and thus $$\dchi(\II'_t)=2.$$
\end{thm}

While only two colours suffice for defective colourings of graphs excluding a fixed immersion, significantly more colours are needed for defective list colouring.

\begin{thm}
\label{DefectiveChoosabilityImmersion} 
For all $t\geq 2$, $$\ldchi(\II_t)=\ldchi(\II'_t)=t-1.$$
\end{thm}

\begin{proof} \citet{DDFMMS14} proved that $\mad(\II'_t)\leq O(t)$. If the 1-subdivision of a graph $H$ is a subgraph of $G$, then $G$ contains $H$ as a strong immersion, implying $\nabla(\II'_t)\leq O(t)$. \cref{DefectiveChoosability} then implies that $\ldchi(\II'_t)$ equals the minimum integer $s$ such that $K_{s,n}$ contains $K_t$ as a strong immersion for some $n$. It is easily seen that $K_{t-1,\binom{t-1}{2}+1}$ contains $K_t$ as a strong immersion, but $K_{t-2,n}$ does not contain $K_t$ as a weak immersion for all $n$. The result follows. 
\end{proof}

It is an open problem to determine the clustered chromatic number and clustered choice number of graphs excluding a (strong or weak) $K_t$ immersion. We have the following lower bounds. Since every graph with maximum degree at most $t-2$ contains no $K_t$ immersion, by \cref{MaxDegreeLowerBound}, 
$$\cchi(\II'_t) \geq \cchi(\II_t) \geq \cchi(\DD_{t-2}) \geq \FLOOR{\frac{t+4}{4}}.$$
By \cref{DefectiveChoosabilityImmersion}, 
$$\lcchi(\II'_t) \geq \lcchi(\II_t) \geq \ldchi(\II_t) =t-1.$$

\section{Minor-Closed Classes}

This section studies defective and clustered colourings of graphs in a minor-closed class. Hadwiger's Conjecture states that every $K_t$-minor-free graph is $(t-1)$-colourable~\citep{Hadwiger43}. This is widely considered one of the most important open problems in graph theory; see~\citep{SeymourHC} for a survey. Defective and clustered colourings of $K_t$-minor-free graphs provide an avenue for attacking Hadwiger's Conjecture. 

\subsection{Excluding a Minor and Bounded Degree}
\label{MinorDegree}

All the known examples of planar graphs that are not 3-colourable with bounded clustering have unbounded maximum degree. This observation motivated several authors~\citep{KMRV97,ADOV03,LMST08} to ask whether planar graphs with bounded maximum degree are $3$-colourable with bounded clustering. \citet{EJ14} solved this question in the affirmative and extended the result to graphs of bounded Euler genus. 

\begin{thm}[\citep{EJ14}]
\label{ClusteringDegreeGenus}
Every graph with maximum degree $\Delta$ and Euler genus $g$ is $3$-colourable with clustering $f(\Delta,g)$, for some function $f$.
\end{thm}

\citet{EJ14} posed the following open problem

\begin{openproblem}[\citep{EJ14}]
Are triangle-free planar graphs with bounded maximum degree 2-colourable with bounded clustering?
\end{openproblem}

Graphs with treewidth $k$ are properly $(k+1)$-colourable, which is best possible for $K_{k+1}$. Moreover, since the standard example $S(k,d)$ has treewidth $k$, the defective chromatic number of graphs with treewidth $k$ equals $k+1$. On the other hand, \citet{ADOV03} proved that every graph with maximum degree $\Delta$ and treewidth $k$ is 2-colourable with clustering $24k\Delta$. The proof was based on a result about tree-partitions by \citet{DO95}, which was improved by \citet{Wood09}. It follows that:

\begin{thm}
\label{ClusteringDegreeTreewidth}
Every graph with maximum degree $\Delta$ and treewidth $k$ is 2-colourable with clustering $\frac52 (k + 1)(\frac72 \Delta - 1)$. 
\end{thm}

\citet{Liu15} proved a list colouring analogue of \cref{ClusteringDegreeTreewidth}: every graph with maximum degree $\Delta$ and treewidth $k$ is 2-choosable with clustering $f(k,\Delta)$, for some function $f$. 

%

\citet{DDOSRSV04} proved the conjecture of Robin Thomas that graphs excluding a fixed minor can be 2-coloured so that each monochromatic subgraph has bounded treewidth. \citet{ADOV03} observed that this result and  \cref{ClusteringDegreeTreewidth} together imply that graphs excluding a fixed minor and with bounded maximum degree  are $4$-colourable with bounded clustering. Answering a question of \citet{EJ14}, this result was improved by \citet{LO17}  (using the graph minor structure theorem):

\begin{thm}[\citep{LO17}]
Every graph containing no $H$-minor and with maximum degree $\Delta$ is $3$-colourable with clustering $f(\Delta,H)$, for some function $f$.
\end{thm}

This theorem generalises \cref{ClusteringDegreeGenus} above. It is open whether graphs with bounded maximum degree and excluding a fixed minor are $3$-choosable with bounded clustering (\cref{Genus3Choosability} is a special case). 

The above results lead to the following connection between clustered and defective colourings, which was implicitly observed by  \citet{EKKOS15}. The second observation was made by \citet{NSSW}. 

\begin{lem}[\citep{EKKOS15,NSSW}]
\label{Defective3Clustered}
For every minor-closed class $\GG$, 
$$\cchi(\GG)\leq 3 \dchi(\GG).$$
Moreover, if some planar graph is not in $\GG$ then
$$\cchi(\GG)\leq 2 \dchi(\GG).$$
\end{lem}

\begin{proof}
As mentioned above, \citet{LO17} proved that for every integer $d$ there is an integer $c=c(\GG,d)$ such that every graph in $\GG$ with maximum degree $d$ has a $3$-colouring with clustering $c$. Let $k:=\dchi(\GG)$. That is, for some integer $d$, every graph $G$ in $\GG$ is $k$-colourable with defect $d$. Apply the result of \citet{LO17} to each monochromatic component of $G$, which has maximum degree at most $d$. Then $G$ is $3k$-colourable with clustering $c$, and $\cchi(\GG)\leq 3k$. For the second claim, \citet{RS-V} proved that a minor-closed class not containing all planar graphs has bounded treewidth. The result follows from the method used above, with \cref{ClusteringDegreeTreewidth} in place of the result of \citet{LO17}.
\end{proof}

\subsection{$K_t$-Minor-Free Graphs}
\label{KtMinorFree}

First consider defective colourings of $K_t$-minor-free graphs. \citet{EKKOS15} proved that every $K_t$-minor-free graph is $(t-1)$-colourable with defect $O(t^2\log t)$. Their proof gives the same result for $(t-1)$-choosability. This result is implied by \cref{OOW} since $K_{t-1,1}^*$ contains a $K_t$ minor. Indeed, the proof of \cref{OOW} in this case is identical to the proof of \citet{EKKOS15} (which predated~\citep{OOW16}). \VdHW\  improved the upper bound on the defect in the result of  \citet{EKKOS15} to $t-2$; see \cref{vdHW} below. \citet{EKKOS15} also showed that the standard example $S(t-2,d)$ is $K_t$-minor-free. Thus, \cref{StandardDefect} implies the following defective version of Hadwiger's Conjecture:

\begin{equation}
\label{DefectiveChromaticNumberNoKtMinor}
\dchi(\MM_{K_t})=\ldchi(\MM_{K_t})=t-1.
\end{equation}

Now consider clustered colourings of $K_t$-minor-free graphs. Note that \eqref{DefectiveChromaticNumberNoKtMinor} implies
$$\cchi(\MM_{K_t})\geq t-1.$$
We now show that the island-based method of \citet{DN17} determines  $\cchi(\MM_{K_t})$ for $t\leq 9$. 

\begin{thm}[\citep{DN17}] 
\label{IslandsMinors}
For $t\leq 9$, $$\cchi(\MM_{K_t})=t-1.$$ In particular, every $K_t$-minor-free graph $G$ is $(t-1)$-choosable with clustering  $$c_t:=\CEIL{2 \left( \frac{ 5 t^{3/2} }{ \sqrt{2}-1}\right)^{2} } .$$
\end{thm}

\begin{proof}
For $t\leq 9$, the exact extremal function for $K_t$-minor-free graphs is known~\citep{Dirac52,Mader67,Mader68,Mader68b,Jorgensen94,ST06}. In particular, every $K_t$-minor-free graph on $n$ vertices has less than $(t-2)n$ edges for $t\leq 9$. For all $t$, every such graph has a balanced separator of size $t^{3/2}n^{1/2}$~\citep{AST90}. By \cref{SeparatorIsland} with $k=t-2$ and $\alpha=1$ and $\beta=\frac12$ and $c=t^{3/2}$, $G$ has a $(t-2)$-island of size at most $c_t$. By \cref{IslandColouring}, $G$ is $(t-1)$-choosable with clustering $c_t$. Since the standard example $S(t-2,d)$ is $K_t$-minor-free, the clustered chromatic number of $K_t$-minor-free graphs is at least $t-1$. Thus for $t\leq 9$, the clustered chromatic number of $K_t$-minor-free graphs equals $t-1$. 
\end{proof}

The only obstacle for extending \cref{IslandsMinors} for larger values of $t$ is that the exact extremal function is not precisely known. Moreover, for large $t$, the maximum average degree of $K_t$-minor-free graphs tends to $\Theta(t\sqrt{\log t})$; see~\citep{Kostochka82,Kostochka84,Thomason84,Thomason01}. Thus \cref{SeparatorIsland} alone cannot determine the clustered chromatic number of $K_t$-minor-free graphs for large $t$. 

\citet{KawaMohar-JCTB07} first proved a $O(t)$ upper bound on 
$\cchi(\MM_{K_t})$. The constants in this result have been successively improved, as shown in \cref{SuccessiveImprovements}, to 
$$\cchi(\MM_{K_t})\leq 2t-2.$$ 
\citet{DN17} have announced a proof that $\cchi(\MM_{K_t})=\lcchi(\MM_{K_t})=t-1$.




%
\setlength{\tabcolsep}{5pt}

\begin{table}[!h]
\caption{\label{SuccessiveImprovements} Upper bounds on the clustered chromatic number of $K_t$-minor-free graphs ($t\geq 3$). }
\begin{center}
\begin{tabular}{llcc}
\hline
 & $\cchi(\MM_{K_t})\leq $ & clustering & choosability \\
  \hline
\citet{KawaMohar-JCTB07} &  $\ceil{\frac{31}{2}t}$ &  $c(t)$ & yes  \\
\citet{Wood10}\,\footnotemark[1] & $\ceil{\frac{7t-3}{2}}$ & $c(t)$ & yes  \\
\citet{EKKOS15} &  $4t-4$ & $c(t)$ & \\
\citet{LO17} & $3t-3$ & $c(t)$ & \\
\citet{Norin15}\,\footnotemark[2] & $2t-2$ & $c(t)$ & \\
\VdHW\ & $2t-2$ & $\lceil{\frac{t-2}{2}}\rceil$ & \\
\citet{DN17} & $2t-2$ & $c(t)$ &  \\
  \hline
\end{tabular}
\end{center}
\end{table}

\footnotetext[1]{This result depended on a result announced by Norine and Thomas \citep{NorineThomas,Thomas09} which has not yet been written.}

\footnotetext[2]{See~\citep{SeymourHC} for some of the details.}

Most of the results shown in \cref{SuccessiveImprovements} depend on the graph minor structure theorem, so the clustering function is large and not explicit. The exception is the self-contained proof of \VdHW, which we now present.

\begin{lem}[\citep{vdHW}]
  \label{MinimalInducedConnectedSubgraph}
  For every set $A$ of $k\geq1$ vertices in a connected graph $G$, and for every
  minimal induced connected subgraph $H$ of $G$ containing $A$,
  \begin{enumerate}[label=(\arabic*),itemsep=0ex,topsep=0ex]
  \item $H$ has maximum degree at most $k$, and
  \item $H$ can be $2$-coloured with clustering $\bceil{\half k}$.
  \end{enumerate}
\end{lem}

\begin{proof}
Say $v\in V(H)$.   Let $T$ be a spanning tree of $H$ that includes each edge of $H$ incident with $v$.   By the minimality of $H$, every leaf of $T$ is in $A$. Thus, the number of leaves in $T$ is at least $\deg_H(v)$ and at most $k$. Hence $\deg_H(v)\leq k$, implying $\Delta(H)\leq k$.

To prove (2) by induction on $|V(H)|$. In the base case,
  $|V(H)|=|A|=k$ and the result is trivial. Now assume that $|V(H)|>k$.
  Thus $V(H-A)\ne\varnothing$, and by the minimality of $H$, every vertex
  in $H-A$ is a cut-vertex of $H$. Consider a leaf-block~$L$ of~$H$. Every
  vertex in $L$, except the one cut-vertex in $L$, is in $A$. There are at
  least two leaf-blocks. Thus $|V(L-v)|\leq\half k$ for some leaf block
  $L$, where $v$ is the one cut-vertex of $H$ in~$L$. Let $H'=H-V(L-v)$ and
  $A'=(A\setminus V(L))\cup\{v\}$. Then $H'$ is a minimal induced connected
  subgraph of $G$ containing $A'$, and $|A'|\leq k$. By induction, $H'$ has
  a $2$-colouring with clustering $\bceil{\half k}$. Colour every vertex in
  $L\setminus\{v\}$ by the colour not assigned to $v$ in $H'$. Now $H$ is
  $2$-coloured with clustering $\bceil{\half k}$.
\end{proof}

Disjoint subgraphs $H_1$ and $H_2$ in a  graph $G$ are \emph{adjacent} if some edge of $G$ has one endpoint in $H_1$ and one endpoint in $H_2$. 

\begin{lem}[\citep{vdHW}]
\label{Decomposition}
  For $t\geq4$, for every $K_t$-minor-free graph $G$, there are induced subgraphs  
   $H_1,\dots,H_n$ of $G$ that partition $V(G)$, and for $i\in[n]$:
  \begin{enumerate}[label=(\arabic*),itemsep=0ex,topsep=0ex]
  \item The subgraph $H_i$ has maximum degree at most $t-2$, and can be $2$-coloured with clustering  
  $\bceil{\half(t-2)}$;
  \item For each component $C$ of $G-\bigl(V(H_1)\cup\dots\cup V(H_i)\bigr)$, at most $t-2$ subgraphs in $H_1,\dots,H_i$ are adjacent to $C$, and these subgraphs are pairwise adjacent. 
  \end{enumerate}
\end{lem}

\begin{proof}
  We may assume that $G$ is connected. We construct $H_1,\dots,H_n$  iteratively, maintaining properties (1) and (2). Let $H_1$ be the  subgraph induced by a single vertex in $G$. Then (1) and~(2) hold for  $i=1$.

  Assume that $H_1,\dots,H_i$ satisfy (1) and (2) for some $i\geq1$, but  $V(H_1),\dots,V(H_i)$ do not partition $V(G)$. Let $C$ be a component of  $G-\bigl(V(H_1)\cup\dots\cup V(H_i)\bigr)$. Let $Q_1,\dots,Q_k$ be the elements of $\{H_1,\dots,H_i\}$ that are adjacent to $C$. By (2),   $Q_1,\dots,Q_k$ are pairwise adjacent and $k\leq t-2$. Since $G$ is  connected, $k\geq1$.

  For $j\in[k]$, let $v_j$ be a vertex in $C$ adjacent to $Q_j$.   By \cref{MinimalInducedConnectedSubgraph} with  $k\leq t-2$, there is an induced connected subgraph $H_{i+1}$ of $C$  containing $v_1,\dots,v_k$ that satisfies~(1).

  Consider a component $C'$ of
  $G-\bigl(V(H_1)\cup\dots\cup V(H_{i+1})\bigr)$. Either $C'$ is disjoint
  from $C$, or $C'$ is contained in $C$. If $C'$ is disjoint from $C$, then
  $C'$ is a component of $G-\bigl(V(H_1)\cup\dots\cup V(H_i)\bigr)$
  and~$C'$ is not adjacent to $H_{i+1}$, implying (2) is maintained for~$C'$.

  Now assume $C'$ is contained in~$C$. The
  subgraphs in $H_1,\dots,H_{i+1}$ that are adjacent to $C'$ are a subset
  of $Q_1,\dots,Q_k,H_{i+1}$, which are pairwise adjacent. Suppose that
  $k=t-2$ and $C'$ is adjacent to all of $Q_1,\dots,Q_{t-2},H_{i+1}$. Then
  $C$ is adjacent to all of $Q_1,\dots,Q_{t-2}$. Contracting each of
  $Q_1,\dots,Q_{t-2}, H_{i+1}, C'$ into a single vertex gives $K_t$ as a
  minor of~$G$, which is a contradiction. Hence~$C'$ is adjacent to at most $t-2$ of
  $Q_1,\dots,Q_{t-2},H_{i+1}$, and property (2) is maintained for~$C'$.
\end{proof}

The following is the main result of \VdHW.

\begin{thm}[\citep{vdHW}]
\label{vdHW}
For $t\geq 4$, every $K_t$-minor-free graph $G$ is $(t-1)$-colourable with defect $t-2$ and is $(2t-2)$-colourable with clustering $\ceil{\half(t-2)}$.
\end{thm}

\begin{proof}
Let $H_1,\dots,H_n$ be the subgraphs from \cref{Decomposition}. By property (2) in \cref{Decomposition}, each subgraph $H_{i}$ is adjacent to at most $t-2$ of $H_1,\dots,H_{i-1}$. For $i=1,2,\dots,n$ colour $H_i$ with one of $t-1$ colours different from the colours assigned to the at most $t-2$ subgraphs in $H_1,\dots,H_{i-1}$ adjacent to $H_i$. Since $H_i$ has maximum degree $t-2$, we obtain a $(t-1)$-colouring of $G$ with defect $t-2$. Each $H_i$ has a  $2$-colouring with clustering $\bceil{\half(t-2)}$. The product of these colourings is a  $(2t-2)$-colouring with clustering $\ceil{\half(t-2)}$. 
\end{proof}

Note that the $O(t)$ upper bound on $\cchi(\MM_{K_t})$ has been extended to the setting of odd minors. In particular, \citet{Kawa08} proved that every graph with no odd $K_t$-minor is $496t$-colourable with clustering at most some function $f(t)$. The bound of $496t$ was improved to $10t-13$ by \citet{KO16}. By \cref{BoundedChoice}, these results do not generalise to the setting of choosability (even for defective colourings), since complete bipartite graphs contain no odd $K_3$ minor and have unbounded maximum average degree.

\subsection{$H$-Minor-Free Graphs}
\label{HMinorFree}

Hadwiger's Conjecture implies that for every graph $H$ with $t$ vertices, the maximum chromatic number of $H$-minor-free graphs equals $t-1$ (since $K_{t-1}$ is $H$-minor-free). However, for clustered and defective colourings, fewer colours often suffice. For example, as discussed in \cref{SurfacesDefective}, \citet{Archdeacon87} proved that graphs embeddable on a fixed surface are defectively 3-colourable, whereas the maximum chromatic number for graphs of Euler genus $g$ is $\Theta(\sqrt{g})$. The natural question arises: what is the defective or clustered chromatic number of the class of $H$-minor-free graphs, for an arbitrary graph $H$? We will see that the answer depends on the structure of $H$, unlike the chromatic number which only depends on $|V(H)|$. 

\citet{OOW16} observed that the following definition is a key to answering this question. Let $T$ be a rooted tree. The \emph{depth} of $T$ is the maximum number of vertices on a root--to--leaf path in $T$. The \emph{closure} of $T$ is obtained from $T$ by adding an edge between every ancestor and descendent in $T$, as illustrated in \cref{ClosureBroom}. The \emph{connected tree-depth}\footnote{This definition is a variant of the more commonly studied notion of the \emph{tree-depth} of $H$, denoted by $\td(H)$, which equals the maximum connected tree-depth of the connected components of $H$. See~\citep{Sparsity} for background on tree-depth. If $H$ is connected, then $\td(H)=\ctd(H)$. In fact, $\td(H)=\ctd(H)$ unless $H$ has two connected components $H_1$ and $H_2$ with $\td(H_1)=\td(H_2)=\td(H)$, in which case $\ctd(H)=\td(H)+1$. \citet{NSSW} introduced connected tree-depth to avoid this distinction.} of a graph $H$, denoted by $\ctd(H)$, is the minimum depth of a rooted tree $T$ such that $H$ is a subgraph of the closure of $T$. Note that the connected tree-depth is closed under taking minors. 

\begin{figure}[h]
\centering
\includegraphics{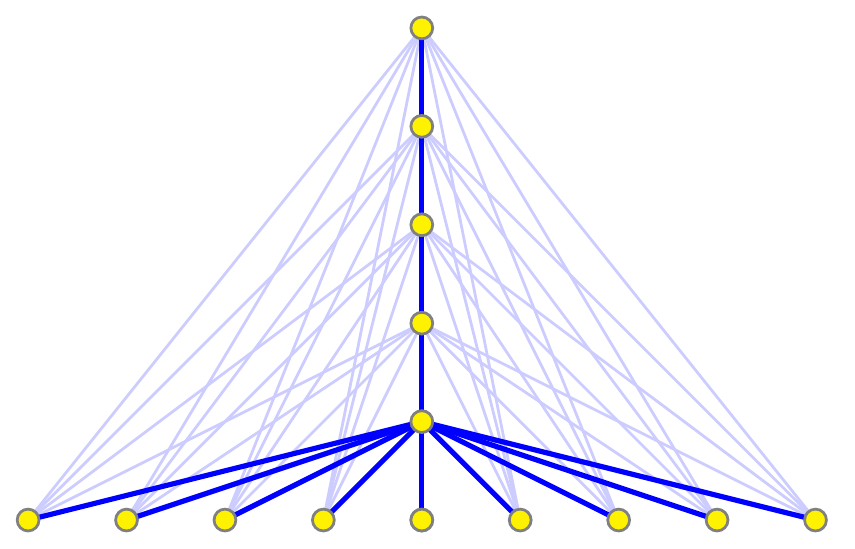}
\caption{\label{ClosureBroom} The closure of a tree of depth 6 contains $K_{5,9}$.}
\end{figure}

Note that the standard example, $S(h,d)$, is the closure of the complete $(d+1)$-ary tree of depth $h+1$. By \cref{StandardDefect}, for every graph $H$, the defective chromatic number of $H$-minor-free graphs satisfies
\begin{equation}
\label{DefectiveTreeDepthLowerBound}
\dchi(\MM_H) \geq \ctd(H)-1,
\end{equation}
as observed by \citet{OOW16}, who conjectured that equality holds.

\begin{conj}[\citep{OOW16}] 
\label{DefectiveTreeDepthConjecture}
For every graph $H$, 
$$\dchi(\MM_H) = \ctd(H)-1.$$
\end{conj}

Since $K_{s,t}$ has connected tree-depth $\min\{s,t\}+1$, by \cref{OOW},  
\cref{DefectiveTreeDepthConjecture} is true if $H=K_{s,t}$, as proved by \citet{OOW16}, who also proved \cref{DefectiveTreeDepthConjecture} if $\ctd(H)=3$. \citet{NSSW} provided further evidence for the conjecture by showing that $\cchi(\MM_H)$ is bounded from above by some function of $\ctd(H)$. This implies that 
both $\dchi(\MM_H)$ and $\cchi(\MM_H)$ are tied to $\ctd(H)$. 

\begin{thm}[\citep{NSSW}]
\label{ClusteredTreeDepth}
For every graph $H$,  
$$\cchi(\MM_H) \leq 2^{\ctd(H)+1}-4.$$
\end{thm}

While \citet{OOW16} conjectured that $\dchi(\MM_H)= \ctd(H)-1$,  the following lower bound by \citet{NSSW}  shows that $\cchi(\MM_H)$ might be larger, thus providing some distinction between defective and clustered colourings. 

\begin{thm}[\citep{NSSW}]
\label{NewLowerBound}
For each $k\geq 2$, if $H_k$ is the the standard example $S(k-1,2)$ then 
$H_k$ has connected tree-depth $k$ and $$\cchi(\MM_{H_k}) \geq 2k-2.$$
\end{thm}

\begin{proof}
Fix an integer $c$. We now recursively define graphs $G_k$ (depending on $c$), and show by induction on $k$ that 
$G_k$ has no $(2k-3)$-colouring with clustering $c$, and $H_k$ is not a minor of $G_k$. 

For the base case $k=2$, let $G_2$ be the path on $c+1$ vertices. Then $G_2$ has no $H_2=S(1,2)=K_{1,3}$ minor, and $G_2$ has no 1-colouring with clustering $c$. 

Assume $G_{k-1}$ is defined for some $k\geq 3$, that $G_{k-1}$ has no $(2k-5)$-colouring with clustering $c$, and $H_{k-1}$ is not a minor of $G_{k-1}$. As illustrated in \cref{NewConstruction}, let $G_k$ be obtained from a path $(v_1,\dots,v_{c+1})$ as follows:  for $i\in\{1,\dots,c\}$ add $2c-1$ pairwise disjoint copies of $G_{k-1}$ complete to $\{v_i,v_{i+1}\}$. 

\begin{figure}[b]
\centering
\includegraphics[width=\textwidth]{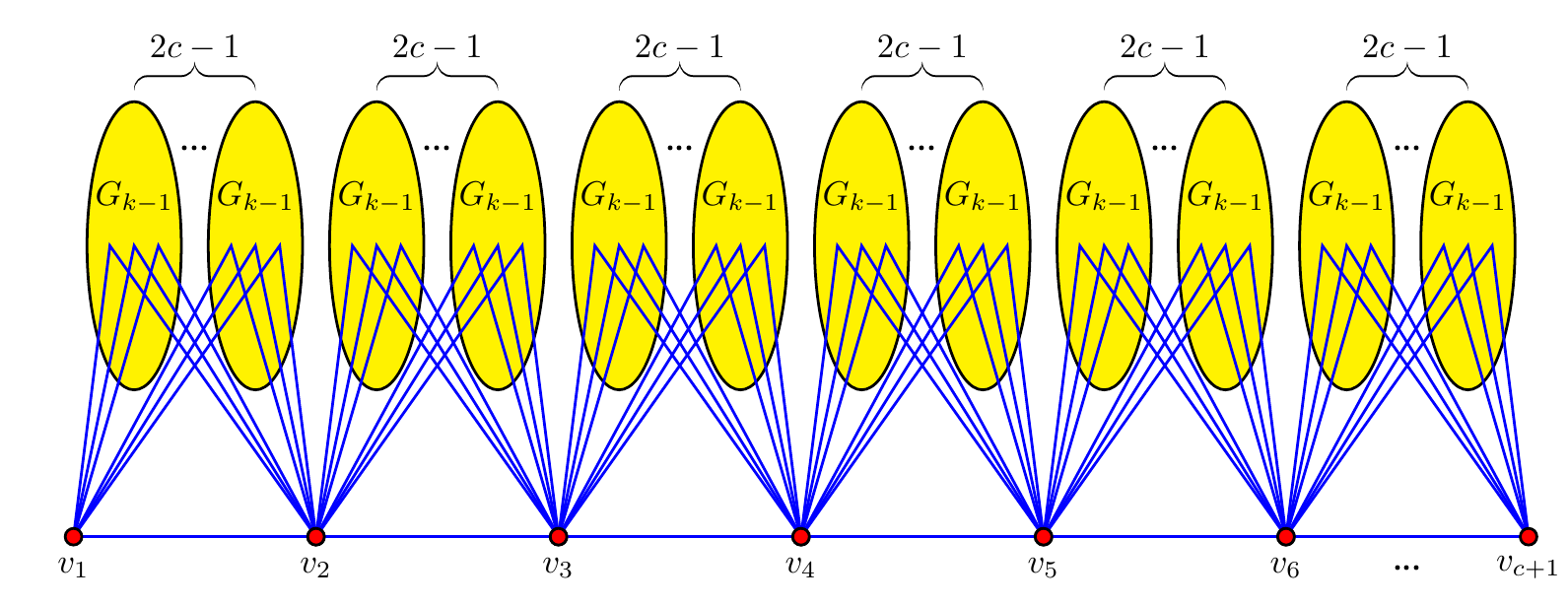}

\vspace*{-3ex}
\caption{\label{NewConstruction} Construction of $G_k$ in \cref{NewLowerBound}.}
\end{figure}

Suppose that $G_k$ has a $(2k-3)$-colouring with clustering $c$. 
Then $v_i$ and $v_{i+1}$ receive distinct colours for some $i\in\{1,\dots,c\}$. 
Consider the $2c-1$ copies of $G_{k-1}$ complete to $\{v_i,v_{i+1}\}$.  
At most $c-1$ such copies contain a vertex assigned the same colour as $v_i$, 
and at most $c-1$ such copies contain a vertex assigned the same colour as $v_{i+1}$. 
Thus some copy avoids both colours. 
Hence $G_{k-1}$ is $(2k-5)$-coloured with clustering $c$, which is a contradiction. 
Therefore $G_k$ has no $(2k-3)$-colouring with clustering $c$. 

It remains to show that $H_k$ is not a minor of $G_k$. 
Suppose that $G_k$ contains a model $\{J_x : x \in V(H_k)\}$ of $H_k$. 
Let $r$ be the root vertex in $H_k$. 
Choose the $H_k$-model to minimise $\sum_{x\in V(H)} |V(J_x)|$. 
Thus $J_r$ is a connected subgraph of $(v_1,\dots,v_{c+1})$. 
Say $J_r=(v_i,\dots,v_j)$. Note that $H_k-r$ consists of three pairwise disjoint copies of $H_{k-1}$. 
The model $X$ of one such copy avoids $v_{i-1}$ and $v_{j+1}$ (if these vertices are defined). 
Since $H_{k-1}$ is connected,  $X$ is contained in a component of $G_k-\{v_{i-1},\dots,v_{j+1}\}$ 
and is adjacent to $(v_i,\dots,v_j)$. 
Each such component is a copy of $G_{k-1}$. 
Thus $H_{k-1}$ is a minor of $G_{k-1}$, which is a contradiction.
Thus $H_{k-1}$ is not a minor of $G_k$. 
\end{proof}

\citet{NSSW} conjectured an analogous upper bound:

\begin{conj}[\citep{NSSW}]
\label{tdConjecture}
For every graph $H$,  $$\cchi(\MM_H)\leq 2\ctd(H)-2.$$ 
\end{conj}

\citet{NSSW} observed that \cref{tdConjecture} holds for every graph $H$ with $\ctd(H)=3$. In this case, as mentioned above, \citet{OOW16} proved that $\dchi(\MM_H) = 2$. Since $H$ is planar, by \cref{Defective3Clustered}, $\cchi(\MM_H)\leq 2\dchi(\MM_H) = 4 = 2 \ctd(H)-2$, as claimed.

In the remainder of this section we prove \cref{ClusteredTreeDepth}. The proof depends on the following Erd\H{o}s-P\'osa Theorem by \citet{RS-V}. For a graph $H$ and integer $p\geq 1$, let $p\,H$ be the disjoint union of $p$ copies of $H$. 

\newcommand{\blah}{\protect ; see~\citep[Lemma~3.10]{RT16}}

\begin{thm}[\citep{RS-V}\blah]
\label{ErdosPosa}
For every graph $H$ with $c$ connected components and for all integers $p,w\geq 1$, 
for every graph $G$ with treewidth at most $w$ and with no $p\, H$ minor,  
there is a set $X\subseteq V(G)$ of size at most $(p-1)(wc-1)$
such that $G - X$ has no $H$ minor. 
\end{thm}

%
%

The next lemma is the heart of the proof of \cref{ClusteredTreeDepth}.

\begin{lem}[\citep{NSSW}]
\label{Heart}
For all integers $h,k,w\geq 1$, every $S(h-1,k-1)$-minor-free graph $G$ of treewidth at most $w$ is $(2^h-2)$-colourable with clustering  $(k-1)(w-1)$. 
\end{lem} 

\begin{proof}
We proceed by induction on $h\geq 1$, with $w$ and $k$ fixed. The case $h=1$ is trivial since $S(0,k-1)$ is the 1-vertex graph. Now assume that $h\geq 2$, the claim holds for $h-1$, and $G$ is a $S(h-1,k-1)$-minor-free graph with treewidth at most $w$. Let $V_0,V_1,\dots$ be the BFS layering of $G$ starting at some vertex $r$. 

Fix $i\geq 1$. Then $G[V_i]$ contains no $k\,S(h-2,k-1)$ as a minor, as otherwise contracting $V_0\cup\dots\cup V_{i-1}$ to a single vertex gives a $S(h-1,k-1)$ minor (since every vertex in $V_i$ has a neighbour in $V_{i-1}$). Since $G$ has treewidth at most $w$, so does $G[V_i]$. By \cref{ErdosPosa} with $H=S(h-2,k-1)$ and $c=1$, there is a set $X_i\subseteq V_i$ of size at most  $(k-1)(w-1)$  such that $G[V_i\setminus X_i]$ has no $S(h-2,k-1)$ minor. By induction, $G[V_i\setminus X_i]$ is $(2^{h-1}-2)$-colourable with clustering  $(k-1)(w-1)$. 
Use one new colour for $X_i$. Thus $G[V_i]$ is $(2^{h-1}-1)$-colourable with clustering  $(k-1)(w-1)$. 
Use disjoint sets of colours for even and odd $i$, and colour $r$ by one of the colours used for even $i$. No edge joins $V_i$ with $V_j$ for $j\geq i+2$. Now $G$ is $(2^{h}-2)$-coloured with clustering  $(k-1)(w-1)$. 
\end{proof}

To drop the assumption of bounded treewidth, we use the following result of \citet*{DDOSRSV04}.

\begin{thm}[\citep{DDOSRSV04}]
\label{DeVos}
For every graph $H$ there is an integer $w$ such that for every graph $G$ with no $H$-minor, there is a partition $V_1,V_2$ of $V(G)$ such that $G[V_i]$ has treewidth at most $w$, for $i\in\{1,2\}$. 
\end{thm}

\begin{proof}[Proof of \cref{ClusteredTreeDepth}]
Let $k:=|V(H)|$ and $h:=\ctd(H)$.  By definition, $H$ is a subgraph of $S(h-1,k-1)$.  Thus every $H$-minor-free graph $G$ contains no $C(h,k)$-minor.  \cref{Heart,DeVos} implies that $G$ is $(2^{h+1}-4)$-colourable with clustering at most some function $g(h,k)$. The claim follows. 
\end{proof}

Note that \citet{NSSW} proved a weaker bound, roughly $\cchi(\MM_H)\leq 4^{\ctd(H)}$, with a self-contained proof avoiding \cref{DeVos} and thus avoiding the graph minor structure theorem. See~\citep{WW09,Woodall04} for more about defective choosability in minor-closed classes.

\subsection{Conjectures}

We now present a conjecture of \citet{NSSW} about the clustered chromatic number of an arbitrary minor-closed class of graphs.  Consider the following recursively defined class of graphs, illustrated in \cref{Xkc}. Let $\XX_{1,c}:=\{P_{c+1},K_{1,c}\}$. For $k\geq 2$, let $\XX_{k,c}$ be the set of graphs obtained by the following three operations. For the first two operations, consider an arbitrary graph $G\in \XX_{k-1,c}$. 
\begin{itemize}[itemsep=0ex,,topsep=0ex]
\item Let $G'$ be the graph obtained from $c$ disjoint copies of $G$ by adding one dominant vertex. Then $G'$ is in $\XX_{k,c}$. 
\item Let $G^+$ be the graph obtained from $G$ as follows: for each $k$-clique $D$ in $G$, add a stable set of $k(c-1)+1$ vertices complete to $D$. Then $G^+$ is in $\XX_{k,c}$. 
\item If $k\geq 3$ and $G\in \XX_{k-2,c}$, then let $G^{++}$ be the graph obtained from $G$ as follows: 
for each $(k-1)$-clique $D$ in $G$, add a path of $(c^2-1)(k-1)+(c+1)$ vertices complete to $D$. Then $G^{++}$ is in $\XX_{k,c}$. 
\end{itemize}

\begin{figure}
\centering
\framebox{\includegraphics{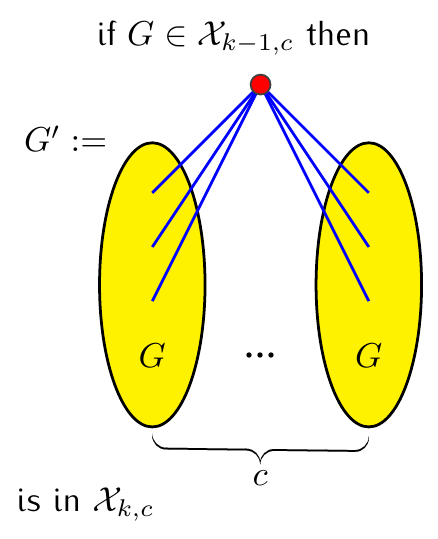}}
\quad
\framebox{\includegraphics{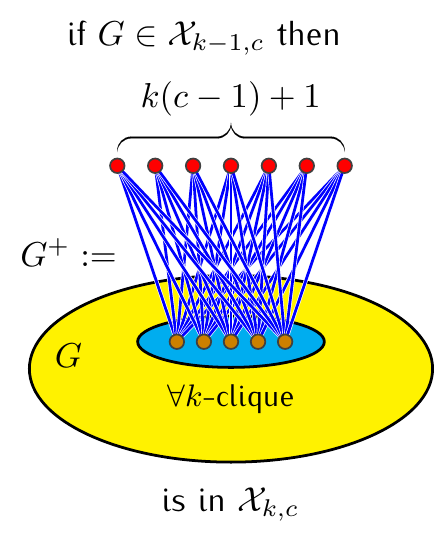}}
\quad
\framebox{\includegraphics{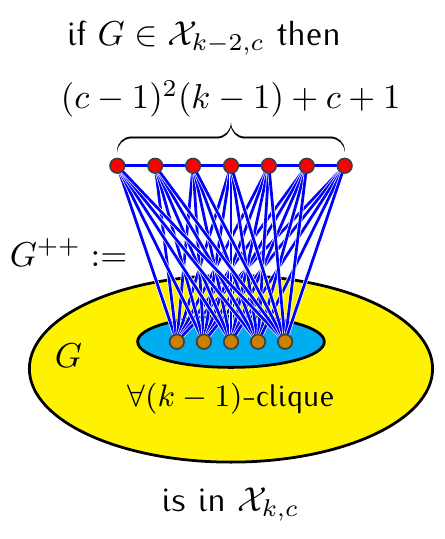}}
\caption{\label{Xkc} Construction of $\XX_{k,c}$.}
\end{figure}

A vertex-coloured graph is \emph{rainbow} if every vertex receives a distinct colour. 

\begin{lem}[\citep{NSSW}]
\label{Rainbow}
For every $c\geq 1$ and $k\geq 2$, for every graph $G\in\XX_{k,c}$, 
every colouring of $G$ with clustering $c$ contains a rainbow $K_{k+1}$.
In particular, no graph in $\XX_{k,c}$ is $k$-colourable with clustering $c$. 
\end{lem}

\begin{proof}
We proceed by induction on $k\geq 1$. In the case $k=1$, every colouring of $P_{c+1}$ or $K_{1,c}$ with clustering $c$ contains a bichromatic edge, and we are done. Now assume the claim for $k-1$ and for $k-2$ (if $k\geq 3$). 

Let $G\in \XX_{k-1,c}$. Consider a colouring of $G'$ with clustering $c$. Say the dominant vertex $v$ is blue. At most $c-1$ copies of $G$ contain a blue vertex. Thus, some copy of $G$ has no blue vertex. By induction, this copy of $G$ contains a rainbow $K_k$. With $v$ we obtain a rainbow $K_{k+1}$. 

Now consider a colouring of $G^+$ with clustering $c$. By induction, the copy of $G$ in $G^+$ contains a clique $\{w_1,\dots,w_k\}$ receiving distinct colours. Let $S$ be the set of $k(c-1)+1$ vertices adjacent to $w_1,\dots,w_k$ in $G^+$. At most $c-1$ vertices in $S$ receive the same colour as $w_i$. Thus some vertex in $S$ receives a colour distinct from the colours assigned to $w_1,\dots,w_k$. Hence $G^+$ contains a rainbow $K_{k+1}$. 


Now suppose $k\geq 3$ and  $G\in \XX_{k-2,c}$. Consider a colouring of $G^{++}$ with clustering $c$. 
By induction, the copy of $G$ in $G^{++}$ contains a clique $\{w_1,\dots,w_{k-1}\}$ receiving distinct colours. 
Let $P$ be the path of $(c^2-1)(k-1)+(c+1)$ vertices in $G^{++}$ complete to $\{w_1,\dots,w_{k-1}\}$. 
Let $X_i$ be the set of vertices in $P$ assigned the same colour as $w_i$, and let $X:=\bigcup_iX_i$. 
Thus $|X_i|\leq c-1$ and $|X|\leq (c-1)(k-1)$. 
Hence $P-X$ has at most $(c-1)(k-1)+1$ components, and $|V(P-X)| \geq (c^2-1)(k-1)+(c+1)- (c-1)(k-1) = c\big( (c-1)(k-1)+1 \big)+1$. 
Some component of $P-X$ has at least $c+1$ vertices, and therefore contains a bichromatic edge $xy$.
Then $\{w_1,\dots,w_{k-1}\}\cup\{x,y\}$ induces a rainbow $K_{k+1}$ in $G^{++}$. 
\end{proof}

\citet{NSSW} conjectured that every minor-closed class that excludes every graph in $\XX_{k,c}$ for some $c$ is $k$-colourable with bounded clustering. More precisely:

\begin{conj}[\citep{NSSW}]
\label{MinorConjecture}
For every minor-closed class $\MM$ of graphs, 
$\cchi(\MM)$ equals the minimum integer $k$ such that $\MM\cap \XX_{k,c}=\emptyset$ for some integer $c$. 
\end{conj}

Note that the lower bound in \cref{MinorConjecture} follows from \cref{Rainbow}.  \cref{MinorConjecture} is trivial when $k=1$, and  \citet{NSSW} proved it when $k=2$. It is easily seen that \cref{MinorConjecture} implies \cref{tdConjecture}; see~\cite{NSSW}. 

\label{ExcludeCompleteBipartiteMinor}

Now consider the class of graphs excluding the complete bipartite graph $K_{s,t}$ as a minor, where $s\leq t$.\cref{OOW,DefectiveTreeDepthLowerBound} imply that $$\dchi(\MM_{K_{s,t}})=\ldchi(\MM_{K_{s,t}})=s.$$
For clustered colouring, \cref{Defective3Clustered} implies $\cchi(\MM_{K_{s,t}})\leq 3s$. This bound was improved by \citet{DN17} who proved that $\cchi(\MM_{K_{s,t}}) \leq 2s+2$, which is the best known upper bound. \VdHW\ proved the lower bound, $\cchi(\MM_{K_{s,t}}) \geq s+1$ for $t\geq\max\{s,3\}$. Their construction is a special case of the construction shown in \cref{Xkc}. \cref{MinorConjecture}  says that $\cchi(\MM_{K_{s,t}})=s+1$. 

\subsection{Circumference}
\label{Circumference}

The \emph{circumference} of a graph $G$ is the length of the longest cycle if $G$ contains a cycle,  and is 2 if $G$ is a forest. This section studies clustered colourings of graphs of given circumference. Let $\CC_k$ be the class of graphs with circumference at most $k$. A graph has circumference at most $k$ if and only if it contains no $C_{k+1}$ minor, where $C_{k+1}$ is the cycle on $k+1$ vertices. Thus $\CC_k$ is a minor-closed class. 

\begin{lem}
\label{StandardCircumference}
For all $k,d\geq 1$, the standard example $S(k,d)$ contains no path on $2^{k+1}$ vertices and no cycle of length at least $2^k+1$.
\end{lem}

\begin{proof}
We proceed by induction on $k\geq 1$ with $d$ fixed. In the base case, $S(1,d)=K_{1,d+1}$, which contains no 4-vertex path and no cycle. Now assume the result for $S(k-1,d)$. Let $v$ be the root vertex of $S(k,d)$. 

Suppose that $S(k,d)$ contains a cycle  $C$ of length at least $2^k+1$. Since $v$ is a cut vertex, $C$ is contained in one copy of $S(k-1,d)$ plus $v$. Thus $S(k-1,d)$ contains a path on $2^k$ vertices, which is a contradiction. Thus $S(k,d)$ has no cycle of length at least $2^k+1$.

If $S(k,d)$ contains a path $P$ on  $2^{k+1}$ vertices, then $P-v$ contains a path component with least $\ceil{ \frac12 (2^{k+1}-1) } = 2^k$ vertices that is contained in a copy of $S(k-1,d)$, which is a contradiction. Hence  $S(k,d)$ contains no path of order $2^{k+1}$.
\end{proof}

\begin{figure}
\centering
\includegraphics{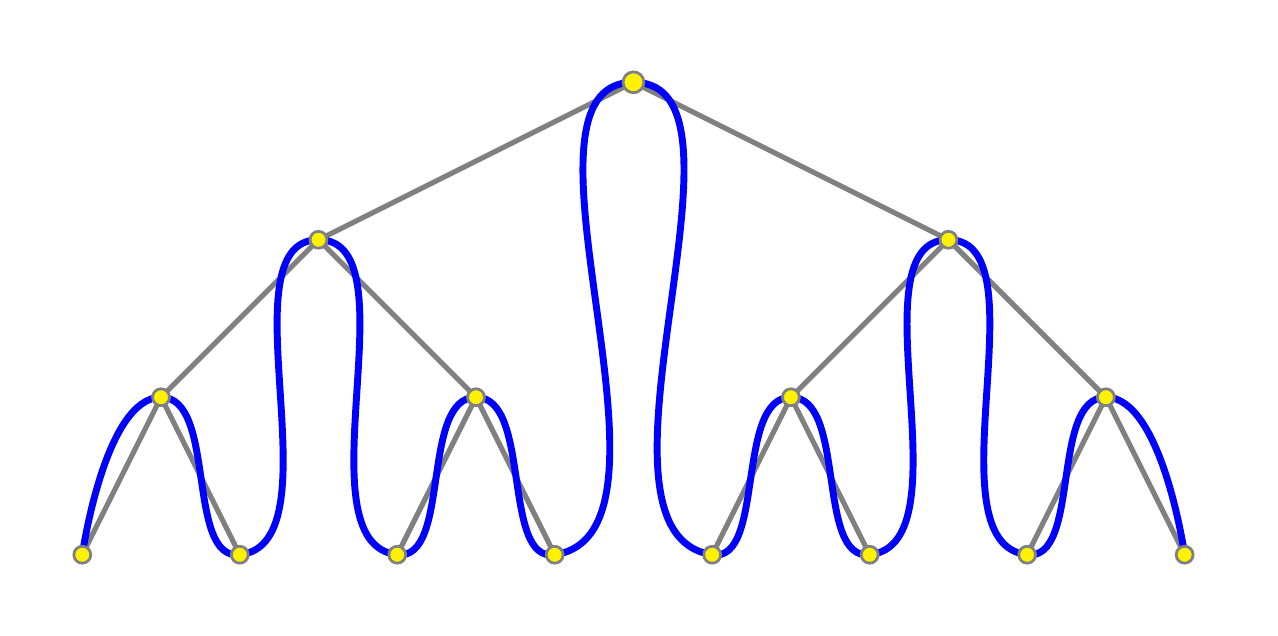}
\caption{\label{TreedepthPath}
The path on $n$ vertices is a subgraph of $S(k,1)$, where $k=\ceil{ \log _{2}(n+1)}$.}
\end{figure}

It follows from \cref{StandardCircumference} that 
$$\ctd(C_{k+1})=1+\ceil{ \log _{2}(k+1)}=2+\floor{\log_2 k}.$$
By \cref{DefectiveTreeDepthLowerBound},  
$$\cchi(\CC_k) \geq \dchi(\CC_k) \geq \ctd(C_{k+1})-1 = 1+\floor{\log_2 k}.$$
\citet{MRW17} proved an upper bound within a factor of 3 of this lower bound. 

\begin{thm}[\citep{MRW17}]
\label{ClusteredCircumference}
For every integer $k \geq 2$, every graph $G$ with circumference at most $k$ is $(3 \log_{2} k)$-colourable with clustering $k$. Thus
$$\dchi(\CC_k) \leq \cchi(\CC_k) \leq 3\log_2 k.$$
\end{thm}

This result is implied by the following lemma with $C=\emptyset$. 

\begin{lem}[\citep{MRW17}]
For every integer $k \geq 2$, for every graph $G$ with circumference at most $k$ and for every pre-coloured clique $C$ of size at most $2$ in $G$, there is a $\floor{3\log_2 k}$-colouring of $G$ with clustering $k$, such that every monochromatic component that intersects $C$ is contained in $C$.
\end{lem}

\begin{proof}
We proceed by induction on $k+|V(G)|$. The result is trivial if $|V(G)|\leq 2$. Now assume $|V(G)|\geq 3$. 

First suppose that $k=2$. Then $G$ is a forest, which is properly 2-colourable. If $|C|\le 1$ or $|C|=2$ and two colours are used on $C$, we obtain the desired colouring (with $2< \floor{3\log_2 k}$ colours). Otherwise, $|C|=2$ with the same colour on the vertices in $C$. Contract the edge $C$ and 2-colour the resulting forest by induction, to obtain the desired colouring of $G$. Now assume that $k \geq 3$.

Suppose that $G$ is not $3$-connected. Then $G$ has a minimal separation $(G_1,G_2)$ with $S := V(G_1 \cap G_2)$ of size at most 2. If $|S|=2$, then add the edge on $S$ if the edge is not already present. Consider both $G_1$ and $G_2$ to contain this edge.
Observe that since the separation is minimal, there is a path in each $G_j$ ($j=1,2$) between the two vertices of $S$. Therefore, adding the edge does not increase the circumference of $G$. Also note that any valid colouring of the augmented graph will be valid for the original graph. 
Since $C$ is a clique, we may assume that $C \subseteq V(G_1)$. By induction, there is a $\floor{3\log_2 k}$-colouring of $G_1$, with $C$ precoloured, such that every monochromatic component of $G_1$ has order at most $k$
and every monochromatic component of $G_1$ that intersects $C$ is contained in $C$. This colours $S$.
By induction, there is a $\floor{3\log_2 k}$-colouring of $G_2$, with $S$ precoloured, such that every monochromatic component of $G_2$  has order at most $k$
and every monochromatic component of $G_2$  that intersects $S$ is contained in $S$. By combining the two colourings, every monochromatic component of $G$ has order at most $k$
and every monochromatic component of $G$ that intersects $C$ is contained in $C$, as required.
Now assume that $G$ is 3-connected.

Every 3-connected graph contains a cycle of length at least 4. Thus $k\geq 4$. 

If $G$ contains no cycle of length $k$, then apply the induction hypothesis for $k-1$; thus we may assume that $G$ contains a cycle $Q$ of length $k$. 
Let $\mathcal{A}$ be the set of cycles in $G$ of length  at least $\ceil{\tfrac12 (k-5)}$. 
Suppose that a cycle $A \in \mathcal{A}$ is disjoint from $Q$. Since $G$ is 3-connected, there are three disjoint paths between $A$ and $Q$. It follows that $G$ contains three cycles with total length at least 
$2(|A|+|Q|+3) > 3k$. Thus $G$ contains a cycle of length greater than $k$, which is a contradiction.
Hence, every cycle in $\mathcal{A}$ intersects $Q$.

Let $S:=V(Q)\cup C$. As shown above, $G':=G-S$ contains no cycle of length at least $\ceil{\frac{1}{2}(k-5)}$. 
Then $G'$ has circumference at most $\ceil{\frac{1}{2}(k-7)}$, which is at most $\floor{\tfrac12 k}$, which is at least 2. By induction (with no precoloured vertices), there is a $\floor{3\log_{2} \floor{\frac12 k}}$-colouring of $G'$ such that every monochromatic component of $G'$ has order at most $\floor{\frac12 k}$. Use a set of colours for $G'$ disjoint from the (at most two) preassigned colours for $C$. Use one new colour for $S\setminus C$, which has size at most $k$. In total, there are at most $\floor{3\log_2 \floor{\frac12 k} }+3 \leq  \floor{3\log_2 k}$ colours. Every monochromatic component of $G$ has order at most $k$, 
and every monochromatic component of $G$ that intersects $C$ is contained in $C$.
\end{proof}

Similar results are obtained for graph classes excluding a fixed path. Let $P_k$ be the path on $k$ vertices. 
It follows from \cref{StandardCircumference} that 
$$\ctd(P_k)=\td(P_k)=\ceil{ \log _{2}(k+1)}.$$ 
Let $\HH_k$ be the class of graphs containing no path of order $k+1$ (or equivalently, with no $P_{k+1}$ minor). Thus \cref{DefectiveTreeDepthConjecture}, in the case of excluded paths, asserts that 
$$\dchi(\HH_{k})=\td(P_{k+1})-1=\ceil{ \log _{2}(k+2)}-1.$$ 
Every graph with no $P_{k+1}$-minor has circumference at most $k$. 
Thus \cref{ClusteredCircumference} implies the following upper bound that is  
which is within a factor of 3 of \cref{DefectiveTreeDepthConjecture} for excluded paths. 

\begin{thm}[\citep{MRW17}]
\label{ClusteredLongestPath}
For every integer $k \geq 2$, every graph $G$ with no path of order $k+1$ is $(3 \log_{2} k)$-colourable with clustering $k$. Thus
$$\dchi(\HH_{k}) \leq \cchi(\HH_{k})\leq \floor{3 \log_2 k},$$ 
\end{thm}


To conclude this section, we show that \cref{DefectiveChoosability} determines the defective choosability of $\CC_k$ and $\PP_k$. 

First consider $\CC_k$. Say $s\leq t$. Then  the circumference of $K_{s,t}$ equals $2s$. If $s\geq \ceil{\frac{k+1}{2}}$ then $K_{s,t}$ contains a $(k+1)$-cycle and is not in $\CC_k$. On the other hand, if $s\leq\ceil{\frac{k+1}{2}}-1$ then $2s \leq k$ and $K_{s,t}\in \CC_k$. Thus $\ldchi(\CC_k)=\ceil{\frac{k+1}{2}}$. 

Now consider $\PP_k$. Say $t>s$. Then $K_{s,t}$ contains a path order $2s+1$, and contains no path of order $2s+2$. Thus $K_{s,t}\in \HH_k$ if and only if $2s+1 \leq k$. 
That is, $K_{s,t}\not\in \HH_k$ if and only if $s>\frac{k-1}{2}$. 
By \cref{DefectiveChoosability},   $\ldchi(\HH_k)=\floor{\frac{k+1}{2}}$.

\section{Thickness}
\label{Thickness}

The \emph{thickness} of a graph $G$ is the minimum integer $k$ such that $G$ is the union of $k$ planar subgraphs; see~\citep{MOS98} for a survey. Let $\TT_k$ be the class of graphs with thickness $k$. Graphs with thickness $k$ have maximum average degree less than $6k$. Thus $\bigchi(\TT_k)\leq 6k$. For $k=2$, which corresponds to the so-called earth-moon problem, it is known that $\bigchi(\TT_2)\in\{9,10,11,12\}$. For $k\geq 3$, complete graphs provide a lower bound of $6k-2$, implying $\bigchi(\TT_k)\in\{6k-2,6k-1,6k\}$. It is an open problem to improve these bounds; see~\citep{Hut93}.

\subsection{Defective Colouring} 

This section studies defective colourings of graphs with given thickness. \citet{Yancey12} first proposed studying this topic. The results in this section are due to \citet{OOW16}. Since the maximum average degree is less than $6k$, \cref{DefectiveMaximumAverageDegree} implies that such graphs are $(3k+1)$-choosable with defect $O(k^2)$, but gives no result with at most $3k$ colours. 

\begin{lem}[\citep{OOW16}] 
\label{StandardThickness}
The standard example $S(2k,d)$ has thickness at most $k$.
\end{lem}

\begin{proof}
We proceed by induction on $k\geq 1$. 
In the base case, $S(2,d)$ is planar, and thus has thickness 1. 
Let $G:=S(2k,d)$ for some $k\geq 2$. 
Let $r$ be the vertex of $G$ such that $G-r$ is the disjoint union of $d+1$ copies of $S(2k-1,d)$. 
For $i\in[d+1]$, let $v_i$ be the vertex of the $i$-th component $C_i$ of $G-r$ such that $C_i-v_i$ is the disjoint union of $d+1$ copies of $S(2k-2,d)$. Let $H:=G-\{r,v_1,v_2,\ldots,v_{d+1}\}$.
Observe that each component of $H$ is isomorphic to $S(2k-2,d)$ and by induction, $H$ has thickness at most $k-1$. Since $G-E(H)$ consists of $d+1$ copies of $K_{2,d'}$ pasted on $r$ for some $d'$, $G-E(H)$ is planar and thus has thickness 1.  
Hence $G$ has thickness at most $k$. 
\end{proof}

\cref{StandardThickness,StandardDefect} imply that $\dchi(\TT_k)\geq 2k+1$. 
\citet{OOW16} proved that equality holds.  In fact, the proof works in the following more general setting, implicitly introduced by \citet{JR00}. For an integer $g\geq 0$, the \emph{$g$-thickness} of a graph $G$ is the minimum integer $k$ such that $G$ is the union of $k$ subgraphs each with Euler genus at most $g$.  Let $\TT^g_k$ be the class of graphs with $g$-thickness $k$. As an aside, note that the $g$-thickness of complete graphs is closely related to bi-embeddings of graphs \citep{AC74,Anderson82,Anderson79,Cabaniss90}. 

\begin{thm}[\citep{OOW16}] 
\label{ColourGenusThickness}
For integers $g\geq 0$ and $k\geq 1$, 
$$\dchi(\TT_k^g)=\ldchi(\TT_k^g)=2k+1.$$
In particular, every graph with $g$-thickness at most $k$ is $(2k+1)$-choosable with defect $2kg+8k^2+2k$. 
\end{thm}

The lower bound in \cref{ColourGenusThickness} follows from \cref{StandardThickness}. 
The upper bound  follows from \cref{light} and the next lemma. 

\begin{lem}[\citep{OOW16}] 
\label{LightEdgeGenusThickness}
For integers $g\geq 0$ and $k\geq 1$, every graph with minimum degree at least $2k+1$ and $g$-thickness at most $k$ has an $(\ell-1)$-light edge, where $\ell:=2kg + 8k^2 + 2k+1$.
\end{lem}

\begin{proof}
We claim that \cref{LightEdgeGen} below with $\delta=2k+1$ implies the result. 
\cref{LightEdgeConditionA,LightEdgeConditionB} are immediately satisfied. 
Let $\beta:= \big( (4k-1)(2k+1)+2k(g-1)\big)$ and $\gamma:= 4k(2k+1)(g-1)$. 
\cref{LightEdgeCondition} requires that $\ell^2 -  \beta \ell - \gamma  >0$. 
The larger root of $\ell^2 -  \beta \ell - \gamma$ is $\frac12 ( \beta+\sqrt{\beta^2+4\gamma} )$, 
which is at most $\beta+\frac{\gamma}{\beta}$ since $\beta+\frac{2\alpha}{\beta}>0$. 
Elementary manipulations show that $\ell>\beta+\frac{\gamma}{\beta}$. 
Thus  \eqref{LightEdgeCondition} is satisfied. 
\end{proof}


%
%
%
%
%
%
%
%
%


\begin{lem} 
\label{LightEdgeGen}
Let $G$ be a graph with $n$ vertices, $g$-thickness at most $k$, and minimum degree at least $\delta$, where
\begin{align}
& 6k \geq \delta \geq 2k+1, \label{LightEdgeConditionA}\\
& (\delta-2k)\ell  >4k\delta, \text{ and} \label{LightEdgeConditionB}\\
& (\delta-2k)\ell^2 -  \big( (4k-1)\delta+2k(g-1) \big)  \ell - 4k(g-1)\delta  >0. 
\label{LightEdgeCondition}
\end{align}
Then $G$ has an $(\ell-1)$-light edge.
\end{lem}

\begin{proof}
By Euler's Formula, $G$ has at most $3k(n+g-2)$ edges, 
and every spanning bipartite subgraph has at most $2k(n+g-2)$ edges. 
Let $X$ be the set of vertices with degree at most $\ell-1$. 
Since vertices in $X$ have degree at least $\delta$ and vertices not in $X$ have degree at least $\ell$,
$$\delta |X| + (n-|X|)\ell \leq \sum_{v\in V(G)} \deg(v)  = 2|E(G)| \leq 6k(n+g-2) .$$ 
Thus
\begin{equation*}
(\ell-6k)n -6k(g-2) \leq (\ell-\delta)|X|.
\end{equation*}
Suppose on the contrary that $X$ is a stable set in $G$. Let $G'$ be the spanning bipartite subgraph of $G$ consisting of all edges between $X$ and $V(G)\setminus X$. Since each of the at least $\delta$ edges incident with each vertex in $X$ are in $G'$, 
$$\delta  |X| \leq |E(G')| \leq 2k(n+g-2).$$ 
Since $\ell>\frac{4k}{\delta-2k}\delta>\delta$ (hence  $\ell-\delta > 0$)  and $\delta\geq 0$, 
\begin{align*}
& & \delta (\ell-6k)n-6k(g-2)\delta  & \leq \delta (\ell-\delta)|X| \leq (\ell-\delta)(2k)(n+g-2)\\
& \Longrightarrow & \big( \delta (\ell-6k) - 2k(\ell-\delta)\big) n  & \leq (\ell-\delta)2k(g-2) + 6k(g-2)\delta\\
& \Longrightarrow & \big( (\delta-2k)\ell -4k\delta\big) n  & \leq 2k(g-2)\ell +4k(g-2)\delta.
\end{align*}
If $n\leq \ell$ then every edge is $(\ell-1)$-light. Now assume that $n\geq \ell+1$. 
Since $(\delta-2k)\ell -4k\delta>0$,
$$
 \big( (\delta-2k)\ell -4k\delta\big) (\ell+1)   \leq 2k(g-2)\ell +4k(g-2)\delta. 
$$
Thus
$$(\delta-2k)\ell^2 -\big( (4k-1)\delta + 2k(g-1)\big)  \ell   - 4k(g-1)\delta \leq 0  , $$
which is a contradiction. Thus $X$ is not a stable set. Hence $G$ contains an $(\ell-1)$-light edge. 
\end{proof}

%


\cref{LightEdgeGen} with $k=1$ and $\ell=2g+13$ implies that every graph $G$ with minimum degree at least $3$ and Euler genus $g$ has a $(2g+12)$-light edge. Note that this bound is within $+10$ of being tight since $K_{3,2g+2}$ has minimum degree 3, embeds in a surface of Euler genus $g$, and every edge has an endpoint of degree $2g+2$. More precise results, which are typically proved by discharging with respect to an embedding, are known~\citep{JT06,Ivanco92,Borodin-JRAM89}. \cref{ColourGenusThickness} then implies that every graph with Euler genus $g$ is $3$-choosable with defect $2g+10$. As mentioned in \cref{SurfacesDefective}, this result with a better degree bound was proved by \citet{Woodall11}; also see~\citep{CE16}. The utility of \cref{LightEdgeGen} is that it applies for $k>1$.

The case $g=0$ and $k=2$ relates to the famous earth--moon problem~\citep{ABG11,GS09,Hut93,Ringel59,JR00}, which asks for the maximum chromatic number of graphs with thickness $2$. The answer is known to be in $\{9,10,11,12\}$.  Since the maximum average degree of every graph $G$ with thickness 2 is less than $12$, the result of \citet{HS06} mentioned in \cref{MaximumAverageDegree} implies that $G$ is $k$-choosable with defect $d$, for $(k,d)\in\{(7,18),(8,9),(9,5),(10,3),(11,2)\}$.  This result gives no bound with at most 6 colours. On the other hand, \cref{ColourGenusThickness} says that the class of graphs with thickness 2 has defective chromatic number and defective choice number equal to $5$. In particular, the method shows that every graph $G$ with thickness 2 is $k$-choosable with defect $d$, for $(k,d)\in\{(5,36),(6,19),(7,12),(8,9),(9,6),(10,4),(11,2)\}$. This 11-colouring result, which is also implied by the result of \citet{HS06}, is close to the conjecture that graphs with thickness 2 are 11-colourable. Improving these degree bounds provides an approach for attacking the earth--moon problem. 

\subsection{Clustered Colouring} 
\label{ClusteredThickness} 

Consider the clustered chromatic number of $\TT^g_k$. We have the following upper bound. 

\begin{prop}
\label{ThicknessClustering}
For all integers $g\geq 0$ and $k\geq 1$, every graph $G$ with $g$-thickness at most $k$ is $(6k+1)$-choosable with clustering $\max\{g,1\}$.
\end{prop}

\begin{proof}
We proceed by induction on $|V(G)|$.
 Let $L$ be a $(6k+1)$-list assignment for $G$. In the base case, the claim is trivial if $|V(G)|=0$. Now assume that $|V(G)|\geq 1$. 
 
 First suppose that $|V(G)|\leq 6kg$. Thus $g\geq 1$. Let $v$ be any vertex of $G$. By induction, $G-v$ is $L$-colourable with clustering $g$. Since $|L(v)|\geq 6g$ and $|V(G-v)|< 6kg$, some colour $c\in L(v)$ is assigned to at most $g-1$ vertices in $G-v$. Colour $v$ by $c$. Thus $G$ is $L$-coloured with clustering $g$.  
 
 Now assume that $|V(G)|>6kg$. Every graph with $g$-thickness at most $k$ and more than $6kg$ vertices has a vertex $v$ of degree at most $6k$, since $|E(G)|<3k(|V(G)|+g) \leq (3k+\frac12)|V(G)|$, implying $G$ has average degree less than $6k+1$. By induction, $G-v$ is $(6k+1)$-choosable with clustering $\max\{g,1\}$. Since $|L(v)|\geq 6k+1$ and $\deg(v)\leq 6k$, some colour $c\in L(v)$ is not assigned to any neighbour of $v$. Colour $v$ by $c$. Thus $v$ is in a singleton monochromatic component, and $G$ is $L$-coloured with clustering $\max\{g,1\}$. 
\end{proof}

Since $\overline{S}(3,c)$ is planar, 
an analogous proof to that of \cref{StandardThickness} shows that $\overline{S}(2k+1,c)$ has thickness at most $k$.
By \cref{StandardClustering,ThicknessClustering}, 
$$2k+2 \leq \cchi(\TT^g_k)\leq 6k+1.$$ 
Closing this gap is an interesting problem because the existing methods say nothing for graphs with given thickness. 
For example, as illustrated in \cref{ThicknessTwoExample}, the 1-subdivision of $K_n$ has thickness 2. Thus thickness 2 graphs have unbounded $\nabla$. Similarly, \cref{SeparatorIsland} is not applicable since graphs with thickness 2 do not have sublinear balanced separators. Indeed, \citet{DSW16} constructed `expander' graphs with thickness 2, bounded degree, and with no $o(n)$ balanced separators. Returning to the earth-moon problem, it is open whether thickness 2 graphs are 11-colourable with bounded clustering.

\begin{figure}[h]
\centering

\includegraphics{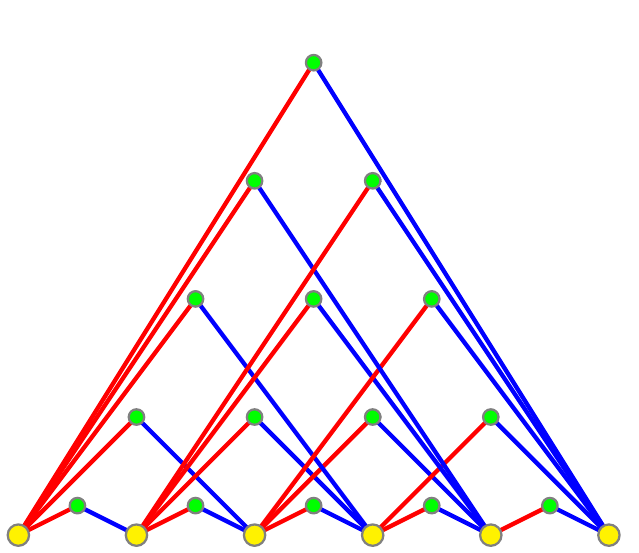}

\caption{The 1-subdivision of $K_n$ has thickness 2. 
\label{ThicknessTwoExample}}
\end{figure}

\section{General Setting}

Consider a graph parameter $f$. Several authors have studied colourings of graphs such that $f$ is bounded for each monochromatic subgraph, or equivalently each monochromatic subgraph satisfies a particular property; see~\citep{Frick93,BDJ02,Jones74,BJ90,WW94,DZ98} for example. For a graph class $\GG$, define $\chigen{f}(\GG)$ to be the minimum integer $k$ such that for some $c$, every graph $G\in \GG$ has a $k$-colouring such that $f(H)\leq c$ for each monochromatic subgraph $H$ of $G$. This definition incorporates the clustered and defective chromatic numbers. In particular, if $f(H)$ is the maximum number of vertices in a connected component of $H$, then $\chigen{f}(\GG)=\cchi(\GG)$, and  if $f(H)=\Delta(H)$ then $\chigen{f}(\GG)=\dchi(\GG)$. There are many choices for $f$ and $\GG$.  

First consider when $f$ is the treewidth. Recall that \citet{DDOSRSV04} proved that $\chigen{\tw}(\GG)\leq 2$ for every minor-closed class $\GG$. Bounded degree classes and $\chigen{\tw}(\DD_\Delta)$ look interesting. In particular, what is the maximum integer $\Delta$ such that every graph with maximum degree $\Delta$ is 2-colourable with bounded monochromatic treewidth? The answer is at least 5 since every graph with maximum degree 5 is 2-colourable with bounded clustering~\citep{HST03}. The answer is at most 25 since \citet{BDN17} proved that for every 2-colouring of the $n \times n\times n$ grid with diagonals (which has maximum degree 26), there is a monochromatic subgraph with unbounded treewidth (as $n\rightarrow\infty$). This upper bound is probably easily improved by eliminating some of the diagonals in the 3-dimensional grid. Can the lower bound be improved? In particular, does every graph with maximum degree 6 have a 2-colouring with bounded monochromatic treewidth?

Let $\eta(H)$ be the maximum integer $t$ such that $K_t$ is a minor of $H$, sometimes called the \emph{Hadwiger number} of $H$. For example, $\eta(H)\leq 1$ if and only if $H$ is edgeless, and  $\eta(H)\leq 2$ if and only if $H$ is a forest. \citet{DOSV-JCTB00} conjectured that for every graph $G$ and integer $k\in[1,\eta(G)]$, $G$ is $k$-colourable with $\eta(H)\leq \eta(G)-k+1$ for each monochromatic subgraph $H$ of $G$. The case $k=\eta(G)$ is Hadwiger's Conjecture. Alternately 2-colouring the layers in a BFS  layering proves this conjecture for $k=2$.  \citet{Gon11} proved it for $k=3$.  $\chigen{\eta}(\DD_\Delta)$ also looks interesting. 


\subsection*{Acknowledgements}

Many thanks to Jan van den Heuvel, Sergey Norin, Alex Scott, Paul Seymour, Patrice Ossona de Mendez, Bojan Mohar,  Sang-il Oum and Bruce Reed with whom I have collaborated on defective and clustered colourings.  This survey has also greatly benefited from insightful comments from Zden{\v{e}}k Dvo{\v{r}}{\'a}k, Louis Esperet, Jacob Fox, Fr{\'e}d{\'e}ric Havet, Gwena\"el Joret, Chun-Hung Liu and Tibor Szab\'o, for which I am extremely grateful. I am particularly thankful to Fr{\'e}d{\'e}ric Havet for allowing his unpublished proof of \cref{WeakDefectiveMaximumAverageDegree} to be included. Many thanks to the referees for detailed and insightful comments that have greatly improved the paper. 

  \let\oldthebibliography=\thebibliography
  \let\endoldthebibliography=\endthebibliography
  \renewenvironment{thebibliography}[1]{%
    \begin{oldthebibliography}{#1}%
      \setlength{\parskip}{0.22ex}%
      \setlength{\itemsep}{0.22ex}%
  }{\end{oldthebibliography}}


\def\soft#1{\leavevmode\setbox0=\hbox{h}\dimen7=\ht0\advance \dimen7
  by-1ex\relax\if t#1\relax\rlap{\raise.6\dimen7
  \hbox{\kern.3ex\char'47}}#1\relax\else\if T#1\relax
  \rlap{\raise.5\dimen7\hbox{\kern1.3ex\char'47}}#1\relax \else\if
  d#1\relax\rlap{\raise.5\dimen7\hbox{\kern.9ex \char'47}}#1\relax\else\if
  D#1\relax\rlap{\raise.5\dimen7 \hbox{\kern1.4ex\char'47}}#1\relax\else\if
  l#1\relax \rlap{\raise.5\dimen7\hbox{\kern.4ex\char'47}}#1\relax \else\if
  L#1\relax\rlap{\raise.5\dimen7\hbox{\kern.7ex
  \char'47}}#1\relax\else\message{accent \string\soft \space #1 not
  defined!}#1\relax\fi\fi\fi\fi\fi\fi}

\end{document}